\documentclass[12pt,reqno]{amsart}
\usepackage{amsmath,amssymb,extarrows}
\usepackage{url}
\usepackage{tikz,enumerate}
\usepackage{diagbox}
\usepackage{appendix}
\usepackage{epic}
\usepackage{makecell}

\usepackage{float} 

\usepackage{cite}

\usepackage{hyperref}
\usepackage{array}

\usepackage{booktabs}

\allowdisplaybreaks[4]

\newtheorem{theorem}{Theorem}
\newtheorem{corollary}[theorem]{Corollary}
\newtheorem{conj}[theorem]{Conjecture}
\newtheorem{lemma}[theorem]{Lemma}

\theoremstyle{definition}

\theoremstyle{remark}
\newtheorem{rem}{Remark}
\numberwithin{equation}{section}
\numberwithin{theorem}{section}
\numberwithin{defn}{section}

\DeclareMathOperator{\diag}{diag}
\DeclareMathOperator{\CT}{\mathrm{CT}}

\begin{document}
\title[Proofs of Mizuno's Conjectures on  Rank Three Nahm Sums]{Proofs of Mizuno's Conjectures on  Rank Three Nahm Sums of Index $(1,2,2)$}

\author{Boxue Wang and Liuquan Wang}
\address[B. Wang]{School of Mathematics and Statistics, Wuhan University, Wuhan 430072, Hubei, People's Republic of China}
\email{boxwang@whu.edu.cn}
\address[L. Wang]{School of Mathematics and Statistics, Wuhan University, Wuhan 430072, Hubei, People's Republic of China}
\email{wanglq@whu.edu.cn;mathlqwang@163.com}



\begin{abstract}
Mizuno provided 15 examples of generalized rank three Nahm sums with symmetrizer $\mathrm{diag}(1,2,2)$ which are conjecturally modular. Using the theory of Bailey pairs and some $q$-series techniques, we establish a number of triple sum Rogers--Ramanujan type identities. These identities confirm the modularity of all of Mizuno's examples except that two Nahm sums are sums of modular forms of weights $0$ and $1$. We also prove Mizuno's conjectural modular transformation formulas for two vector-valued functions consisting of Nahm sums with symmetrizers $\mathrm{diag}(1,1,2)$ and $\mathrm{diag}(1,2,2)$.
\end{abstract}

\maketitle

\tableofcontents

\section{Introduction}\label{sec-intro}
We begin by introducing some standard $q$-series notation:
\begin{align}
  & (a;q)_n:=\prod\limits_{k=0}^{n-1} (1-aq^k), \quad n\in \mathbb{Z}_{\geq 0},  \label{aq-defn} \\
   & (a;q)_\infty:=\prod\limits_{k=0}^{\infty} (1-aq^k), \quad |q|<1. \label{aq-infinite-defn}
\end{align}
For convenience, we also use compressed notation:
\begin{align}
    (a_1,a_2,\dots,a_m;q)_n:=(a_1;q)_n(a_2;q)_n\cdots (a_m;q)_n, \quad n \in \mathbb{Z}_{\geq 0}\cup \{\infty\}.
\end{align}

The famous Rogers--Ramanujan identities assert that
\begin{align}
    \sum_{n=0}^\infty\frac{q^{n^2}}{(q;q)_n}
    =
    \frac{1}{(q,q^4;q^5)_\infty}, \quad
    \sum_{n=0}^\infty\frac{q^{n^2+n}}{(q;q)_n}
    =
    \frac{1}{(q^2,q^3;q^5)_\infty}.\label{RR}
\end{align}
They were first proved by Rogers \cite{Rogers1894} and later rediscovered by Ramanujan. The Andrews--Gordon identity (see \cite{Andrews1974,Gordon1961}) generalized \eqref{RR} to arbitrary odd moduli. It states that for integer $k\geq 2$ and $1\leq i \leq k$,
\begin{align}
\sum_{n_1,\dots,n_{k-1}\geq 0} \frac{q^{N_1^2+\cdots+N_{k-1}^2+N_i+\cdots +N_{k-1}}}{(q;q)_{n_1}\cdots (q;q)_{n_{k-2}} (q;q)_{n_{k-1}}}=\frac{(q^i,q^{2k+1-i},q^{2k+1};q^{2k+1})_\infty}{(q;q)_\infty},\label{AG}
\end{align}
where $N_j=n_j+\cdots+n_{k-1}$ if $j\leq k-1$ and $N_k=0$.  As the even moduli companion of it, Bressoud's identity \cite{Bressoud1980} asserts that for integers $k\geq 2$ and $1\leq i \leq k$,
\begin{align}\label{eq-Bressoud}
\sum_{n_1,\dots,n_{k-1}\geq 0} \frac{q^{N_1^2+\cdots+N_{k-1}^2+N_i+\cdots +N_{k-1}}}{(q;q)_{n_1}\cdots (q;q)_{n_{k-2}} (q^2;q^2)_{n_{k-1}}} =\frac{(q^i,q^{2k-i},q^{2k};q^{2k})_\infty}{(q;q)_\infty}
\end{align}
where $N_j$ is defined as before.

The Rogers--Ramanujan identities motivate people to find similar sum-product identities wherein the sum side is a mixed sum of some basic hypergeometric series, and the product side is a sum of some nice infinite products. Such kind of identities are usually called Rogers--Ramanujan type identities. So far there are lots of such identities in the literature and they attract great attention from different areas. The reader may consult Sills' book \cite{Sills-book} for more details.

Rogers--Ramanujan type identities have important implications in combinatorics and number theory, and they are also closely related to Lie algebras and physics. In combinatorics, the sum side and product side of some Rogers--Ramanujan type identities can be interpreted as generating functions of two different kinds of partitions, and the identity implies that such partitions are equinumerous. Around the 1980s, Lepowsky and Wilson \cite{Lepowsky1984,Lepowsky1985} developed a Lie-theoretic approach to discover and prove Rogers--Ramanujan type identities. The sum side of such identities can sometimes be regarded as characters of standard modules of some affine Kac--Moody Lie algebras which turn out to be the infinite products in the product side. In particular, the identities \eqref{AG} and \eqref{eq-Bressoud} correspond to the standard modules of the Kac--Moody Lie algebra $A_1^{(1)}$ of odd and even levels, respectively.

If we look at the Rogers--Ramanujan identities from the aspect of modular forms, it is clear that the product sides of them are modular after multiplying by suitable powers of $q$. Here following \cite{Mizuno,Zagier} we say that a $q$-series $f(q)$ is modular to mean that the function $\widetilde{f}(\tau)=f(e^{2\pi i\tau})$ is invariant with respect to some finite index subgroup of $\mathrm{SL}(2,\mathbb{Z})$. In contrast, modular properties of the sum sides are not easy to be observed or proved. A natural and very important problem in the theory of $q$-series and modular forms is to characterize modular basic hypergeometric series. In this direction, Nahm considered the following particular class of multi-sum basic hypergeometric series:
$$f_{A,b,c}(q):=\sum_{n=(n_1,\dots,n_r)^\mathrm{T} \in (\mathbb{Z}_{\geq 0})^r} \frac{q^{\frac{1}{2}n^\mathrm{T} An+n^\mathrm{T} b+c}}{(q;q)_{n_1}\cdots (q;q)_{n_r}},$$
where $r\geq 1$ is a positive integer, $A$ is a real positive definite symmetric $r\times r$ matrix, $b$ is a vector of length $r$, and $c$ is a scalar. These are called \textit{Nahm sum} or \textit{Nahm series}. With motivation from physics, Nahm proposed the following problem: determine all rational matrix $A$, rational vector $b$ and rational scalar $c$ such that the Nahm series $f_{A,b,c}(q)$ is modular. We call such $(A,b,c)$ a \textit{modular triple}. These modular Nahm series are expected to be characters of some rational conformal field theories.

A conjecture including a criterion on $A$ so that $A$ is the matrix part of a modular triple was stated in \cite{Zagier}, and it is usually referred to as Nahm's conjecture. Zagier \cite{Zagier} confirmed this conjecture in the rank one case by proving that there are exactly seven modular triples $(A,b,c)$:
\begin{align}
   & (1/2,0,-1/40), \quad (1/2,1/2,1/40), \quad (1,0,-1/48), \quad (1,1/2,1/24), \nonumber \\
   &(1,-1/2,1/24), \quad (2,0,-1/60), \quad (2,1,11/60).
\end{align}
Note that the last two triples correspond to the Rogers--Ramanujan identities \eqref{RR}.

While Nahm's problem has been solved for the rank one case, a solution for the general case seems far from reach. When the rank $r\geq 2$, Nahm's conjecture needs further modifications as counterexamples were found by Vlasenko and Zwegers \cite{VZ}. Nevertheless, one direction of Nahm's conjecture has been confirmed by recent work of Calegari, Garoufalidis and Zagier \cite{CGZ}.

To get a better understanding of Nahm's problem, a number of works have been done to find as many modular triples as possible.
After an extensive search, Zagier \cite{Zagier} provided 11 and 12 sets of possible modular triples in the rank two and rank three cases, respectively. The modularity of these rank two modular triples were confirmed in several works such as the works of Zagier \cite{Zagier}, Vlasenko and Zwegers \cite{VZ}, Cherednik and Feigin \cite{Feigin}, Wang \cite{Wang-rank2} and Cao, Rosengren and Wang \cite{CRW}. Zagier's rank three modular triples have all been confirmed in the works of Zagier \cite{Zagier} and Wang \cite{Wang-rank3}.

In 2023, Mizuno \cite{Mizuno} considered the following generalized Nahm series:
\begin{align}\label{eq-general-Nahm-sum}
   {f}_{A,b,c,d}(q):= \sum_{n=(n_1,\dots,n_r)^\mathrm{T} \in (\mathbb{Z}_{\geq 0})^r} \frac{q^{\frac{1}{2}n^\mathrm{T}ADn+n^\mathrm{T}b+c}}{(q^{d_1};q^{d_1})_{n_1}\cdots (q^{d_r};q^{d_r})_{n_r}}.
\end{align}
Here $d=(d_1,\dots,d_r)\in \mathbb{Z}_{>0}^r$,   $b \in \mathbb{Q}^r$ is a vector and  $c \in \mathbb{Q}$ is a scalar. Following \cite{Mizuno}, we call $A \in \mathbb{Q}^{r \times r} $ a symmetrizable matrix with the symmetrizer $D := \mathrm{diag}(d_1,\dots, d_r)$ if $AD$ is symmetric positive definite, and we call $(A,b,c,d)$ a modular quadruple when $f_{A,b,c,d}(q)$ is modular. Following the strategy in \cite{Zagier} based on asymptotic analysis, Mizuno provided 14 sets of possible rank two modular quadruples, and he presented 19 and 15 sets of possible rank three modular quadruples with symmetrizers $\mathrm{diag}(1,1,2)$ and $\mathrm{diag}(1,2,2)$, respectively.

For convenience, following the notion in \cite{Wang2022}, we  call ${f}_{A,b,c,d}(q)$ in \eqref{eq-general-Nahm-sum} a Nahm sum of index $(d_1,\dots,d_r)$, and call a Rogers--Ramanujan type identity involving such Nahm sums an identity of index $(d_1,\dots,d_r)$. For instance, the Andrews--Gordon identity \eqref{AG} and Bressoud's identity \eqref{eq-Bressoud} are of indices $(1,1,\dots,1)$ and $(1,1,\dots,1,2)$, respectively.

So far the modularity of Mizuno's rank two examples (see \cite[Table 2]{Mizuno}) have almost all been confirmed. Mizuno \cite{Mizuno} justified the modularity of four sets of rank two modular quadruples.  Together with some results in the literature, the authors \cite{WW-rank2} confirmed the modularity of eight other sets of Mizuno's rank two modular quadruples, and proposed conjectural identities for the remaining two sets. One of the remaining sets correspond to the Kande--Russell mod 9 conjectural identities (see \cite{K2015,KR2019} or \cite[Conjecture 6.1]{Kursungoz-JCTA}).

As for the rank three case, we \cite{WW-I} proved that all the 19 sets of Nahm sums of index $(1,1,2)$ listed in \cite[Table 2]{Mizuno} are indeed modular. As a sequel to \cite{WW-I}, this work is devoted to confirming the modularity of all of Mizuno's rank three Nahm sums of index $(1,2,2)$ listed in \cite[Table 3]{Mizuno}. We will establish Rogers--Ramanujan type identities for each of the Nahm sums by expressing them in terms of infinite products. The product side will show their modularity clearly.

It should be emphasized that several examples considered in the current paper are more difficult and deeper than the Nahm sums studied in \cite{WW-I}, and the proofs are more complicated. The main method used in \cite{WW-I} is the Bailey machinery, but here we need to use various techniques such as $q$-difference equations, constant term methods and construct some new interesting Bailey pairs. Besides proving all of Mizuno's conjectures on Nahm sums of index $(1,2,2)$, we also obtain several surprising results such as the distinct behavior of two Nahm sums (see Theorem \ref{thm-S}).

For convenience, we label the examples from 1 to 15 according to their order of appearance in \cite[Table 3]{Mizuno}. The corresponding matrices and number of vectors for these examples are listed in Table \ref{tab-matrix}, and the details can be found in Section \ref{sec-exam}. Here we classify the 15 examples into several groups. Those examples in the same group share analogous pattern or are treated using similar techniques.
\begin{table}[htbp]
\begin{tabular}{c|ccc} \hline
  \makecell{{\small Exam.} \\ {\small No.} } & Matrix $A$ & Matrix $AD$  & \makecell{{\small Number of} \\ {\small vectors $b$}}   \\
  \hline
1 & $\left(\begin{smallmatrix}
1 & 0 & 1\\
0 & 2 & 2\\
1/2&1 & 2
\end{smallmatrix}\right)$ & $\left(\begin{smallmatrix}
2 & 0 & 1\\
0 & 4 & 2\\
1 & 2 & 2
\end{smallmatrix}\right)$ & 4 \\
 7 & $\left(\begin{smallmatrix}
1 & 1 & 1\\
1 & 2 & 2\\
1/2&1 &3/2\\
\end{smallmatrix}\right)$ & $\left(\begin{smallmatrix}
2 & 2 & 1\\
2 & 4 & 2\\
1 & 2 &3/2\\
\end{smallmatrix}\right)$ & 3\\
 8 & $\left(\begin{smallmatrix}
1 & 0 & 1\\
0 & 2 & 2\\
1/2 & 1 & 5/2
\end{smallmatrix}\right)$ & $\left(\begin{smallmatrix}
2 & 0 & 1\\
0 & 4 & 2\\
1 & 2 & 5/2
\end{smallmatrix}\right)$  & 4 \\
\hline
 2 & $\left(\begin{smallmatrix}
1 & 1 & 1\\
1 & 2 & 2\\
1/2 & 1 & 2
\end{smallmatrix}\right)$ & $\left(\begin{smallmatrix}
2 & 2 & 1\\
2 & 4 & 2\\
1 & 2 & 2
\end{smallmatrix}\right)$ & 3 \\
 9 & $\left(\begin{smallmatrix}
1 & 1 & 1\\
1 & 3 & 3\\
1/2&3/2&5/2
\end{smallmatrix}\right)$ &
 $\left(\begin{smallmatrix}
2 & 2 & 1\\
2 & 6 & 3\\
1 & 3&5/2
\end{smallmatrix}\right)$ & 3 \\
 11 & $\left( \begin{smallmatrix}
4 & 4 & 4\\
4 & 6 & 6\\
2 & 3 & 4\\
\end{smallmatrix}\right)$ & $\left(\begin{smallmatrix}
8 & 8 & 4\\
8 & 12 & 6\\
4 & 6 & 4\\
\end{smallmatrix}\right)$ &  2\\
\hline
 3 & $\left(\begin{smallmatrix}
1 & 1 & 0\\
1 & 4 & 4\\
0 & 2 & 3
\end{smallmatrix}\right)$ & $\left(\begin{smallmatrix}
2 & 2 & 0\\
2 & 8 & 4\\
0 & 4 & 3
\end{smallmatrix}\right)$  &  2 \\
\hline
 4 & $\left(\begin{smallmatrix}
2 & 1 & 0\\
1 & 2 & 2\\
0 & 1 & 2
\end{smallmatrix}\right)$ & $\left(\begin{smallmatrix}
4 & 2 & 0\\
2 & 4 & 2\\
0 & 2 & 2
\end{smallmatrix}\right)$ & 3 \\
\hline
 5 & $\left(\begin{smallmatrix}
2 & 2 & 2\\
2 & 4 & 4\\
1 & 2 & 3
\end{smallmatrix}\right)$ & $\left(\begin{smallmatrix}
4 & 4 & 2\\
4 & 8 & 4\\
2 & 4 & 3
\end{smallmatrix}\right)$ & 10 \\
 \hline
6 & $\left(\begin{smallmatrix}
4 & 4 & 2\\
4 & 5 & 2\\
1 & 1 & 1
\end{smallmatrix}\right)$  & $\left(\begin{smallmatrix}
8 & 8 & 2\\
8 & 10 & 2\\
2 & 2 & 1
\end{smallmatrix}\right)$  & 2 \\
\hline
 10 & $\left(\begin{smallmatrix}
3/2&1/2&1\\
1/2&3/2&1\\
1/2&1/2&1
\end{smallmatrix}\right)$  & $\left(\begin{smallmatrix}
3&1&1\\
1&3&1\\
1&1&1
\end{smallmatrix}\right)$ & 10 \\
12 & $\left(\begin{smallmatrix}
3/2&1/2& 1\\
1/2&3/2& 1\\
1/2&1/2&3/2
\end{smallmatrix}\right)$ &
$\left(\begin{smallmatrix}
3&1& 1\\
1&3& 1\\
1&1&3/2
\end{smallmatrix}\right)$  & 6 \\
13 & $\left(\begin{smallmatrix}
3/2&1/2& 2\\
1/2&3/2& 2\\
1 & 1 & 4
\end{smallmatrix}\right)$ &
$\left(\begin{smallmatrix}
3 & 1 & 2\\
1 & 3 & 2\\
2 & 2 & 4
\end{smallmatrix}\right)$ & 6 \\
14 & $\left(\begin{smallmatrix}
5/2&3/2& 1\\
3/2&5/2& 1\\
1/2&1/2& 1
\end{smallmatrix}\right)$ & $\left(\begin{smallmatrix}
5&3& 1\\
3&5& 1\\
1&1& 1
\end{smallmatrix}\right)$ &  4 \\
15 & $\left(\begin{smallmatrix}
5/2&3/2& 2\\
3/2&5/2& 2\\
1 & 1 & 2
\end{smallmatrix}\right)$ & $\left(\begin{smallmatrix}
5&3& 2\\
3&5& 2\\
2&2& 2
\end{smallmatrix}\right)$ & 6 \\
  \hline
\end{tabular}
\\[2mm]
\caption{Matrices in Mizuno's examples with symmetrizer $D=\diag(2,2,1)$}
\label{tab-matrix}
\end{table}

We prove identities of index $(1,2,2)$ for each of the Nahm series ${f}_{A,b,c,d}(q)$. For instance, for the four modular quadruples in Example 1, we establish the following theorem to prove their modularity.
\begin{theorem}\label{thm-exam-1}
    We have
    \begin{align}
    &\sum_{i,j,k\ge 0}\frac{q^{i^2+2j^2+k^2+2ij+ik}}{(q;q)_i(q^2;q^2)_j(q^2;q^2)_k}=
    \frac{(q^{20},q^{24},q^{44};q^{44})_\infty}{(q,q^3,q^4;q^4)_\infty}
    +q\frac{(-q,q^{10},-q^{11};q^{11})_\infty
    }{(q^2,q^2,q^4;q^4)_\infty},\label{table3.1.1}
    \\
    &\sum_{i,j,k\ge 0}\frac{q^{i^2+2j^2+k^2+2ij+ik+k}}{(q;q)_i(q^2;q^2)_j(q^2;q^2)_k}=
    \frac{(-q^5,q^6,-q^{11};-q^{11})_\infty}{(q^2,q^2,q^4;q^4)_\infty}
    +q\frac{(q^{12},q^{32},q^{44};q^{44})_\infty}{(q,q^3,q^4;q^4)_\infty},\label{table3.1.2}\\
    &\sum_{i,j,k\ge 0}\frac{q^{i^2+2j^2+k^2+2ij+ik+i+k}}{(q;q)_i(q^2;q^2)_j(q^2;q^2)_k}=
    \frac{(q^4,-q^7,-q^{11};-q^{11})_\infty}{(q^2,q^2,q^4;q^4)_\infty}
    +q^2\frac{(q^8,q^{36},q^{44};q^{44})_\infty}{(q,q^3,q^4;q^4)_\infty},\label{table3.1.3}\\
    &\sum_{i,j,k\ge 0}\frac{q^{i^2+2j^2+k^2+2ij+ik+2i+2j+k}}{(q;q)_i(q^2;q^2)_j(q^2;q^2)_k}=
    \frac{(q^2,-q^9,-q^{11};-q^{11})_\infty}{(q^2,q^2,q^4;q^4)_\infty}
    +q^3\frac{(q^4,q^{40},q^{44};q^{44})_\infty}{(q,q^3,q^4;q^4)_\infty}\label{table3.1.4}.
\end{align}
\end{theorem}
Note that the products $(-q^3,q^8,-q^{11};-q^{11})_\infty$ and $(q^{16},q^{28},q^{44};q^{44})_\infty$ are missing in the product side. We find the following companion result, which represents these missing products by combination of Nahm sums of the same type:
\begin{theorem}\label{thm-exam-1-2}
We have
\begin{align}
&\sum_{i,j,k\ge 0}\frac{q^{i^2+2j^2+k^2+2ij+ik+i}(1+q^{2i+2j+k})}{(q;q)_i(q^2;q^2)_j(q^2;q^2)_k} \nonumber \\
     & =
    \frac{(-q^3,q^8,-q^{11};-q^{11})_\infty}{(q^2,q^2,q^4;q^4)_\infty}
    +\frac{(q^{16},q^{28},q^{44};q^{44})_\infty}{(q,q^3,q^4;q^4)_\infty},  \label{revise-add-exam1-companion} \\
     &\sum_{i,j,k\ge 0}\frac{q^{i^2+2j^2+k^2+2ij+ik+k}(1+q^{i+1}-q^{2i})}{(q;q)_i(q^2;q^2)_j(q^2;q^2)_k} \nonumber \\
     &=
    q\frac{(-q^3,q^8,-q^{11};-q^{11})_\infty}{(q^2,q^2,q^4;q^4)_\infty}
    +q\frac{(q^{16},q^{28},q^{44};q^{44})_\infty}{(q,q^3,q^4;q^4)_\infty}\label{add-exam1-companion}.
\end{align}
\end{theorem}

Along our investigation of Mizuno's examples, we establish some beautiful identities as byproducts. For example, in order to confirm the modularity of Example 4, we prove the following remarkable relations between different Nahm sums.
\begin{theorem}\label{thm-exam-4-relation}
We have
\begin{align}
   & \sum_{i,j,k\ge 0}\frac{q^{(i+j)^2+(j+k)^2+k^2+j}x^{i+2j+2k}}{(q;q)_i(q^2;q^2)_j(q^2;q^2)_k}
    =
    \sum_{i,j,k\ge 0}\frac{q^{i^2+(i+j)^2+(i+j+k)^2}x^{3i+2j+k}}{(q;q)_i(q;q)_j(q;q)_k}, \label{conj-u-id-1}
    \\
& \sum_{i,j,k\ge 0}\frac{q^{(i+j)^2+(j+k)^2+k^2+j+2k}x^{i+2j+2k}}{(q;q)_i(q^2;q^2)_j(q^2;q^2)_k}
=\sum_{i,j,k\ge 0}\frac{q^{i^2+(i+j)^2+(i+j+k)^2+2i+j}x^{3i+2j+k}}{(q;q)_i(q;q)_j(q;q)_k}. \label{conj-u-id-2}
\end{align}
\end{theorem}
We use various methods to establish identities for Mizuno's examples. The Nahm sums in Example 3 can be directly reduced to some known double sums by summing over one of the indices first. In Examples 2, 6 and 9--15, this reduction process is achieved by using some finite summation formulas. The remaining examples are considerably more difficult. For Examples 1, 7 and 8 we will split the Nahm sums into two parts according to the parity of one of the summation indices and then evaluate the two parts separately. Examples 7 and 8 are closely connected in the sense that some parts of their Nahm sums can be converted to each other. We will use $q$-difference equations in Example 4 especially for proving Theorem \ref{thm-exam-4-relation}.  For Example 5 we need to use the constant term method. Note that Bailey pairs play an important role in Examples 1, 5, 7 and 8.

After expressing the Nahm sums in terms of infinite products, we find that all of the Nahm sums in \cite[Table 3]{Mizuno} are modular of weight zero except for two cases in Example 5. For these two cases we have the following product representations.
\begin{theorem}\label{thm-S}
 We have
 \begin{align}
      &\sum_{i,j,k\ge 0}\frac{q^{\frac{3}{2}i^2+2j^2+4k^2+2ij+4ik+4jk+\frac{1}{2}i+j}}{(q;q)_i(q^2;q^2)_j(q^2;q^2)_k}=\frac{1}{4}\left(3\frac{(q^2;q^2)_\infty}{(q;q)_\infty}
      +\frac{(q;q)_\infty^3}{(q^2;q^2)_\infty}\right), \label{S1-product} \\
       &\sum_{i,j,k\ge 0}\frac{q^{\frac{3}{2}i^2+2j^2+4k^2+2ij+4ik+4jk+\frac{5}{2}i+3j+4k}}{(q;q)_i(q^2;q^2)_j(q^2;q^2)_k}=\frac{1}{4}q^{-1}\left(\frac{(q^2;q^2)_\infty}{(q;q)_\infty}
       -\frac{(q;q)_\infty^3}{(q^2;q^2)_\infty}  \right). \label{S2-product}
 \end{align}
\end{theorem}
Clearly, after multiplying both sides of \eqref{S1-product} and \eqref{S2-product} by $q^{1/24}$ and $q^{25/24}$, respectively, the product sides express these Nahm sums as sums of modular forms of weights 0 and 1 and hence they are in the ring of modular forms. The definition of modularity can easily be adjusted to include these cases.

Mizuno also observed some deep relations between two sets of Nahm sums with indices $(1,1,2)$ and $(1,2,2)$. Namely, for matrices
$$A=\begin{pmatrix}
        1 & 0 & \frac{1}{2} \\
        0 & 2 & 1 \\
        1 & 2 & 2
    \end{pmatrix} \quad \text{and} \quad A^{\mathrm{T}}=\begin{pmatrix}
        1 & 0 & 1 \\ 0 & 2 & 2 \\ \frac{1}{2} & 1 & 2
    \end{pmatrix}$$
with symmetrizers $d=(1,1,2)$ and $d^{\vee}=(2,2,1)$, respectively, we define the following Nahm sums associated with the matrices $A\diag(1,1,2)$ and $A^\mathrm{T}\diag(2,2,1)$:
\begin{align}
&F(u,v,w;q^{\frac{1}{2}}):=\sum_{i,j,k\ge 0}\frac{q^{\frac{1}{2}i^2+j^2+2k^2+ik+2jk}u^iv^jw^k}{(q;q)_i(q;q)_j(q^2;q^2)_k}, \label{F-defn}\\
&G(u,v,w;q):=\sum_{i,j,k\geq 0} \frac{q^{i^2+2j^2+k^2+2ij+ik}u^{i}v^jw^k}{(q;q)_{i}(q^2;q^2)_j(q^2;q^2)_k}.\label{G-defn}
\end{align}
For $\sigma \in \{0,1\}$, we also consider the following partial Nahm sums:
\begin{align}
&F_\sigma(u,v,w;q^{\frac{1}{2}}):=\sum_{\substack{i\equiv \sigma \!\!\!\pmod{2} \\ i,j,k\ge 0}}\frac{q^{\frac{1}{2}i^2+j^2+2k^2+ik+2jk}u^iv^jw^k}{(q;q)_i(q;q)_j(q^2;q^2)_k}, \label{Fc-defn} \\
&G_\sigma(u,v,w;q):=\sum_{\begin{smallmatrix}
i\equiv \sigma \!\! \pmod{2} \\ i,j,k\geq 0
\end{smallmatrix}}\frac{q^{i^2+2j^2+k^2+2ij+ik}u^{i}v^jw^k}{(q;q)_{i}(q^2;q^2)_j(q^2;q^2)_k}.\label{Gc-defn}
\end{align}
Let $q=e^{2\pi i \tau}$ where $\mathrm{Im}~ \tau>0$ and let $\zeta_N=e^{2\pi i/N}$ throughout this paper. We define two vector-valued functions:
\begin{align}
g(\tau)&:=\begin{pmatrix}q^{-\frac{5}{88}}F_0(1,1,1;q^{\frac{1}{2}}) \\ q^{-\frac{1}{88}}F_1(1,1,q;q^{\frac{1}{2}}) \\ q^{\frac{7}{88}}(F_1(1,q,q;q^{\frac{1}{2}})+qF_1(q,q^2,q^3;q^{\frac{1}{2}})) \\
q^{\frac{19}{88}}F_0(1,q,q^2;q^{\frac{1}{2}}) \\ q^{\frac{35}{88}}F_0(q,q,q^2;q^{\frac{1}{2}})  \end{pmatrix}, \\
g^{\vee}(\tau)&:=\begin{pmatrix} q^{-\frac{7}{88}}G(1,1,1;q) \\ q^{\frac{25}{88}}(G(q,1,1;q)+G(q^3,q^2,q;q)) \\
q^{\frac{1}{88}}G(1,1,q;q) \\ q^{\frac{9}{88}}G(q,1,q;q) \\ q^{\frac{49}{88}}G(q^2,q^2,q;q)  \end{pmatrix}.
\end{align}
It is easy to check that \cite[p.\ 28]{Mizuno}
\begin{align}
g(\tau+1)&=\diag(\zeta_{88}^{-5},\zeta_{88}^{43},\zeta_{88}^{51},\zeta_{88}^{19},\zeta_{88}^{35} )g(\tau), \\   g^\vee(\tau+1)&=\diag(\zeta_{88}^{-7},\zeta_{88}^{25},\zeta_{88}^1,\zeta_{88}^9,\zeta_{88}^{49}) g^\vee(\tau).
\end{align}
Mizuno \cite[Eq.\ (54)]{Mizuno} conjectured the following transformation formula between $g(\tau)$ and $g^\vee(\tau)$.
\begin{conj}(Cf.\ \cite[Eq.\ (54)]{Mizuno})\label{conj-M}
We have
\begin{align}\label{eq-M54}
g^\vee\left(-\frac{1}{2\tau}\right)=2Sg(\tau), \quad g\left(-\frac{1}{2\tau}\right)=Sg^\vee (\tau)
\end{align}
where
\begin{align*}
    S=\begin{pmatrix}
    \alpha_5 & \alpha_4 & \alpha_3 & \alpha_2 & \alpha_1 \\
    \alpha_4 & \alpha_1 & -\alpha_2 & -\alpha_5 & -\alpha_3 \\
    \alpha_3 & -\alpha_2 & -\alpha_4 & \alpha_1 & \alpha_5 \\
    \alpha_2 & -\alpha_5 & \alpha_1 & \alpha_3 & -\alpha_4 \\
    \alpha_1 & -\alpha_3 & \alpha_5 & -\alpha_4 & \alpha_2
    \end{pmatrix}, \quad \alpha_k=\sqrt{\frac{2}{11}}\sin \frac{k\pi}{11}.
\end{align*}
\end{conj}

We further define
\begin{align}
U(\tau)&:=\begin{pmatrix}q^{-\frac{5}{88}}F(1,1,1;q^{\frac{1}{2}}) \\ q^{-\frac{1}{88}}F(1,1,q;q^{\frac{1}{2}}) \\ q^{\frac{7}{88}}(F(1,q,q;q^{\frac{1}{2}})+qF(q,q^2,q^3;q^{\frac{1}{2}})) \\
q^{\frac{19}{88}}F(1,q,q^2;q^{\frac{1}{2}}) \\ q^{\frac{35}{88}}F(q,q,q^2;q^{\frac{1}{2}})  \end{pmatrix}, \\
V(\tau)&:=\begin{pmatrix} q^{-\frac{7}{88}}G_{0}(1,1,1;q) \\ q^{\frac{25}{88}}(G_0(q,1,1;q)+G_1(q^3,q^2,q;q)) \\
q^{\frac{1}{88}}G_1(1,1,q;q) \\ q^{\frac{9}{88}}G_1(q,1,q;q) \\ q^{\frac{49}{88}}G_1(q^2,q^2,q;q)  \end{pmatrix}
\end{align}
and
\begin{align}\label{wU-wV-defn}
U^*(\tau):=2g(\tau)-U(\tau), \quad {V}^*(\tau):=g^\vee(\tau)-V(\tau).
\end{align}
It is easy to check that
\begin{align}
U(\tau+2)=\diag(\zeta_{44}^{-5},\zeta_{44}^{43},\zeta_{44}^{51},\zeta_{44}^{19},\zeta_{44}^{35} )U(\tau), \\   V(\tau+1)=\diag(\zeta_{88}^{-7},\zeta_{88}^{25},\zeta_{88}^1,\zeta_{88}^9,\zeta_{88}^{49}) V(\tau).
\end{align}

We establish the following transformation formulas between $U(\tau), U^*(\tau)$,  $V(\tau)$ and $V^*(\tau)$ and thereby confirm Mizuno's conjectural formula \eqref{eq-M54}.
\begin{theorem}\label{thm-Mizuno-conj}
We have
\begin{align}\label{eq-conj-tran}
V\left(-\frac{1}{2\tau}\right)=SU(\tau), \quad U\left(-\frac{1}{2\tau}\right)=2SV(\tau)
\end{align}
and
\begin{align}\label{eq-conj-tran-add}
V^*\left(-\frac{1}{2\tau}\right)=SU^*(\tau), \quad U^*\left(-\frac{1}{2\tau}\right)=2SV^*(\tau).
\end{align}
As a consequence, Conjecture \ref{conj-M} holds.
\end{theorem}

The rest of this paper is organized as follows. In Section \ref{sec-pre} we collect some useful identities, and we review some basic facts about Bailey pairs and modular forms. We also present four new Bailey pairs and some new single-sum Rogers--Ramanujan type identities. In Section \ref{sec-exam} we discuss Mizuno's rank three Nahm sums of index $(1,2,2)$  one by one according to the order in Table \ref{tab-matrix} and present identities for each of them.  Meanwhile, employing the theory of modular forms, we will prove the modular transformation formulas stated in Conjecture \ref{conj-M} and Theorem \ref{thm-Mizuno-conj}.

\section{Preliminaries}\label{sec-pre}
\subsection{Auxiliary identities and Bailey pairs}
First, we need the $q$-binomial theorem \cite[Theorem 2.1]{Andrews}:
\begin{align}\label{eq-qbinomial}
    \sum_{n=0}^\infty \frac{(a;q)_nz^n}{(q;q)_n}=\frac{(az;q)_\infty}{(z;q)_\infty}, \quad |z|<1
\end{align}
and the Jacobi triple product identity \cite[Theorem 2.8]{Andrews}
\begin{align}\label{Jacobi}
(q,z,q/z;q)_\infty=\sum_{n=-\infty}^\infty (-1)^nq^{\binom{n}{2}}z^n.
\end{align}
As two important corollaries of \eqref{eq-qbinomial}, Euler's $q$-exponential identities state that \cite[Corollary 2.2]{Andrews}
\begin{align}
\sum_{n=0}^{\infty}\frac{z^n}{(q;q)_n}
&=
\frac{1}{(z;q)_{\infty}}, \quad|z|<1, \label{euler-1} \\
\sum_{n=0}^{\infty}\frac{q^{(^n_2)} z^n}{(q;q)_n}
&=
(-z;q)_{\infty}. \label{euler-2}
\end{align}
The identities \eqref{eq-qbinomial}--\eqref{euler-2} will be used frequently and often without mention.

 A pair of sequences $(\alpha_n(a;q),\beta_n(a;q))$ is called a Bailey pair relative to $a$ if for all $n\geq 0$,
 \begin{align}\label{defn-BP}
     \beta_n(a;q)=\sum_{k=0}^n\frac{\alpha_k(a;q)}{(q;q)_{n-k}(aq;q)_{n+k}}.
 \end{align}

 \begin{lemma}[Bailey's Lemma]
Suppose that $(\alpha_n(a;q),\beta_n(a;q))$ is a Bailey pair relative to $a$. Then $(\alpha_n'(a;q),\beta_n'(a;q))$ is also a Bailey pair relative to $a$ where
 \begin{equation}
 \begin{split}
 \alpha_n'(a;q)&:=\frac{(\rho_1,\rho_2;q)_n(aq/\rho_1\rho_2)^n}{(aq/\rho_1,aq/\rho_2;q)_n}\alpha_n(a;q), \\ \beta_n'(a;q)&:=\sum_{r=0}^n\frac{(\rho_1,\rho_2;q)_r(aq/\rho_1\rho_2;q)_{n-r}(aq/\rho_1\rho_2)^r}{(aq/\rho_1,aq/\rho_2;q)_n(q;q)_{n-r}}\beta_r(a;q).
 \end{split}
 \end{equation}
 Equivalently, if $(\alpha_n(a;q),\beta_n(a;q))$ is a Bailey pair relative to $a$, then
\begin{align}\label{eq-Bailey-general-id}
&\frac{1}{(aq/\rho_1,aq/\rho_2;q)_n}\sum_{j=0}^n \frac{(\rho_1,\rho_2;q)_j(aq/\rho_1\rho_2;q)_{n-j}}{(q;q)_{n-j}}\Big(\frac{aq}{\rho_1\rho_2} \Big)^j\beta_j(a;q) \nonumber \\
&=\sum_{r=0}^n \frac{(\rho_1,\rho_2;q)_r}{(q;q)_{n-r}(aq;q)_{n+r}(aq/\rho_1,aq/\rho_2;q)_r}\Big(\frac{aq}{\rho_1\rho_2} \Big)^r \alpha_r(a;q).
\end{align}
 \end{lemma}
If we let $\rho_1,\rho_2 \to \infty$, we obtain the Bailey pair \cite[Eq. (S1)]{BIS}:
\begin{align}
        \alpha'_n(a;q)=a^nq^{n^2}\alpha_n(a;q), \quad
        \beta'_n(a;q)=\sum_{r=0}^n\frac{a^rq^{r^2}}{(q;q)_{n-r}}\beta_r(a;q). \label{Bailey lemma-infty}
\end{align}

If we let $n,\rho_1,\rho_2\to \infty$ in \eqref{eq-Bailey-general-id}, we obtain the following result (see, e.g. \cite[Eq.\ (1.2.8)]{MSZ}).
\begin{lemma}\label{lem-BP-sum}
If $(\alpha_n(a;q),\beta_n(a;q))$ is a Bailey pair relative to $a$, we have
\begin{align}
    \sum_{n=0}^\infty a^nq^{n^2}\beta_n(a;q)=\frac{1}{(aq;q)_\infty}\sum_{n=0}^{\infty}a^nq^{n^2}\alpha_n(a;q). \label{cor1}
\end{align}
\end{lemma}
If we substitute the Bailey pair \eqref{Bailey lemma-infty} into Lemma \ref{lem-BP-sum}, we obtain
\begin{align}\label{id-BP-twice}
\sum_{n=0}^\infty a^nq^{n^2} \sum_{r=0}^n \frac{a^rq^{r^2}}{(q;q)_{n-r}}\beta_r(a;q)=\frac{1}{(aq;q)_\infty} \sum_{n=0}^\infty a^{2n}q^{2n^2}\alpha_n(a;q).
\end{align}

We need some known Bailey pairs from \cite[p.\ 469]{Slater1951}.  Below we always let $\alpha_{0}(a;q)=1$ and agree that $\alpha_{n}(a;q)=0$ when $n<0$ unless otherwise stated.

We first recall some Bailey pairs in Group G of Slater's list \cite[p.\ 469]{Slater1951}:
\begin{align}
    &\alpha_n(1;q)=(-1)^nq^{\frac{1}{2}n^2+\frac{1}{2}\binom{n}{2}}(1+q^{\frac{1}{2}n}), \quad \beta_n(1;q)=\frac{1}{(-q^{1/2};q)_n(q^2;q^2)_n}; \tag{G1}
    \label{B-G(1)}
    \\
    &\alpha_n(q;q)=(-1)^nq^{\frac{3}{2}\binom{n+1}{2}}\frac{q^{-n}-q^{n+1}}{1-q}, \quad \beta_n(q;q)=\frac{1}{(-q^{{1}/{2}};q)_n(q^2;q^2)_n};  \tag{G1*} \label{B-G(1.1)} \\
    &\alpha_n(q;q)=(-1)^n q^{\frac{3}{2}\binom{n+1}{2}}\frac{q^{-{n}/{2}}-q^{{(n+1)}/{2}}}{1-q^{{1}/{2}}}, \quad \beta_n(q;q)=\frac{1}{(-q^{{3}/{2}};q)_n(q^2;q^2)_n};   \tag{G2}\label{G(2)} \\
    &\alpha_n(1;q)=(-1)^nq^{\frac{3}{2}\binom{n}{2}}(1+q^{\frac{3}{2}n}), \quad \beta_n(1;q)=\frac{q^n}{(-q^{1/2};q)_n(q^2;q^2)_n}; \tag{G3} \label{G(3)} \\
   &\alpha_n(1;q)=(-1)^nq^{\frac{1}{2}\binom{n}{2}}(1+q^n), \quad \beta_n(1;q)=\frac{(-1)^nq^{\frac{1}{2}n^2}}{(-q^{1/2};q)_n(q^2;q^2)_n}; \quad \tag{G4} \label{G(4)}
    \\
    &\alpha_n(q;q)=\frac{q^{\frac{1}{2}\binom{n}{2}}(1-q^{n+\frac{1}{2}})}{1-q^{\frac{1}{2}}}, \quad \beta_n(q;q)=\frac{(-1)^nq^{\frac{1}{2}n^2}}{(-q^{3/2};q)_n(q^2;q^2)_n}; \quad \tag{G5} \label{G(5)}
    \\
    &\alpha_n(1;q)=(-1)^nq^{\frac{1}{2}\binom{n}{2}-\frac{1}{2}n}(1+q^{\frac{3}{2}n}), \quad \beta_n(1;q)=\frac{(-1)^nq^{\frac{1}{2}n^2-n}}{(-q^{1/2};q)_n(q^2;q^2)_n}.  \tag{G4*} \label{G(4.1)}
\end{align}
Note that we write $\alpha_n(a,q)$ in a unified expression while Slater \cite{Slater1951} wrote different expressions for even  and odd values of $n$ separately. Here \eqref{G(4.1)} is the Bailey pair without label above \eqref{G(4)} in \cite{Slater1951}, and a typo for \eqref{G(5)} has been corrected here. The Bailey pair \eqref{B-G(1.1)} appeared in \cite[Eq.\ (4.4)]{Warnaar2003} (see also  \cite{WW-I}).

Next, we list some Bailey pairs in Group C \cite[p.\ 469]{Slater1951}:
\begin{align}
\begin{split}
    &\alpha_{2n}(1;q)=(-1)^nq^{3n^2}(q^n+q^{-n}), \quad \alpha_{2n+1}(1;q)=0, \\
    &\beta_n(1;q)=\frac{1}{(q;q)_n(q,q^2)_n};
\end{split} \tag{C1} \label{C(1)}  \\
    \begin{split}
    &\alpha_{2n}(q;q)=(-1)^nq^{3n^2+n}, \quad \alpha_{2n+1}(q;q)=(-1)^{n+1}q^{3n^2+5n+2}, \\
     &\beta_n(q;q)=\frac{1}{(q;q)_n(q^3;q^2)_n};
    \end{split} \tag{C3} \label{C(3)}
    \\
   \begin{split}
    &\alpha_{2n}(q;q)=(-1)^nq^{3n^2+3n}, \quad \alpha_{2n+1}(q;q)=(-1)^{n+1}q^{3n^2+3n}, \\
     &\beta_n(q;q)=\frac{q^n}{(q;q)_n(q^3;q^2)_n};
     \end{split} \tag{C4} \label{C(4)}
    \\
    \begin{split}
    &\alpha_{2n}(1;q)=(-1)^nq^{n^2}(q^n+q^{-n}), \quad \alpha_{2n+1}(1;q)=0,\\
     &\beta_n(1;q)=\frac{q^{\frac{1}{2}(n^2-n)}}{(q;q)_n(q;q^2)_n};
     \end{split} \tag{C5} \label{C(5)}
    \\
    \begin{split}
    &\alpha_{2n}(q;q)=(-1)^nq^{n^2-n}, \quad \alpha_{2n+1}(q;q)=(-1)^{n+1}q^{n^2+3n+2}, \\
      &\beta_n(q;q)=\frac{q^{\frac{1}{2}(n^2-n)}}{(q;q)_n(q^3;q^2)_n};
    \end{split} \tag{C6} \label{C(6)}
    \\
    \begin{split}
      &\alpha_{2n}(q;q)=(-1)^nq^{n^2+n}, \quad \alpha_{2n+1}(q;q)=(-1)^{n+1}q^{n^2+n}, \\
    &\beta_n(q;q)=\frac{q^{\frac{1}{2}(n^2+n)}}{(q;q)_n(q^3;q^2)_n}.
    \end{split}   \tag{C7} \label{C(7)}
\end{align}

The following lemma generates a Bailey pair related to $a/q$ from a Bailey pair related to $a$.
\begin{lemma}\label{lem-DJK}
(Cf. \cite[Lemma 1.12]{DJK}).
If $(\alpha_n(a;q),\beta_n(a;q))$ is a Bailey pair related to $a$, then
\begin{align}
    &\alpha_n'(a/q;q):=(1-a)\Big(\frac{1-bq^n}{1-b}\frac{\alpha_n(a;q)}{1-aq^{2n}}-\frac{q^{n-1}(aq^{n-1}-b)}{1-b}\frac{\alpha_{n-1}(a;q)}{1-aq^{2n-2}}\Big), \nonumber \\
    &\beta_n'(a/q;q):=\frac{(bq;q)_n}{(b;q)_n}\beta_n(a;q) \label{Douse-lemma1.12}
\end{align}
is a Bailey pair related to $a/q$.
\end{lemma}
In particular, when $b\rightarrow \infty$, \eqref{Douse-lemma1.12} becomes
\begin{equation}\label{eq-b-infinity-reduce}
\begin{split}
   & \alpha_n'(a/q;q)=(1-a)\Big(\frac{q^n\alpha_n(a;q) }{1-aq^{2n}}-\frac{q^{n-1}\alpha_{n-1}(a;q)}{1-aq^{2n-2}}  \Big), \\
   & \beta_n'(a/q;q)=q^n\beta_n(a;q).
\end{split}
\end{equation}

Conversely, the following result due to Lovejoy \cite[Lemma 2.3]{Lovejoy2022} gives a way to generate a Bailey pair related to $aq$ from a Bailey pair related to $a$.
\begin{lemma}
If $(\alpha_n(a;q),\beta_n(a;q))$ is a Bailey pair related to $a$, then
\begin{align}\label{eq-raise-power}
&\alpha_n'(aq;q)=\frac{(1-aq^{2n+1})(aq/b;q)_n(-b)^nq^{n(n-1)/2}}{(1-aq)(bq;q)_n} \nonumber \\
&\qquad \qquad \qquad \times \sum_{r=0}^n \frac{(b;q)_r}{(aq/b;q)_r}(-b)^{-r}q^{-r(r-1)/2}\alpha_r(a;q),  \\
&\beta_n'(aq;q)=\frac{(b;q)_n}{(bq;q)_n}\beta_n(a;q) \nonumber
\end{align}
is a Bailey pair related to $aq$.
\end{lemma}
In particular, when $b\rightarrow \infty$, \eqref{eq-raise-power} becomes
\begin{equation}\label{eq-b-infinity-raise}
\alpha_n'(aq;q)=\frac{1-aq^{2n+1}}{1-aq}q^{-n}\sum_{r=0}^n \alpha_r(a;q), \quad \beta_n'(aq;q)=q^{-n}\beta_n(a;q).
\end{equation}
Applying \eqref{eq-b-infinity-raise} to the Bailey pair \eqref{C(4)}, we obtain the following Bailey pair:
\begin{equation}
 \begin{split}
    &\alpha_{2n}(q^2;q)=(-1)^nq^{3n^2+n}\frac{1-q^{4n+2}}{1-q^2}, \quad \alpha_{2n+1}(q^2;q)=0, \\
    &\beta_n(q^2;q)=\frac{1}{(q;q)_n(q^3;q^2)_n}.
    \end{split} \tag{C4*} \label{C4*}
\end{equation}
Applying \eqref{eq-b-infinity-raise} to the Bailey pair \eqref{C(7)}, we obtain the following Bailey pair:
\begin{equation}
    \begin{split}
        &\alpha_{2n}'(q^2;q)=(-1)^n\frac{1-q^{4n+2}}{1-q^2}q^{n^2-n}, \quad \alpha_{2n+1}'(q^2;q)=0, \\
        &\beta_n'(q^2;q)=\frac{q^{\frac{1}{2}(n^2-n)}}{(q;q)_n(q^3;q^2)_n}.
    \end{split} \tag{C7*} \label{C7*}
\end{equation}
Similarly, we can also obtain the Bailey pair \eqref{B-G(1.1)} from \eqref{G(3)} using \eqref{eq-b-infinity-raise}.

When $b=0$, \eqref{eq-raise-power} becomes
\begin{equation}\label{eq-b-0-raise}
\begin{split}
&\alpha_n'(aq;q)=\frac{(1-aq^{2n+1})a^nq^{n^2}}{1-aq}\sum_{r=0}^na^{-r}q^{-r^2}\alpha_r, \\
&\beta_n'(aq;q)=\beta_n(a;q).
\end{split}
\end{equation}
Applying \eqref{eq-b-0-raise} to the Bailey pair \eqref{G(4.1)}, we obtain the following Bailey pair:
\begin{equation}
    \begin{split}
    &\alpha_n'(q;q)=(-1)^n\frac{1-q^{2n+1}}{1-q}q^{\frac{1}{4}n^2-\frac{3}{4}n},\\
    &\beta_n'(q;q)=\frac{(-1)^nq^{\frac{1}{2}n^2-n}}{(-q^{1/2};q)_n(q^2;q^2)_n}.
\end{split}\label{G4.1-new}
\end{equation}

For Example 1 we need the following Bailey pair which appears to be new.
\begin{lemma}\label{lem-exam1-new-BP}
The following $(\alpha_n(1;q),\beta_n(1;q))$ is a Bailey pair related to $1$:
\begin{equation}\label{eq-exam1-key-BP}
\begin{split}
&\alpha_0(1;q)=1, ~ \alpha_{2n}(1;q)=(-1)^nq^{3n^2}(q^{3n}+q^{-3n}), ~ \alpha_{2n+1}(1;q)=0, \\
&\beta_n(1;q)=\frac{q^n(1+q^{-1})-q^{2n-1}}{(q;q)_n(q;q^2)_n}.
\end{split}
\end{equation}
\end{lemma}
\begin{proof}
Applying \eqref{eq-b-0-raise} to the \eqref{C(1)} Bailey pair, after simplifications, we obtain a new Bailey pair: for $n\geq 0$,
\begin{equation}\label{eq-exam1-BP-iterate-1}
\begin{split}
&\alpha_{2n}^{(1)}(q;q)=(-1)^n\frac{1-q^{4n+1}}{1-q}q^{3n^2-n}, \\
&\alpha_{2n+1}^{(1)}(q;q)=(-1)^n\frac{1-q^{4n+3}}{1-q}q^{3n^2+3n+1},   \\
&\beta_n^{(1)}(q;q)=\frac{1}{(q;q)_n(q;q^2)_n}.
\end{split}
\end{equation}
Applying \eqref{eq-b-infinity-reduce} to this Bailey pair, we obtain the Bailey pair:
\begin{equation}\label{eq-exam1-BP-iterate-2}
\begin{split}
&\alpha_0^{(2)}(1;q)=1, \\
&\alpha_{2n}^{(2)}(1;q)=(-1)^nq^{3n^2}(q^n+q^{-n}), \quad n\geq 1, \\
&\alpha_{2n+1}^{(2)}(1;q)=(-1)^n\Big(q^{3n^2+5n+2}-q^{3n^2+n}\Big), \quad n\geq 0, \\
&\beta_n^{(2)}(1;q)=\frac{q^n}{(q;q)_n(q;q^2)_n}, \quad n\geq 0.
\end{split}
\end{equation}
Next, applying \eqref{eq-b-0-raise} to \eqref{eq-exam1-BP-iterate-2}, we obtain the Bailey pair: for $n\geq 0$,
\begin{equation}\label{eq-exam1-BP-iterate-3}
\begin{split}
&\alpha_{2n}^{(3)}(q;q)=\frac{(1-q^{4n+1})q^{4n^2}}{1-q}\Big((-1)^nq^{-n(n+1)}(1+q)-(-1)^{n}q^{-n(n-1)+1}\Big), \\
&\alpha_{2n+1}^{(3)}(q;q)=\frac{(1-q^{4n+3})q^{(2n+1)^2}}{1-q}\Big((-1)^nq^{-n(n+1)}(1+q)-(-1)^nq^{-n^2-3n-1}\Big), \\
&\beta_{n}^{(3)}(q;q)=\frac{q^n}{(q;q)_n(q;q^2)_n}.
\end{split}
\end{equation}
Applying \eqref{eq-b-infinity-reduce} to \eqref{eq-exam1-BP-iterate-3}, after simplifications we obtain the  Bailey pair:
\begin{align}
&\alpha_0^{(4)}(1;q)=1, \nonumber  \\
&\alpha_{2n}^{(4)}(1;q)=(-1)^nq^{3n^2+n}-(-1)^nq^{3n^2+3n+1}+(-1)^nq^{3n^2+n+1}+(-1)^nq^{3n^2-n} \nonumber \\
&\qquad \qquad \qquad +(-1)^nq^{3n^2-n+1}-(-1)^nq^{3n^2-3n+1},  \\
&\alpha_{2n+1}^{(4)}(1;q)=(-1)^nq^{3n^2+5n+2}+(-1)^nq^{3n^2+5n+3} -(-1)^nq^{3n^2+n}-(-1)^nq^{3n^2+n+1}, \nonumber \\
&\beta_n^{(4)}(1;q)=\frac{q^{2n}}{(q;q)_n(q;q^2)_n}. \nonumber
\end{align}
Now we construct a Bailey pair using linear combinations:
\begin{align*}
(\alpha_n(1;q),\beta_n(1;q))=(1+q^{-1})(\alpha_n^{(2)}(1;q),\beta_n^{(2)}(1;q))-q^{-1}(\alpha_n^{(4)}(1;q),\beta_n^{(4)}(1;q)).
\end{align*}
It is easy to see that $\beta_n(1;q)$ has the desired form in \eqref{eq-exam1-key-BP}. After simplifications, we find that $\alpha_n(1;q)$ agrees with the expression in \eqref{eq-exam1-key-BP} as well.
\end{proof}

\begin{lemma} \label{lem-exam1-new-BP2}
The following $(\alpha_n(1;q),\beta_n(1;q))$ is a Bailey pair related to $1$:
\begin{equation}\label{eq-exam1-key-BP2}
    \begin{split}
    &\alpha_0(1;q)=1, ~~\alpha_n(1;q)=(-1)^nq^{\frac{3}{4}n^2-\frac{5}{4}n}(1+q^{\frac{5}{2}n}), \quad n\geq 1,\\
    &\beta_n(1;q)=\frac{-q^{-\frac{1}{2}}+q^{n}+q^{2n-\frac{1}{2}}}{(-q^{1/2};q)_n(q^2,q^2)_n}, \quad n\geq 0.
\end{split}
\end{equation}
\end{lemma}
\begin{proof}
    Applying \eqref{eq-b-infinity-reduce} to \eqref{B-G(1.1)}, we obtain the Bailey pair:
\begin{equation}\label{eq-exam1-BP2-part-2}
    \begin{split}
    &\alpha_0^{(1)}(1;q)=1, \\
    &\alpha_n^{(1)}(1;q)=(-1)^nq^{\frac{3}{4}n^2-\frac{3}{4}n}(1+q^{\frac{3}{2}n}), \quad n\geq 1, \\
    &\beta_n^{(1)}(1;q)=\frac{q^{n}}{(-q^{1/2};q)_n(q^2,q^2)_n}, \quad n\geq 0.
\end{split}
\end{equation}

    Applying \eqref{eq-b-0-raise} to \eqref{eq-exam1-BP2-part-2}, we obtain the Bailey pair:
    \begin{equation}\label{eq-exam1-BP2-3}
    \begin{split}
    &\alpha_n^{(2)}(q;q)=\frac{1-q^{2n+1}}{1-q}(-1)^{n}q^{\frac{3}{4}n^2-\frac{3}{4}n}(1-q^{\frac{n+1}{2}}+q^{n+\frac{1}{2}}), \quad n\geq 0,\\
    &\beta_n^{(2)}(q;q)=\frac{q^{n}}{(-q^{1/2};q)_n(q^2,q^2)_n}, \quad n\geq 0.
\end{split}
\end{equation}
Applying \eqref{eq-b-infinity-reduce} to \eqref{eq-exam1-BP2-3}, we obtain the Bailey pair:
    \begin{equation}\label{eq-exam1-BP2-part3}
    \begin{split}
    &\alpha_0^{(3)}(1;q)=1, \\
    &\alpha_n^{(3)}(1;q)=(-1)^nq^{\frac{3}{4}n^2-\frac{5}{4}n+\frac{1}{2}}
    -(-1)^nq^{\frac{3}{4}n^2-\frac{3}{4}n+\frac{1}{2}}
    +(-1)^nq^{\frac{3}{4}n^2-\frac{1}{4}n}\\
    & \qquad +(-1)^nq^{\frac{3}{4}n^2+\frac{1}{4}n}
    -(-1)^nq^{\frac{3}{4}n^2+\frac{3}{4}n+\frac{1}{2}}
    +(-1)^nq^{\frac{3}{4}n^2+\frac{5}{4}n+\frac{1}{2}}, \quad n\geq 1,\\
    &\beta_n^{(3)}(1;q)=\frac{q^{2n}}{(q^2,q^2)_n(-q^{1/2};q)_n}, \quad n\geq 0.
\end{split}
\end{equation}
Now we obtain the desired Bailey pair $(\alpha_n(1;q),\beta_n(1;q))$ \eqref{eq-exam1-key-BP2} by
\begin{align*}
-q^{-\frac{1}{2}}(\alpha_n^{(0)}(1;q),\beta_n^{(0)}(1;q))+(\alpha_n^{(1)}(1;q),\beta_n^{(1)}(1;q))+q^{-\frac{1}{2}}(\alpha_n^{(3)}(1;q),\beta_n^{(3)}(1;q)),
\end{align*}
where $(\alpha_n^{(0)}(1;q),\beta_n^{(0)}(1;q))$ is the Bailey pair \eqref{B-G(1)}.
\end{proof}

\begin{lemma}\label{lem-exam1-new-BP3}
    The following $(\alpha_n(q;q),\beta_n(a;q))$ is a Bailey pair related to $q$:
    \begin{equation}\label{eq-exam1-key-BP3}
    \begin{split}
    &\alpha_{2n}(q;q)=(-1)^nq^{3n^2-n}, \quad \alpha_{2n+1}(q;q)=(-1)^{n+1}q^{3(n+1)^2+(n+1)},\\
    &\beta_n(q;q)=\frac{1+q^{n+1}-q^{2n+1}}{(q,q)_n(q^3;q^2)_n}.
\end{split}
\end{equation}
\end{lemma}
\begin{proof}
    Applying \eqref{eq-b-infinity-reduce} to \eqref{C(4)}, we obtain the Bailey pair:
    \begin{align}
    &\alpha_{2n}^{(1)}(1;q)=\frac{(-1)^n(1-q)}{(1-q^{4n+1})(1-q^{4n-1})}
    (-q^{3n^2-n-1}+q^{3n^2+3n}+q^{3n^2+5n}-q^{3n^2+9n-1}), \nonumber
    \\&\alpha^{(1)}_{2n+1}(1;q)=\frac{(-1)^n(1-q^2)}{(1-q^{4n+3})(1-q^{4n+1})}(-q^{3n^2+5n}+q^{3n^2+9n+2}), \label{eq-exam1-key-BP3-part3-1} \\
    &\beta_n^{(1)}(1;q)=\frac{q^{2n}}{(q,q)_n(q^3;q^2)_n}. \nonumber
\end{align}
    Applying \eqref{eq-b-0-raise} to \eqref{eq-exam1-key-BP3-part3-1}, we obtain the Bailey pair:
    \begin{equation}\label{eq-exam1-key-BP3-part3}
    \begin{split}
    &\alpha_{2n}^{(2)}(q;q)=-(-1)^nq^{3n^2-n-1}+(-1)^nq^{3n^2+n-1}+(-1)^nq^{3n^2+3n},
    \\&\alpha_{2n+1}^{(2)}(q;q)=-(-1)^{n}q^{3n^2+3n}-(-1)^nq^{3n^2+5n+1}+(-1)^nq^{3n^2+7n+3},\\
    &\beta_n^{(2)}(q;q)=\frac{q^{2n}}{(q,q)_n(q^3;q^2)_n}.
\end{split}
\end{equation}
Now we obtain the desired Bailey pair $(\alpha_n(q;q),\beta_n(q;q))$ \eqref{eq-exam1-key-BP3} by
\begin{align*}
(\alpha_n^{(3)}(q;q),\beta_n^{(3)}(q;q))+q(\alpha_n^{(0)}(q;q),\beta_n^{(0)}(q;q))-q(\alpha_n^{(2)}(1;q),\beta_n^{(2)}(1;q)),
\end{align*}
where $(\alpha_n^{(3)}(q;q),\beta_n^{(3)}(q;q))$ is the Bailey pair \eqref{C(3)} and
$(\alpha_n^{(0)}(q;q),\beta_n^{(0)}(q;q))$ is the Bailey pair \eqref{C(4)}.
\end{proof}

\begin{lemma}\label{lem-exam1-new-BP-add}
The following $(\alpha_n(q;q),\beta_n(q;q))$ is a Bailey pair related to $q$: for $n\geq 0$,
\begin{equation}\label{eq-exam1-key-BP-add}
\begin{split}
&\alpha_n(q;q)=(-1)^nq^{\frac{3}{4}(n^2-n)}\frac{1-q^{3n+\frac{3}{2}}}{1-q^{\frac{1}{2}}}, \\
&\beta_n(q;q)=\frac{q^n(1+q+q^{n+\frac{1}{2}})}{(-q^{3/2};q)_{n}(q^2;q^2)_n}.
\end{split}
\end{equation}
\end{lemma}
It is possible to prove this lemma using the method in the proofs of Lemmas \ref{lem-exam1-new-BP}--\ref{lem-exam1-new-BP3} but the calculations is a bit messy. Here we present a different approach which is also applicable to prove Lemmas \ref{lem-exam1-new-BP}--\ref{lem-exam1-new-BP3}.
\begin{proof}
By the definition in \eqref{defn-BP}, it suffices to show that
\begin{align}\label{add-LR-relation}
L_n(q):=\frac{q^n(1+q+q^{n+\frac{1}{2}})}{(-q^{1/2};q)_{n+1}(q^2;q^2)_n}=\sum_{k=0}^n\frac{(-1)^kq^{\frac{3}{4}(k^2-k)}(1-q^{3k+\frac{3}{2}})}{(q;q)_{n-k}(q;q)_{n+k+1}}=:R_n(q).
\end{align}
On the one hand, using the Mathematica package \texttt{q-Zeil} \cite{q-Zeil} we find that
\begin{align}
R_n(q^2)=\frac{q^2(1+q^2+q^{2n+1})}{(1+q^2+q^{2n-1})(1-q^{4n})(1+q^{2n+1})}R_{n-1}(q^2).
\end{align}
On the other hand, by definition it is easy to see that $L_n(q^2)$ satisfies the same recurrence relation.  Since
\begin{align*}
L_0(q^2)=R_0(q^2)=\frac{1+q+q^2}{1+q},
\end{align*}
we deduce that $L_n(q^2)=R_n(q^2)$ for any $n\geq 0$. Replacing $q^2$ by $q$ we obtain \eqref{add-LR-relation}.
\end{proof}

As consequences of the new Bailey pairs in Lemmas \ref{lem-exam1-new-BP}--\ref{lem-exam1-new-BP-add}, we find the following new identities.
\begin{corollary}
We have
\begin{align}
\sum_{n=0}^\infty \frac{q^{n^2+n-1}(1+q-q^n)}{(q;q)_n(q;q^2)_n}&=\frac{(q^4,q^{10},q^{14};q^{14})_\infty}{(q;q)_\infty}, \label{id-mod14-1}\\
\sum_{n=0}^\infty \frac{q^{2n^2}(-1+q^{2n+1}+q^{4n})}{(-q;q^2)_n(q^4;q^4)_n}&=q\frac{(q,q^{6},q^{7};q^{7})_\infty}{(q^2;q^2)_\infty}, \label{id-mod7} \\
\sum_{n=0}^\infty \frac{q^{n^2+n}(1+q^{n+1}-q^{2n+1})}{(q;q)_n(q;q^2)_{n+1}}&=\frac{(q^6,q^{8},q^{14};q^{14})_\infty}{(q;q)_\infty}, \label{id-mod14-2} \\
\sum_{n=0}^\infty \frac{q^{2n^2+4n}(1+q^2+q^{2n+1})}{(-q;q^2)_{n+1}(q^4;q^4)_n}&=\frac{(q^3,q^4,q^7;q^7)_\infty}{(q^2;q^2)_\infty}. \label{id-mod7-add}
\end{align}
\end{corollary}
\begin{proof}
Substituting the Bailey pairs \eqref{eq-exam1-key-BP}, \eqref{eq-exam1-key-BP2}, \eqref{eq-exam1-key-BP3} and \eqref{eq-exam1-key-BP-add} into Lemma \ref{lem-BP-sum} and using \eqref{Jacobi}, we obtain the desired identities.
\end{proof}
The identities \eqref{id-mod14-1} and \eqref{id-mod14-2} give new companions to three mod 14 identities in Rogers' work \cite{Rogers1894} (see \cite[Eqs.\ (2.14.1)--(2.14.3)]{MSZ}. The identities \eqref{id-mod7} and \eqref{id-mod7-add} gives new companions to the the Rogers--Selberg mod 7 identities (see \cite{Rogers1894,Rogers1917} or \cite[Eqs.\ (2.7.1)--(2.7.3)]{MSZ}).

For the two special cases in Example 5 we need the following new Bailey pair.
\begin{lemma}\label{lem-new-BP}
The following $(\alpha_n(1;q),\beta_n(1;q))$ is a Bailey pair relative to 1:
    \begin{align}
        &\alpha_0(1;q)=1, \quad
        \alpha_n(1;q)=\sum_{i=-n}^{n-1}(-1)^{i+n}q^{i^2}, \\
        &\beta_n(1;q)=\frac{(-1)^n}{(q^2;q^2)_n} \sum_{k=0}^n (-1)^k \frac{(-q;q)_k}{(q;q)_k}.
    \end{align}
\end{lemma}
\begin{proof}
By the definition in \eqref{defn-BP}, it suffices to show that
\begin{align}
\beta_n(1;q)=\sum_{k=1}^n \frac{1}{(q;q)_{n-k}(q;q)_{n+k}} \sum_{i=-k}^{k-1} (-1)^{i+k}q^{i^2} +\frac{1}{(q;q)_n^2}.
\end{align}
This is equivalent to
\begin{align}\label{LR-equal}
    L_n(q)=R_n(q), \quad n\geq 1,
\end{align}
where
\begin{align}
L_n(q)&:=\sum_{k=0}^{n-1} (-1)^k \frac{(-q;q)_k}{(q;q)_k}, \\
R_n(q)&:=(-1)^n (q^2;q^2)_n \sum_{k=1}^n \frac{1}{(q;q)_{n-k}(q;q)_{n+k}}\sum_{i=-k}^{k-1} (-1)^{i+k}q^{i^2}.
\end{align}
For convenience, we agree that $L_0(q)=R_0(q)=0$.
By definition we have
\begin{align}\label{LR-initial}
L_1(q)=R_1(q)=1, \quad L_2(q)=R_2(q)=-\frac{2q}{1-q}.
\end{align}
Clearly we have
\begin{align}\label{L-rec}
   \widetilde{L}_n(q):=L_n(q)-L_{n-1}(q)=(-1)^{n-1} \frac{(-q;q)_{n-1}}{(q;q)_{n-1}}.
\end{align}
Now we are going to show that $R_n(q)$ satisfies the same recurrence relation.

For convenience we denote
\begin{align}
    S_k(q):=\sum_{i=-k}^{k-1} (-1)^{i+k}q^{i^2}, \quad k\geq 1.
\end{align}
It is easy to see that $S_1(q)=q-1$. Moreover,
we agree that $S_0(q)=0$. It is straightforward to verify that
\begin{align}
    S_k(q)+S_{k-1}(q)=q^{k^2}-q^{(k-1)^2}, \quad k\geq 1.
\end{align}
We have
\begin{align}
&\widetilde{R}_n(q):=R_n(q)-R_{n-1}(q)\nonumber \\
&=(-1)^{n}(q^2;q^2)_{n-1} \left((1-q^{2n})\sum_{k=1}^n \frac{S_k(q)}{(q;q)_{n-k}(q;q)_{n+k}}+\sum_{k=1}^{n-1} \frac{S_k(q)}{(q;q)_{n-1-k}(q;q)_{n-1+k}}  \right) \nonumber \\
&=(-1)^n(q^2;q^2)_{n-1}\sum_{k=1}^n \frac{S_k(q)}{(q;q)_{n-k}(q;q)_{n+k}}\Big(1-q^{2n}+(1-q^{n-k})(1-q^{n+k})\Big) \nonumber \\
&=(-1)^n(q^2;q^2)_{n-1}\sum_{k=1}^n \frac{S_k(q)}{(q;q)_{n-k}(q;q)_{n+k}}(2-q^{n-k}-q^{n+k}) \nonumber \\
&=(-1)^n(q^2;q^2)_{n-1}\Big(\sum_{k=1}^n \frac{S_k(q)}{(q;q)_{n-k}(q;q)_{n+k}}(1-q^{n-k}) \nonumber \\
&\qquad \qquad \qquad \qquad \qquad +\sum_{k=1}^n \frac{S_k(q)}{(q;q)_{n-k}(q;q)_{n+k}}(1-q^{n+k})\Big) \nonumber \\
&=(-1)^n(q^2;q^2)_{n-1}\Big(\sum_{k=1}^n \frac{S_k(q)}{(q;q)_{n-k-1}(q;q)_{n+k}}+\sum_{k=1}^n \frac{S_k(q)}{(q;q)_{n-k}(q;q)_{n+k-1}}\Big) \nonumber \\
&=(-1)^n(q^2;q^2)_{n-1}\Big(\sum_{k=2}^n \frac{S_{k-1}(q)}{(q;q)_{n-k}(q;q)_{n+k-1}} +\sum_{k=1}^n \frac{S_k(q)}{(q;q)_{n-k}(q;q)_{n+k-1}}\Big) \nonumber \\
&=(-1)^n(q^2;q^2)_{n-1}\sum_{k=1}^n \frac{q^{k^2}-q^{(k-1)^2}}{(q;q)_{n-k}(q;q)_{n+k-1}}=A_n(q)-B_n(q), \label{wR-AB}
\end{align}
where
\begin{align}
A_n(q)&:=(-1)^n(q^2;q^2)_{n-1}\sum_{k=1}^n \frac{q^{k^2}}{(q;q)_{n-k}(q;q)_{n+k-1}}, \\
B_n(q)&:=(-1)^n(q^2;q^2)_{n-1}\sum_{k=1}^n \frac{q^{(k-1)^2}}{(q;q)_{n-k}(q;q)_{n+k-1}} \nonumber \\
&=(-1)^n(q^2;q^2)_{n-1}\sum_{k=0}^{n-1} \frac{q^{k^2}}{(q;q)_{n-k-1}(q;q)_{n+k}}.
\end{align}
Using the Mathematica package \texttt{q-Zeil} \cite{q-Zeil}, we find that
\begin{align}
A_n(q)&=(-1)^{n+1}\frac{q^{2n-2}(1-q)(-q;q)_{n-1}}{(1+q^{n-1})(1-q^{2n-1})(q;q)_{n-1}}+\frac{1+q+q^{2n-3}+q^{2n-1}}{1-q^{2n-1}}A_{n-1}(q) \nonumber \\
&\qquad \qquad -\frac{q(1-q^{2n-4})}{1-q^{2n-1}}A_{n-2}(q), \label{A-rec} \\
B_n(q)&=(-1)^n\frac{q^{2n-2}(1-q)(-q;q)_{n-1}}{(1+q^{n-1})(1-q^{2n-1})(q;q)_{n-1}}+\frac{1+q+q^{2n-3}+q^{2n-1}}{1-q^{2n-1}}B_{n-1}(q) \nonumber \\
&\qquad \qquad -\frac{q(1-q^{2n-4})}{1-q^{2n-1}}B_{n-2}(q). \label{B-rec}
\end{align}
In view of \eqref{wR-AB} and combining \eqref{A-rec} with \eqref{B-rec}, we deduce that
\begin{align}\label{wR-rec}
\widetilde{R}_n(q)&=2(-1)^{n+1}\frac{q^{2n-2}(1-q)(-q;q)_{n-1}}{(1+q^{n-1})(1-q^{2n-1})(q;q)_{n-1}}
+\frac{1+q+q^{2n-3}+q^{2n-1}}{1-q^{2n-1}}\widetilde{R}_{n-1}(q)\nonumber \\
&\qquad -q\frac{1-q^{2n-4}}{1-q^{2n-1}}\widetilde{R}_{n-2}(q).
\end{align}

We claim that $\widetilde{L}_n(q)$ also satisfies the recurrence relation \eqref{wR-rec}. In fact, using \eqref{L-rec}, this claim is equivalent to
\begin{align}
\frac{(-q;q)_{n-1}}{(q;q)_{n-1}}&=2\frac{(1-q)q^{2n-2}(-q;q)_{n-1}}{(1+q^{n-1})(1-q^{2n-1})(q;q)_{n-1}}  \\
&\qquad +\frac{1+q+q^{2n-3}+q^{2n-1}}{1-q^{2n-1}} \cdot \frac{(-q;q)_{n-2}}{(q;q)_{n-2}}-q\frac{1-q^{2n-4}}{1-q^{2n-1}}\cdot \frac{(-q;q)_{n-3}}{(q;q)_{n-3}}. \nonumber
\end{align}
Clearing the denominators on both sides, this amounts to show that
\begin{align}
    (1+q^{n-1})(1-q^{2n-1})&=2(1-q)q^{2n-2}+(1+q+q^{2n-3}+q^{2n-1})(1-q^{n-1}) \nonumber \\
    &\quad -q(1-q^{n-2})^2(1-q^{n-1}).
\end{align}
which is clearly true by direct calculations. This proves the claim on $\widetilde{L}_n(q)$.

Note that
\begin{align}
\widetilde{L}_1(q)=\widetilde{R}_1(q)=1, \quad \widetilde{L}_2(q)=\widetilde{R}_2(q)=-\frac{1+q}{1-q}.
\end{align}
Since $\widetilde{L}_n(q)$ and $\widetilde{R}_n(q)$ share the same recurrence relation stated in \eqref{wR-rec}, we conclude that $\widetilde{L}_n(q)=\widetilde{R}_n(q)$ for any $n\geq 1$. This then implies that $L_n(q)$ and $R_n(q)$ satisfy the same recurrence relation  stated in \eqref{L-rec}. In view of \eqref{LR-initial}, we conclude that $L_n(q)=R_n(q)$ for any $n\geq 1$. This proves \eqref{LR-equal} and hence the desired assertion.
\end{proof}

\subsection{Basic facts about modular forms}
In order to prove the modular transformation formulas stated in Section \ref{sec-intro} and discuss modularity of Nahm sums, we need some knowledge from the theory of modular forms.

Recall that $q=e^{2\pi i\tau}$ ($\mathrm{Im}~\tau>0$). We define the Dedekind eta function
\begin{align}
\eta(\tau):=q^{1/24}(q;q)_\infty
\end{align}
and Weber's modular functions \cite{We}:
\begin{align}\label{Weber-defn}
\mathfrak{f}(\tau):=q^{-1/48} (-q^{1/2};q)_\infty, \ \  \mathfrak{f}_1(\tau):=q^{-1/48} (q^{1/2};q)_\infty, \ \ \mathfrak{f}_2(\tau):=q^{1/24}(-q;q)_\infty.
\end{align}
The following properties are well-known:
\begin{align}
&\eta(-1/\tau)=\sqrt{-i \tau} \eta(\tau), \ \ \eta(\tau+1)=e^{\pi i  /12} \eta(\tau), \label{eta-tran} \\
&\mathfrak{f}(-1/\tau) = \mathfrak{f}(\tau), \ \ \ \mathfrak{f}_2(-1/\tau) = \frac{1}{\sqrt{2}}  \mathfrak{f}_1(\tau), \ \ \mathfrak{f}_1(-1/\tau)=\sqrt{2} \mathfrak{f}_2(\tau), \label{Weber-1} \\
&\mathfrak{f}(\tau+1) =e^{-\pi i /24}  \mathfrak{f}_1(\tau), \ \ \mathfrak{f}_1(\tau+1) = e^{-\pi i/24} \mathfrak{f}(\tau),  \ \ \mathfrak{f}_2(\tau+1) = e^{\pi i/12} \mathfrak{f}_2(\tau). \label{Weber-2}
\end{align}
For $m\in \mathbb{Q}_{>0}$ and $j\in \mathbb{Q}$, we define the theta series \cite[p.\ 215]{Wakimoto}
\begin{align}\label{hg-defn}
    h_{j,m}(\tau):=\sum_{k\in \mathbb{Z}} q^{m(k+\frac{j}{2m})^2}, \quad
    g_{j,m}(\tau):=\sum_{k\in \mathbb{Z}} (-1)^kq^{m(k+\frac{j}{2m})^2}.
\end{align}
By \eqref{Jacobi} we have the product representations:
\begin{align}
h_{j,m}(\tau)&=q^{\frac{j^2}{4m}}(-q^{m-j},-q^{m+j},q^{2m};q^{2m})_\infty,  \label{add-h-product}\\ g_{j,m}(\tau)&=q^{\frac{j^2}{4m}}(q^{m+j},q^{m-j},q^{2m};q^{2m})_\infty.  \label{add-g-product}
\end{align}

It is easy to verify the following properties \cite[p.\ 215]{Wakimoto}:
\begin{align}
&h_{j,m}(\tau)=h_{-j,m}(\tau)=h_{2m+j,m}(\tau), \quad g_{j,m}(\tau)=g_{-j,m}(\tau)=-g_{2m+j,m}(\tau), \label{g-h-period}\\
&h_{j,m}(\tau)=h_{2j,4m}(\tau)+h_{4m-2j,4m}(\tau), \label{h-h-change} \\
&g_{j,m}(\tau)=h_{2j,4m}(\tau)-h_{4m-2j,4m}(\tau), \label{g-h-change}\\
&h_{j,m}(2\tau)=h_{2j,2m}(\tau), \quad g_{j,m}(2\tau)=g_{2j,2m}(\tau). \label{hg-double}
\end{align}

\begin{lemma}\label{lem-modular}
(Cf. \cite[p.\ 215, Theorem 4.5]{Wakimoto}.) For $j\in \mathbb{Z}$ and $m\in \frac{1}{2}\mathbb{N}$ we have
\begin{align}
h_{j,m}\left(-\frac{1}{\tau}\right)&=\frac{(-i\tau)^{\frac{1}{2}}}{\sqrt{2m}} \sum_{0\leq k\leq 2m-1} e^{\frac{\pi ijk}{m}}h_{k,m}(\tau), \label{h-modualr} \\
g_{j,m}\left(-\frac{1}{\tau}\right)&=\frac{(-i\tau)^{\frac{1}{2}}}{\sqrt{2m}}\sum_{\begin{smallmatrix}
    0\leq k\leq 4m-1 \\ k ~~\text{odd}
\end{smallmatrix}} e^{\frac{\pi ijk}{2m}} h_{\frac{k}{2},m}(\tau). \label{g-modular}
\end{align}
For $j+m\in \mathbb{Z}$ we have
\begin{align}\label{fg-translation}
h_{j,m}(\tau+1)=e^{\frac{\pi ij^2}{2m}}h_{j,m}(\tau), \quad g_{j,m}(\tau+1)=e^{\frac{\pi ij^2}{2m}}g_{j,m}(\tau).
\end{align}
\end{lemma}

For any odd integer $m$ we define
\begin{align}
\epsilon_m:=\left\{\begin{array}{ll}
1 & \text{if $m\equiv 1\pmod{4}$} \\
i & \text{if $m\equiv 3 \pmod{4}$}
\end{array}\right..
\end{align}
We recall a well-known formula for the quadratic Gauss sum. For $m$ odd and $\gcd(a,m)=1$ we have (see e.g.\ \cite[p.\ 26, Theorem 1.5.2]{BEW-book})
\begin{align}\label{eq-Gauss-original}
\sum_{n=0}^{m-1} e^{2\pi i \frac{an^2}{m}}=\epsilon_m\sqrt{m}\left(\frac{a}{m}\right)
\end{align}
where $\left(\frac{a}{m}\right)$ is the Jacobi symbol. After completing the square, it is easy to deduce from \eqref{eq-Gauss-original} that
\begin{align}\label{Gauss-sum}
\sum_{n=0}^{m-1} e^{2\pi i \frac{an^2+bn}{m}}=\epsilon_m \sqrt{m}\left(\frac{a}{m}\right)e^{-2\pi i \frac{\psi(a)b^2}{m}},
\end{align}
where $\psi(a)$ is some number with $4\psi(a) a\equiv 1$ (mod $m$).

We will need the function
\begin{align}\label{eq-Ta}
T(a,m)&:=\sum_{\substack{0\leq k\leq 2m-1 \\ k~~\text{odd}}} e^{\frac{\pi i (k^2+ak)}{2m}}.
\end{align}
Sometimes we omit $m$ and write $T(a,m)$ as $T(a)$. By definition it is easy to see that
\begin{align}\label{T-period}
T(a+2m,m)=-T(a,m).
\end{align}
\begin{lemma}\label{lem-sum}
We have
\begin{align}\label{exp-reduce}
\sum_{\substack{0\leq k \leq 4m-1 \\ k~~\text{odd}}} e^{\frac{\pi ik(k+a)}{2m}}
=\left\{\begin{array}{ll}
2T(a,m), & \text{$a$ even}, \\
0 &\text{$a$ odd}. \end{array}\right.
\end{align}
If $a$ is even and $m$ is odd, then
\begin{align}\label{Tam-result}
T(a,m)=\epsilon_m\sqrt{m}e^{\frac{\pi i(1-a-(a-2)^2\delta_m)}{2m}}
\end{align}
where  $\delta_m$ is any integer satisfying $4\delta_m\equiv 1 \!\pmod{m}$.
\end{lemma}
\begin{proof}
We have
\begin{align}
&\sum_{\substack{0\leq k \leq 4m-1 \\ k~~\text{odd}}} e^{\frac{\pi ik(k+a)}{2m}}=\sum_{\substack{0\leq k \leq 2m-1 \\ k~~\text{odd}}} \Big(e^{\frac{\pi ik(k+a)}{2m}}+e^{\frac{\pi i(k+2m)(k+2m+a)}{2m}}\Big) \nonumber \\
&=(1+e^{\pi ia})\sum_{\substack{0\leq k\leq 2m-1 \\ k~~\text{odd}}} e^{\frac{\pi i (k^2+ak)}{2m}}.
\end{align}
This proves \eqref{exp-reduce}.

When $a$ is even and $m$ is odd, we write $a=2a_0$ with $a_0\in \mathbb{Z}$. We have
\begin{align}\label{Ta-proof}
&T(a)=\sum_{k=1}^m e^{\frac{\pi i ((2k-1)^2+a(2k-1))}{2m}}=e^{\frac{\pi i(1-a)}{2m}}\sum_{k=1}^m e^{\frac{\pi i (2k^2+(a-2)k)}{m}} \nonumber \\
&=e^{\frac{\pi i(1-a)}{2m}}\sum_{k=1}^m e^{\frac{2\pi i (k^2+(a_0-1)k)}{m}}.
\end{align}
Substituting \eqref{Gauss-sum} into \eqref{Ta-proof} we obtain \eqref{Tam-result}.
\end{proof}

We establish the following new lemma which will be invoked in the proof of Theorem \ref{thm-Mizuno-conj}.
\begin{lemma}\label{lem-add-g-tran}
Let $1\leq j \leq m$ be an odd integer. We have
\begin{align}\label{add-g-tran}
&g_{j,m}\left(-\frac{\tau+1}{4\tau}\right)\nonumber \\
&=\sqrt{\frac{-\tau}{m}} \epsilon_m \sum_{\substack{1\leq \ell \leq m-1 \\ \ell~ \text{odd}}} \left(e^{\pi i \frac{1-(j+\ell)-(j+\ell-2)^2\delta_m}{2m}}+e^{\pi i \frac{1-(j-\ell)-(j-\ell-2)^2\delta_m}{2m}} \right)
g_{\ell,m}\left(\frac{\tau+1}{4}\right).
\end{align}
\end{lemma}
\begin{proof}
From \eqref{g-modular} we have
\begin{align}
g_{j,m}\left(-\frac{\tau+1}{4\tau}\right)=\frac{1}{\sqrt{2m}}\left(\frac{-4i\tau}{\tau+1}\right)^{\frac{1}{2}}\sum_{\substack{0\leq k\leq 4m-1 \\ k~~\text{odd}}} e^{\frac{\pi ijk}{2m}}h_{\frac{k}{2},m}\left(\frac{4\tau}{\tau+1}\right). \label{add-g-tran-start}
\end{align}
Using \eqref{fg-translation} and \eqref{h-modualr} we deduce that
\begin{align}
&h_{\frac{k}{2},m}\left(\frac{4\tau}{\tau+1}\right)=h_{2k,4m}\left(\frac{\tau}{\tau+1}\right)=h_{2k,4m}\left(1-\frac{1}{\tau+1}\right) \nonumber \\
&=e^{\frac{\pi ik^2}{2m}}h_{2k,4m}\left(-\frac{1}{\tau+1}\right)
=e^{\frac{\pi ik^2}{2m}} \times \frac{(-i(\tau+1))^{\frac{1}{2}}}{\sqrt{8m}}\sum_{0\leq \ell \leq 8m-1} e^{\frac{\pi i\ell k}{2m}}h_{\ell,4m}(\tau+1). \label{add-h-tran-start}
\end{align}
Substituting \eqref{add-h-tran-start} into \eqref{add-g-tran-start}, we deduce that
\begin{align}
&g_{j,m}\left(-\frac{\tau+1}{4\tau}\right)=\frac{\sqrt{-\tau}}{2m} \sum_{\substack{0\leq \ell \leq 8m-1}} h_{\ell,4m}(\tau+1) \sum_{\substack{0\leq k\leq 4m-1 \\ k~~\text{odd}}} e^{\frac{\pi ik(k+j+\ell )}{2m}} \nonumber \\
&=\frac{\sqrt{-\tau}}{m} \sum_{\substack{0\leq \ell \leq 8m-1 \\ \ell ~\text{odd}}} h_{\ell,4m}(\tau+1) T(j+\ell) \quad \text{(by \eqref{exp-reduce})} \nonumber \\
&=\frac{\sqrt{-\tau}}{m}
 \sum_{\substack{0\leq \ell \leq 4m-1 \\ \ell ~\text{odd}}} \Big(T(j+\ell)h_{\ell,4m}(\tau+1)+T(j+8m-\ell)h_{8m-\ell,4m}(\tau+1) \Big) \nonumber \\
 &=\frac{\sqrt{-\tau}}{m}  \sum_{\substack{0\leq \ell \leq 4m-1 \\ \ell ~\text{odd}}} \Big(T(j+\ell)+T(j-\ell)\Big)h_{\ell,4m}(\tau+1) \Big)
 \quad \text{(by \eqref{T-period} and \eqref{g-h-period})}\nonumber \\
&=\frac{\sqrt{-\tau}}{m}
 \sum_{\substack{1\leq \ell \leq 2m-1 \\ \ell ~\text{odd}}} \Big(\big(T(j+\ell)+T(j-\ell)\big)h_{\ell,4m}(\tau+1) \nonumber \\
 &\qquad \qquad \qquad +\big(T(j+4m-\ell)+T(j-4m+\ell)\big)h_{4m-\ell,4m}(\tau+1)\Big)  \nonumber \\
 &=\frac{\sqrt{-\tau}}{m}
 \sum_{\substack{1\leq \ell \leq 2m-1 \\ \ell ~\text{odd}}} \big(T(j+\ell)+T(j-\ell)\big)\Big(h_{\ell,4m}(\tau+1) +h_{4m-\ell,4m}(\tau+1)\Big) \nonumber \\
&=\frac{\sqrt{-\tau}}{m} \sum_{\substack{1\leq \ell \leq 2m-1 \\ \ell ~\text{odd}}} \big(T(j+\ell)+T(j-\ell)\big) h_{\ell/2,m}(\tau+1).
\end{align}
Here for the last two equalities we used \eqref{T-period} and \eqref{h-h-change}, respectively.  Continuing this process, we have
\begin{align}
&g_{j,m}\left(-\frac{\tau+1}{4\tau}\right)=\frac{\sqrt{-\tau}}{m}
 \Bigg(\sum_{\substack{1\leq \ell \leq m-1 \\ \ell ~\text{odd}}} \Big( \big(T(j+\ell)+T(j-\ell)\big) h_{\ell/2,m}(\tau+1) \nonumber \\
&\qquad \qquad \qquad + \big(T(j+2m-\ell)+T(j-2m+\ell)\big) h_{(2m-\ell)/2,m}(\tau+1)\Big) \nonumber \\
&\qquad \qquad \qquad +\big(T(j+m)+T(j-m)\big)h_{m/2,m}(\tau+1) \Bigg) \nonumber \\
&=\frac{\sqrt{-\tau}}{m}
 \sum_{\substack{1\leq \ell \leq m-1 \\ \ell ~\text{odd}}}  \big(T(j+\ell)+T(j-\ell)  \big) \Big(h_{\ell/2,m}(\tau+1) - h_{(2m-\ell)/2,m}(\tau+1)\Big) \nonumber \\
&=\frac{\sqrt{-\tau}}{m}  \sum_{\substack{1\leq \ell \leq m-1 \\ \ell ~\text{odd}}} \big(T(j+\ell)+T(j-\ell)  \big)   g_{\ell/4,m/4}(\tau+1) \nonumber \\
&=\frac{\sqrt{-\tau}}{m}
 \sum_{\substack{1\leq \ell \leq m-1 \\ \ell ~\text{odd}}}  \big(T(j+\ell)+T(j-\ell)  \big)   g_{\ell,m}\left(\frac{\tau+1}{4}\right). \label{add-g-tran-proof-end}
\end{align}
Here for the second equality we used \eqref{T-period} to conclude that
\begin{align}
T(j+m)+T(j-m)=0,
\end{align}
and for the last three equalities we used \eqref{T-period}, \eqref{g-h-change} and \eqref{hg-double}, respectively.
Now by \eqref{add-g-tran-proof-end} and \eqref{Tam-result} we obtain \eqref{add-g-tran}.
\end{proof}

Let $P_2(t):=\{t\}^2-\{t\}+\frac{1}{6}$ be the second periodic Bernoulli polynomial where $\{t\}=t-\lfloor t\rfloor$ is the fractional part of $t$. Let $g,\delta$ be positve integers with $0<g<\delta$. The generalized Dedekind eta function is defined as
\begin{align}
    \eta_{\delta;g}(\tau):=q^{\frac{\delta}{2}P_2(g/\delta)} \prod\limits_{\substack{m\equiv \pm g \!\!\!\! \pmod{\delta} \\ m\in \mathbb{Z}_{>0}}} (1-q^m).
\end{align}
Note that $\eta(\tau)$ and $\eta_{\delta;g}(\tau)$ are modular forms of weights $1/2$ and $0$, respectively.

A generalized Dedekind eta product of level $N$ has the form
\begin{align}
    f(\tau)=\prod\limits_{\begin{smallmatrix}
        \delta |N \\ 0<g<\delta
    \end{smallmatrix}} \eta_{\delta;g}^{r_{\delta,g}}(\tau)
    \quad \text{where} \quad r_{\delta,g}\in \left\{\begin{array}{ll} \frac{1}{2}\mathbb{Z} & \text{if $g=\delta/2$,} \\
    \mathbb{Z} & \text{otherwise.}
    \end{array}\right.
\end{align}
For convenience, we use the notation
\begin{align}
&J_m:=(q^m;q^m)_\infty=q^{-m/24}\eta(m\tau), \label{J-defn} \\
&J_{a,m}:=(q^a,q^{m-a},q^m;q^m)_\infty=q^{-\frac{m}{2}P_2(a/m)-\frac{m}{24}}\eta_{m;a}(\tau)\eta(m\tau). \label{Jam-defn}
\end{align}

As in \cite{Wang-rank2,Wang-rank3,WW-rank2,WW-I}, our strategy of justifying the modularity of a specific Nahm sum consists of two steps. First, we establish Rogers--Ramanujan type identities which express a Nahm sum ${f}_{A,b,0,d}(q)$ as a sum of infinite products:
\begin{align}
   {f}_{A,b,0,d}(q)=f_1(q)+\cdots+f_{m}(q),
\end{align}
where each $f_i(q)$ is a product/quotient expressed by $J_m$ and $J_{a,m}$. Second, from \eqref{J-defn} and \eqref{Jam-defn} it is easy to find a unique $c_i$ such that $q^{c_i}f_i(q)$ becomes a generalized Dedekind eta product ($i=1,2,\dots,m$). If $c_1=c_2=\cdots=c_m=c$ and all of $q^{c_i}f_i(q)$ are  modular functions, then we conclude that ${f}_{A,b,c,d}(q)=q^c{f}_{A,b,0,d}(q)$ is a modular function. We refer the reader to \cite[Sect.\ 2.2]{Wang-rank2}  for more details of the second step. Note that the values of $c_i$ can be computed using \eqref{J-defn} and \eqref{Jam-defn} directly, and in many cases they can also be computed by writing infinite products as theta series using \eqref{Jacobi} (see e.g.\ \eqref{hg-defn}--\eqref{add-g-product}). Moreover, these values can also be  easily computed with the aid of the Maple package \texttt{thetaids} developed by Frye and Garvan \cite{Frye-Garvan}. Furthermore, this package allows us to check whether $q^{c_i}f_i(q)$ is a modular function on some congruence subgroup or not, and it can be used to automatically prove some theta function identities.

We have to conduct the above two steps for each of Mizuno's rank three Nahm sums associated the quadruples in \cite[Table 3]{Mizuno}. Since the second step is routine and easily achieved using the method in \cite{Frye-Garvan}, we will omit the details and only focus on the first step.

\section{Mizuno's rank three Nahm sums of index $(1,2,2)$}\label{sec-exam}
In this section we set $d=(2,2,1)$ and $D=\diag(2,2,1)$. We will write the summation indices $(n_1,n_2,n_3)$ as $(k,j,i)$ or $(j,k,i)$ so that the Nahm sum is of index $(1,2,2)$.

\subsection{Examples 1, 7 and 8} We will split the Nahm sums in these examples into two parts according to the parity of one of the summation indices.

\subsubsection{Example 1}
This example corresponds to
\begin{align*}
A=\begin{pmatrix}
1 & 0 & 1\\
0 & 2 & 2\\
1/2&1 & 2
\end{pmatrix},\quad
AD=\begin{pmatrix}
2 & 0 & 1\\
0 & 4 & 2\\
1 & 2 & 2
\end{pmatrix},\quad
b\in  \bigg\{
\begin{pmatrix}
0 \\ 0 \\ 0
\end{pmatrix},
\begin{pmatrix}
1 \\ 0 \\ 0
\end{pmatrix},
\begin{pmatrix}
1 \\ 0 \\ 1
\end{pmatrix},
\begin{pmatrix}
1 \\ 2 \\ 2
\end{pmatrix}
\bigg\}.
\end{align*}
The Nahm sums involved here are special instances of the function $G(u,v,w)=G(u,v,w;q)$ defined in \eqref{G-defn}.
We split it into two parts according to the parity of $i$:
\begin{align}\label{G-split}
    G(u,v,w)=G_0(u,v,w)+G_1(u,v,w)
\end{align}
where $G_\sigma(u,v,w)=G_\sigma(u,v,w;q)$ ($\sigma=0,1$) was defined in \eqref{Gc-defn}.

\begin{theorem}\label{thm-G-parity}
We have
\begin{align}
&G_0(1,1,1;q)
=
\frac{(-q;q^2)_\infty(q^{20},q^{24},q^{44};q^{44})_\infty}{(q^2;q^2)_\infty}, \label{G-1-even}
\\
&G_1(1,1,1;q)=q\frac{(-q^2;q^2)_\infty(-q,q^{10},-q^{11};-q^{11})_\infty}{(q^2;q^2)_\infty}, \label{G-1-odd} \\
&G_0(1,1,q;q)=\frac{(-q^2;q^2)_\infty(-q^5,q^6,-q^{11};-q^{11})_\infty}{(q^2;q^2)_\infty}, \label{G-2-even} \\
&G_1(1,1,q;q)
=
q\frac{(-q;q^2)_\infty(q^{12},q^{32},q^{44};q^{44})_\infty}{(q^2;q^2)_\infty}, \label{G-2-odd}
\\
&G_0(q,1,q;q)=\frac{(-q^2;q^2)_\infty(q^4,-q^7,-q^{11};-q^{11})_\infty}{(q^2;q^2)_\infty}, \label{G-3-even} \\
&G_1(q,1,q;q)
=
q^2\frac{(-q;q^2)_\infty(q^{8},q^{36},q^{44};q^{44})_\infty}{(q^2;q^2)_\infty}, \label{G-3-odd}
\\
&G_0(q^2,q^2,q;q)=\frac{(-q^2;q^2)_\infty(q^2,-q^9,-q^{11};q^{11})_\infty}{(q^2;q^2)_\infty}, \label{G-4-even} \\
&G_1(q^2,q^2,q;q)=q^3\frac{(-q;q^2)_\infty(q^{4},q^{40},q^{44};q^{44})_\infty}{(q^2;q^2)_\infty}, \label{G-4-odd} \\
& G_0(q,1,1;q)+G_1(q^3,q^2,q;q)
=
\frac{(-q;q^2)_\infty(q^{16},q^{28},q^{44};q^{44})_\infty}{(q^2;q^2)_\infty}, \label{G-sum-1} \\
& G_1(q,1,1;q)+G_0(q^3,q^2,q;q)=\frac{(-q^2;q^2)_\infty}{(q^2;q^2)_\infty}(-q^3,q^8,-q^{11};-q^{11})_\infty. \label{G-sum-2}
\end{align}
\end{theorem}

\begin{proof}
We have
\begin{align}
&G_\sigma(u,v,w)=\sum_{i,j\ge 0}\frac{q^{(2i+\sigma)^2+2j^2+2(2i+\sigma)j}u^{2i+\sigma}v^j}{(q;q)_{2i+\sigma}(q^2;q^2)_j}
\sum_{k\ge 0}\frac{q^{k^2+(2i+\sigma)k}w^k}{(q^2;q^2)_k} \nonumber \\
&=\sum_{i,j\ge 0}\frac{q^{(2i+\sigma)^2+2j^2+2(2i+\sigma)j}u^{2i+\sigma}v^j(-wq^{2i+\sigma+1};q^2)_\infty}{(q;q)_{2i+\sigma}(q^2;q^2)_j}\nonumber \\
&=(-wq^{\sigma+1};q^2)_\infty\sum_{i,j\ge 0}\frac{q^{(2i+\sigma)^2+2j^2+2(2i+\sigma)j}u^{2i+\sigma}v^j}{(q;q)_{2i+\sigma}(q^2;q^2)_j(-wq^{1+\sigma};q^2)_i}.
\label{Ftable3-1}
\end{align}

(1) Setting $(\sigma,u,v,w)=(0,1,1,1)$ in \eqref{Ftable3-1}, we have
\begin{align}\label{proof-1-1-even}
&G_0(1,1,1)=
(-q;q^2)_\infty\sum_{i,j\ge 0}\frac{q^{4i^2+4ij+2j^2}}{(q;q)_{2i}(q^2;q^2)_j(-q;q^2)_i} \nonumber  \\
&=(-q;q^2)_\infty\sum_{n=0}^\infty q^{2n^2}\sum_{i=0}^n\frac{q^{2i^2}}{(q^2;q^2)_{i}(q^2;q^2)_{n-i}(q^2;q^4)_i}.
\end{align}
Substituting the Bailey pair \eqref{C(1)} into \eqref{id-BP-twice} and then replacing $q$ by $q^2$, we deduce from \eqref{proof-1-1-even} that
\begin{align}\label{proof-1-1-even-result}
&G_0(1,1,1)=\frac{(-q;q^2)_\infty}{(q^2;q^2)_\infty}\Big(1+\sum_{n=1}^\infty (-1)^nq^{22n^2}(q^{2n}+q^{-2n})\Big) \nonumber \\
&=\frac{(-q;q^2)_\infty}{(q^2;q^2)_\infty}(q^{20},q^{24},q^{44};q^{44})_\infty.  \quad (\text{by \eqref{Jacobi}})
\end{align}

(2) Setting $(\sigma,u,v,w)=(1,1,1,1)$ in \eqref{Ftable3-1}, we have
\begin{align}\label{proof-1-1-odd}
     &G_1(1,1,1)=q(-q^2;q^2)_\infty\sum_{i,j\ge 0}\frac{q^{4i^2+4ij+2j^2+4i+2j}}{(q;q)_{2i+1}(q^2;q^2)_j(-q^2;q^2)_i} \nonumber\\
     &=q\frac{(-q^2;q^2)_\infty}{1-q}\sum_{n= 0}^\infty q^{2n^2+2n}\sum_{i=0}^n\frac{q^{2i^2+2i}}{(q^3;q^2)_{i}(q^2;q^2)_{n-i}(q^4;q^4)_i}.
\end{align}
Substituting the Bailey pair \eqref{G(2)} into \eqref{id-BP-twice} and then replacing $q$ by $q^2$, we deduce that
\begin{align}
    &\sum_{n=0}^\infty q^{2n^2+2n}\sum_{i=0}^n \frac{q^{2i^2+2i}}{(-q^3;q^2)_{i}(q^2;q^2)_{n-i}(q^4;q^4)_i} \nonumber \\
    &=\frac{1}{(1-q)(q^4;q^2)_\infty}\Big(\sum_{n=0}^\infty (-1)^n q^{\frac{11}{2}(n^2+n)}(q^{-n}-q^{n+1})\Big) \nonumber \\
    &=\frac{(q,q^{10},q^{11};q^{11})_\infty}{(1-q)(q^4;q^2)_\infty}.  \quad (\text{by \eqref{Jacobi}})
\end{align}
Substituting this identity with $q$ replaced by $-q$ into \eqref{proof-1-1-odd}, we obtain \eqref{G-1-odd}.

 (3) Setting $(\sigma,u,v,w)=(0,1,1,q)$ in \eqref{Ftable3-1}, we have
 \begin{align}
     &G_0(1,1,q)=(-q^2;q^2)_\infty\sum_{i,j\ge 0}\frac{q^{4i^2+2j^2+4ij}}{(q;q)_{2i}(q^2;q^2)_j(-q^2;q^2)_i} \nonumber \\
     &=(-q^2;q^2)_\infty\sum_{n=0}^\infty q^{2n^2}\sum_{i=0}^n\frac{q^{2i^2}}{(q;q^2)_{i}(q^2;q^2)_{n-i}(q^4;q^4)_i}. \label{exam1-G0-3}
\end{align}
Substituting the Bailey pair \eqref{B-G(1)} into \eqref{id-BP-twice} and then replacing $q$  by $q^2$,  we deduce that
\begin{align}
    &\sum_{n=0}^\infty q^{2n^2}\sum_{i=0}^n\frac{q^{2i^2}}{(-q;q^2)_{i}(q^2;q^2)_{n-i}(q^4;q^4)_i}\nonumber \\
    &=\frac{1}{(q^2;q^2)_\infty} \Big(1+\sum_{n=1}^\infty (-1)^nq^{\frac{11}{2}n^2-\frac{1}{2}n}(1+q^n)\Big) \nonumber \\
    &=\frac{(q^5,q^6,q^{11};q^{11})_\infty}{(q^2;q^2)_\infty}.  \quad (\text{by \eqref{Jacobi}}) \label{exam1-G0-3-proof}
\end{align}
Substituting \eqref{exam1-G0-3-proof} with $q$ replaced by $-q$ into \eqref{exam1-G0-3}, we obtain \eqref{G-2-even}.

(4) Setting $(\sigma,u,v,w)=(1,1,1,q)$ in \eqref{Ftable3-1}, we have
 \begin{align*}
     &G_1(1,1,q)=q(-q^3,q^2)_\infty\sum_{i,j\ge 0}\frac{q^{4i^2+4ij+2j^2+4i+2j}}{(q;q)_{2i+1}(q^2;q^2)_j(-q^3;q^2)_i} \nonumber \\
     &=q\frac{(-q^3;q^2)_\infty}{1-q}\sum_{n=0}^\infty q^{2n^2+2n}\sum_{i=0}^n\frac{q^{2i^2+2i}}{(q^2;q^2)_{n-i}(q^2;q^2)_i(q^6;q^4)_i}.
\end{align*}
Substituting the Bailey pair \eqref{C(3)} into \eqref{id-BP-twice} and then replacing $q$ by $q^2$, we deduce that
\begin{align*}
     &G_1(1,1,q)=q\frac{(-q^3;q^2)_\infty}{(1-q)(q^4;q^2)_\infty}\Big(\sum_{n=0}^\infty(-1)^nq^{22n^2+10n}+\sum_{n=0}^\infty (-1)^{n+1}q^{22n^2+34n+12}\Big) \nonumber \\
     &=q\frac{(-q;q^2)_\infty}{(q^2;q^2)_\infty}(q^{12},q^{32},q^{44};q^{44})_\infty.  \quad (\text{by \eqref{Jacobi}})
 \end{align*}

(5) Setting $(\sigma,u,v,w)=(0,q,1,q)$ in \eqref{Ftable3-1}, we have
 \begin{align}
     &G_0(q,1,q)=(-q^2;q^2)_\infty\sum_{i,j\ge 0}\frac{q^{4i^2+4ij+2j^2+2i}}{(q;q)_{2i}(q^2;q^2)_j(-q^2;q^2)_i} \nonumber \\
     &=(-q^2;q^2)_\infty\sum_{n=0}^\infty q^{2n^2}\sum_{i=0}^n\frac{q^{2i^2+2i}}{(q;q^2)_i(q^2;q^2)_{n-i}(q^4;q^4)_i}. \label{exam1-G0-5}
\end{align}
Substituting the Bailey pair \eqref{G(3)} into \eqref{id-BP-twice} and then replacing $q$ by  $q^2$, we deduce that
\begin{align}
    &\sum_{n=0}^\infty q^{2n^2}\sum_{i=0}^n\frac{q^{2i^2+2i}}{(-q;q^2)_i(q^2;q^2)_{n-i}(q^4;q^4)_i} \nonumber \\
    &=\frac{1}{(q^2;q^2)_\infty} \sum_{n=0}^\infty (-1)^nq^{\frac{11}{2}n^2-\frac{3}{2}n}(1+q^{3n}) \nonumber \\
    &=\frac{(q^4,q^7,q^{11};q^{11})_\infty}{(q^2;q^2)_\infty}. \quad (\text{by \eqref{Jacobi}}) \label{exam1-G0-5-proof}
\end{align}
Substituting \eqref{exam1-G0-5-proof} with $q$ replaced by $-q$ into \eqref{exam1-G0-5}, we obtain \eqref{G-3-even}.

(6) Setting $(\sigma,u,v,w)=(1,q,1,q)$ in \eqref{Ftable3-1}, we have
\begin{align*}
     &G_1(q,1,q)=(-q^3;q^2)_\infty\sum_{i,j\ge 0}\frac{q^{4i^2+4ij+2j^2+6i+2j+2}}{(q;q)_{2i+1}(q^2;q^2)_j(-q^3;q^2)_i} \\
     &=q^2\frac{(-q^3;q^2)_\infty}{1-q}\sum_{n= 0}^\infty q^{2n^2+2n}\sum_{i=0}^n\frac{q^{2i^2+4i}}{(q^2;q^2)_{n-i}(q^2;q^2)_i(q^6;q^4)_i}.
\end{align*}
Substituting the Bailey pair \eqref{C(4)} into \eqref{id-BP-twice} and then replacing $q$ by $q^2$,  we deduce that
\begin{align*}
    &G_1(q,1,q)=q^2\frac{(-q^3;q^2)_\infty}{(1-q)(q^4;q^2)_\infty}\Big(\sum_{n=0}^\infty (-1)^nq^{22n^2+14n}+\sum_{n=0}^\infty(-1)^{n+1}q^{22n^2+30n+8}\Big) \nonumber \\
 &=q^2\frac{(-q;q^2)_\infty(q^8,q^{36},q^{44};q^{44})_\infty}{(q^2;q^2)_\infty}. \quad (\text{by \eqref{Jacobi}})
 \end{align*}

 (7)  Setting $(\sigma,u,v,w)=(0,q^2,q^2,q)$ in \eqref{Ftable3-1}, we have
  \begin{align}
     &G_0(q^2,q^2,q)=(-q^2;q^2)_\infty\sum_{i,j\ge 0}\frac{q^{4i^2+4ij+2j^2+4i+2j}}{(q;q)_{2i}(q^2;q^2)_j(-q^2;q^2)_i} \nonumber \\
     &=(-q^2;q^2)_\infty\sum_{n= 0}^\infty q^{2n^2+2n}\sum_{i=0}^n\frac{q^{2i^2+2i}}{(q;q^2)_{i}(q^2;q^2)_{n-i}(q^4;q^4)_i}. \label{exam1-G0-7}
\end{align}
Substituting the Bailey pair \eqref{B-G(1.1)} into \eqref{id-BP-twice} and then replacing $q$ by  $q^2$,  we deduce that
\begin{align}
    &\sum_{n= 0}^\infty q^{2n^2+2n}\sum_{i=0}^n\frac{q^{2i^2+2i}}{(-q;q^2)_{i}(q^2;q^2)_{n-i}(q^4;q^4)_i} \nonumber \\
    &=\frac{1}{(q^2;q^2)_\infty} \sum_{n=0}^\infty (-1)^n q^{\frac{11}{2}(n^2+n)}(q^{-2n}-q^{2n+2})  \nonumber \\
    &=\frac{(q^2,q^9,q^{11};q^{11})_\infty}{(q^2;q^2)_\infty}. \quad (\text{by \eqref{Jacobi}}) \label{exam1-G0-7-proof}
\end{align}
Substituting \eqref{exam1-G0-7-proof} with $q$ replaced by $-q$ into \eqref{exam1-G0-7}, we obtain \eqref{G-4-even}.

(8) Setting $(\sigma,u,v,w)=(1,q^2,q^2,q)$ in \eqref{Ftable3-1}, we have
 \begin{align*}
     &G_1(q^2,q^2,q)=(-q^3;q^2)_\infty\sum_{i,j\ge 0}\frac{q^{4i^2+4ij+2j^2+8i+4j+3}}{(q;q)_{2i+1}(q^2;q^2)_j(-q^3;q^2)_i}  \\
     &=q^3\frac{(-q^3;q^2)_\infty}{1-q}\sum_{n= 0}^\infty q^{2n^2+4n}\sum_{i=0}^n\frac{q^{2i^2+4i}}{(q^2;q^2)_{n-i}(q^2;q^2)_i(q^6;q^4)_i}.
\end{align*}
Substituting the Bailey pair \eqref{C4*} into \eqref{id-BP-twice} and then replacing $q$ by $q^2$, we deduce that
\begin{align*}
     &G_1(q^2,q^2,q)=q^3\frac{(-q^3;q^2)_\infty}{(1-q)(q^6;q^2)_\infty}\sum_{n=0}^\infty (-1)^n q^{22n^2+18n}\frac{1-q^{8n+4}}{1-q^4} \nonumber \\
     &=q^3\frac{(-q;q^2)_\infty}{(q^2;q^2)_\infty}(q^4,q^{40},q^{44};q^{44})_\infty. \quad (\text{by \eqref{Jacobi}})
 \end{align*}

(9) By \eqref{Ftable3-1} we have
\begin{align}
G_0(q,1,1)&=(-q;q^2)_\infty \sum_{i,j\geq 0} \frac{q^{4i^2+2j^2+4ij+2i}}{(q;q)_{2i}(q^2;q^2)_j(-q;q^2)_i}  \nonumber \\
&=(-q;q^2)_\infty \sum_{n=0}^\infty q^{2n^2} \sum_{i=0}^n \frac{q^{2i^2+2i}}{(q^2;q^2)_{n-i}(q^2;q^2)_i(q^2;q^4)_i}, \\
G_1(q^3,q^2,q)&=(-q^3;q^2)_\infty \sum_{i,j\geq 0}\frac{q^{(2i+1)^2+2j^2+2(2i+1)j+3(2i+1)+2j}}{(q;q)_{2i+1}(q^2;q^2)_j(-q^3;q^2)_i} \nonumber \\
&=(-q;q^2)_\infty\sum_{n=0}^\infty q^{2(n+1)^2+2}\sum_{i=0}^n \frac{q^{2i^2+6i}}{(q^2;q^2)_{n-i}(q^2;q^2)_{i}(q^2;q^4)_{i+1}} \nonumber \\
&=(-q;q^2)_\infty \sum_{n=1}^\infty q^{2n^2} \sum_{i=1}^n \frac{q^{2i^2+2i-2}}{(q^2;q^2)_{n-i}(q^2;q^2)_{i-1}(q^2;q^4)_i}.
\end{align}
Adding them together, we deduce that
\begin{align}
G_0(q,1,1)+G_1(q^3,q^2,q)=(-q;q^2)_\infty \sum_{n=0}^\infty q^{2n^2} \sum_{i=0}^n \frac{q^{2i^2+2i}(1+q^{-2}-q^{2i-2})}{(q^2;q^2)_{n-i}(q^2;q^2)_i(q^2;q^4)_i}.
\end{align}
Substituting the Bailey pair \eqref{eq-exam1-key-BP} into \eqref{id-BP-twice} and then replacing $q$ by $q^2$, we deduce that
\begin{align*}
   & G_0(q,1,1)+G_1(q^3,q^2,q)=\frac{(-q;q^2)_\infty}{(q^2;q^2)_\infty} \Big(1+\sum_{n=1}^\infty (-1)^nq^{22n^2}(q^{6n}+q^{-6n})  \Big) \nonumber \\
&=\frac{(-q;q^2)_\infty}{(q^2;q^2)_\infty} (q^{16},q^{28},q^{44};q^{44})_\infty. \quad (\text{by \eqref{Jacobi}})
\end{align*}

(10) By \eqref{Ftable3-1} we have
\begin{align}
&G_1(q,1,1;q)=(-q^2;q^2)_\infty \sum_{i,j\geq 0} \frac{q^{4i^2+4ij+2j^2+6i+2j+2}}{(q;q)_{2i+1}(q^2;q^2)_j(-q^2;q^2)_i} \nonumber \\
&=q^2(-q^2;q^2)_\infty \sum_{n=0}^\infty q^{2n^2+2n} \sum_{i=0}^n  \frac{q^{2i^2+4i}}{(q^2;q^2)_{n-i}(q;q^2)_{i+1}(q^4;q^4)_i}, \\
&G_0(q^3,q^2,q;q)=(-q^2;q^2)_\infty \sum_{i,j\geq 0} \frac{q^{4i^2+2j^2+4ij+6i+2j}}{(q;q)_{2i}(q^2;q^2)_j(-q^2;q^2)_i} \nonumber \\
&=(-q^2;q^2)_\infty \sum_{n=0}^\infty q^{2n^2+2n} \sum_{i=0}^n \frac{q^{2i^2+4i}}{(q^2;q^2)_{n-i}(q;q^2)_i(q^4;q^4)_i}.
\end{align}
Adding them together and then replacing $q$ by $-q$, we deduce that
\begin{align}
&G_1(-q,1,1;-q)+G_0(-q^3,q^2,-q;-q) \nonumber \\
&=(-q^2;q^2)_\infty \sum_{n=0}^\infty q^{2n^2+2n} \sum_{i=0}^n \frac{q^{2i^2+4i}(1+q^2+q^{2i+1})}{(q^2;q^2)_{n-i}(-q;q^2)_{i+1}(q^4;q^4)_i}.
\end{align}
Substituting \eqref{eq-exam1-key-BP-add} into \eqref{id-BP-twice} and replacing $q$ by $q^2$, we deduce that
\begin{align}
&G_1(-q,1,1;-q)+G_0(-q^3,q^2,-q;-q) \nonumber \\
&=\frac{(-q^2;q^2)_\infty}{(q^4;q^2)_\infty} \times \frac{1}{1+q} \sum_{n=0}^\infty (-1)^nq^{\frac{11}{2}n^2+\frac{5}{2}n}\frac{1-q^{6n+3}}{1-q} \nonumber \\
&=\frac{(-q^2;q^2)_\infty}{(q^2;q^2)_\infty} \sum_{n=-\infty}^\infty (-1)^n q^{\frac{11}{2}n^2+\frac{5}{2}n} \nonumber \\
&=\frac{(-q^2;q^2)_\infty}{(q^2;q^2)_\infty}(q^3,q^8,q^{11};q^{11})_\infty. \label{add-proof-last}
\end{align}
Replacing $q$ by $-q$ in \eqref{add-proof-last}, we obtain \eqref{G-sum-2}.
\end{proof}
The modularity of the Nahm sums in this example are now clear from Theorem \ref{thm-exam-1} stated in Section 1.

\begin{proof}[Proof of Theorem \ref{thm-exam-1}]
This follows from Theorem \ref{thm-G-parity} and \eqref{G-split}.
\end{proof}

Now we are going to prove Theorem \ref{thm-exam-1-2}. We first establish two further identities on the function $G_\sigma(u,v,w)$.
\begin{theorem}\label{thm-G-parity2}
    We have
    \begin{align}
        &G_0(1,1,q)+qG_0(q,1,q)-G_0(q^2,1,q)=q\frac{(-q^3,q^8,-q^{11};-q^{11})_\infty}{(q^2,q^2,q^4;q^4)_\infty}, \label{add-exam1-G-1}
        \\
        &G_1(1,1,q)+qG_1(q,1,q)-G_1(q^2,1,q)=q\frac{(q^{16},q^{28},q^{44};q^{44})_\infty}{(q,q^3,q^4;q^4)_\infty}. \label{add-exam1-G-2}
    \end{align}
\end{theorem}
\begin{proof}
    (1) Setting $(\sigma,u,v,w)=(0,1,1,q), (0,q,1,q)$ and $(0,q^2,1,q)$ in \eqref{Ftable3-1}, we have
    \begin{align}\label{add-exam1-G0-sum-start}
        &G_0(1,1,q)+qG_0(q,1,q)-G_0(q^2,1,q) \nonumber \\
        &=(-q^2;q^2)_\infty\sum_{i,j\ge 0}\frac{q^{4i^2+2j^2+4ij}(1+q^{2i+1}-q^{4i})}{(q;q)_{2i}(q^2;q^2)_j(-q^2;q^2)_i} \nonumber
        \\
        &=(-q^2;q^2)_\infty\sum_{n=0}^\infty q^{2n^2}\sum_{i=0}^n\frac{q^{2i^2}(1+q^{2i+1}-q^{4i})}{(q;q^2)_i(q^2;q^2)_{n-i}(q^4;q^4)_i}.
    \end{align}
    Substituting the Bailey pair \eqref{eq-exam1-key-BP2} into \eqref{id-BP-twice} and then replacing $q$ by $q^2$, we deduce that
    \begin{align}\label{add-exam1-G0-sum-mid}
        &\sum_{n=0}^\infty q^{2n^2}\sum_{i=0}^n\frac{q^{2i^2}(1-q^{2i+1}-q^{4i})}{(-q;q^2)_i(q^2;q^2)_{n-i}(q^4;q^4)_i} \nonumber\\
        &=-\frac{q}{(q^2;q^2)_\infty}\Big(1+\sum_{n=1}^\infty(-1)^nq^{\frac{11}{2}n^2-\frac{5}{2}n}(1+q^{5n})\Big)\nonumber\\
        &=-q\frac{(q^3,q^8,q^{11};q^{11})_\infty}{(q^2;q^2)_\infty}. \quad (\text{by \eqref{Jacobi}})
    \end{align}
    Substituting \eqref{add-exam1-G0-sum-mid} with $q$ replaced by $-q$ into \eqref{add-exam1-G0-sum-start}, we obtain \eqref{add-exam1-G-1}.

    (2) Setting $(\sigma,u,v,w)=(1,1,1,q), (1,q,1,q)$ and $(1,q^2,1,q)$ in \eqref{Ftable3-1}, we have
    \begin{align}\label{add-exam1-G-2-start}
        &G_1(1,1,q)+qG_1(q,1,q)-G_1(q^2,1,q) \nonumber
        \\
        &=q(-q^3;q^2)_\infty\sum_{i,j\ge 0}\frac{q^{4i^2+2j^2+4ij+4i+2j}(1+q^{2i+2}-q^{4i+2})}{(q;q)_{2i+1}(q^2;q^2)_j(-q^3;q^2)_i} \nonumber
        \\
        &=q\frac{(-q^3;q^2)_\infty}{1-q}\sum_{n=0}^\infty q^{2n^2+2n}\sum_{i=0}^n\frac{q^{2i^2+2i}(1+q^{2i+2}-q^{4i+2})}{(q^2;q^2)_{n-i}(q^2;q^2)_i(q^6;q^4)_i}.
    \end{align}
Substituting the Bailey pair \eqref{eq-exam1-key-BP3} into \eqref{id-BP-twice} and then replacing $q$ by $q^2$, we deduce from \eqref{add-exam1-G-2-start} that
\begin{align*}
      &G_1(1,1,q)+qG_1(q,1,q)-G_1(q^2,1,q) \nonumber
        \\
    &=q\frac{(-q^3;q^2)_\infty}{(1-q)(q^4;q^2)_\infty}\Big(\sum_{n=0}^\infty(-1)^nq^{22n^2+6n}+\sum_{n=0}^\infty (-1)^{n+1}q^{22n^2+38n+16}\Big) \nonumber \\
     &=q\frac{(-q;q^2)_\infty}{(q^2;q^2)_\infty}(q^{16},q^{28},q^{44};q^{44})_\infty.  \quad (\text{by \eqref{Jacobi}}) \qedhere
 \end{align*}
\end{proof}

\begin{proof}[Proof of Theorem \ref{thm-exam-1-2}]
The left side of \eqref{revise-add-exam1-companion} is
\begin{align}
&G(q,1,1)+G(q^3,q^2,q) \nonumber \\
&=\big( G_0(q,1,1)+G_1(q^3,q^2,q) \big)+\big( G_1(q,1,1)+G_0(q^3,q^2,q)\big).
\end{align}
Adding \eqref{G-sum-1} and \eqref{G-sum-2} together, we obtain \eqref{revise-add-exam1-companion}.

The left side of \eqref{add-exam1-companion} is
\begin{align}
&G(1,1,q)+qG(q,1,q)-G(q^2,1,q) \nonumber \\
&=\sum_{\sigma=0}^1 \left(G_\sigma(1,1,q)+qG_\sigma(q,1,q)-G_\sigma(q^2,1,q) \right).
\end{align}
The identity \eqref{add-exam1-companion} then follows from Theorem \ref{thm-G-parity2}.
\end{proof}

We are now able to prove the modular transformation formulas stated in Section \ref{sec-intro}.
\begin{proof}[Proof of Theorem \ref{thm-Mizuno-conj}]
Since $S^{-1}=2S$, then second formulas in \eqref{eq-conj-tran}, \eqref{eq-conj-tran-add} and \eqref{eq-M54} follow after replacing $\tau$ by $-1/(2\tau)$ in the first formulas of them. It remains to prove the first formulas in these equations.

For convenience, for $i=1,2,3,4,5$, we denote the $i$-th component of $U(\tau)$, $U^*(\tau)$, $V(\tau)$ and $V^*(\tau)$ by $U_i(\tau)$, $U_i^*(\tau)$, $V_i(\tau)$ and $V_i^*(\tau)$, respectively. We divide our proof into three parts.

(1) From \cite[Theorems 3.17 and 3.18]{WW-I} and \eqref{add-g-product} we have
\begin{align}
U_1(2\tau)=q^{-\frac{5}{44}}\frac{(-q;q^2)_\infty}{(q^2;q^2)_\infty}(q^5,q^6,q^{11};q^{11})_\infty=\frac{\mathfrak{f}(2\tau)}{\eta(2\tau)}g_{1,11}(\frac{\tau}{2}),
\label{U-g-1}\\
U_2(2\tau)=q^{-\frac{1}{44}}\frac{(-q;q^2)_\infty}{(q^2;q^2)_\infty}(q^4,q^7,q^{11};q^{11})_\infty=\frac{\mathfrak{f}(2\tau)}{\eta(2\tau)}g_{3,11}(\frac{\tau}{2}),
\label{U-g-2}\\
U_3(2\tau)=q^{\frac{7}{44}}\frac{(-q;q^2)_\infty}{(q^2;q^2)_\infty}(q^3,q^8,q^{11};q^{11})_\infty=\frac{\mathfrak{f}(2\tau)}{\eta(2\tau)}g_{5,11}(\frac{\tau}{2}),
\label{U-g-3} \\
U_4(2\tau)=q^{\frac{19}{44}}\frac{(-q;q^2)_\infty}{(q^2;q^2)_\infty}(q^2,q^9,q^{11};q^{11})_\infty=\frac{\mathfrak{f}(2\tau)}{\eta(2\tau)}g_{7,11}(\frac{\tau}{2}),
\label{U-g-4} \\
U_5(2\tau)=q^{\frac{35}{44}}\frac{(-q;q^2)_\infty}{(q^2;q^2)_\infty}(q,q^{10},q^{11};q^{11})_\infty=\frac{\mathfrak{f}(2\tau)}{\eta(2\tau)}g_{9,11}(\frac{\tau}{2}). \label{U-g-5}
\end{align}
Similarly, from Theorem \ref{thm-G-parity} and \eqref{add-g-product} we have
\begin{align}
V_1(\tau)=q^{-\frac{7}{88}}\frac{(-q;q^2)_\infty}{(q^2;q^2)_\infty}(q^{20},q^{24},q^{44};q^{44})_\infty=\frac{\mathfrak{f}(2\tau)}{\eta(2\tau)}g_{1,11}(2\tau),
\\
V_2(\tau)=q^{\frac{25}{88}}\frac{(-q;q^2)_\infty}{(q^2;q^2)_\infty}(q^{16},q^{28},q^{44};q^{44})_\infty=\frac{\mathfrak{f}(2\tau)}{\eta(2\tau)}g_{3,11}(2\tau),
\\
V_3(\tau)=q^{\frac{89}{88}}\frac{(-q;q^2)_\infty}{(q^2;q^2)_\infty} (q^{12},q^{32},q^{44};q^{44})_\infty=\frac{\mathfrak{f}(2\tau)}{\eta(2\tau)}g_{5,11}(2\tau),
\\
V_4(\tau)=q^{\frac{185}{88}}\frac{(-q;q^2)_\infty}{(q^2;q^2)_\infty}(q^8,q^{36},q^{44};q^{44})_\infty=\frac{\mathfrak{f}(2\tau)}{\eta(2\tau)}g_{7,11}(2\tau),
\\
V_5(\tau)=q^{\frac{313}{88}}\frac{(-q;q^2)_\infty}{(q^2;q^2)_\infty}(q^4,q^{40},q^{44};q^{44})_\infty=\frac{\mathfrak{f}(2\tau)}{\eta(2\tau)}g_{9,11}(2\tau).
\end{align}
Applying \eqref{eta-tran}, \eqref{Weber-1} and Lemma \ref{lem-modular}, we deduce that for $j\in \{1,2,3,4,5\}$,
\begin{align}
&V_j\left(-\frac{1}{4\tau}\right)=\frac{\mathfrak{f}(-1/(2\tau))}{\eta(-1/(2\tau))}g_{2j-1,11}\left(-\frac{1}{2\tau}\right)
=\sqrt{\frac{1}{22}}\frac{\mathfrak{f}(2\tau)}{\eta(2\tau)}\sum_{k=0}^{21} e^{\frac{\pi i(2k+1)(2j-1)}{22}}h_{k+\frac{1}{2},11}(2\tau) \nonumber \\
&=\sqrt{\frac{1}{22}}\frac{\mathfrak{f}(2\tau)}{\eta(2\tau)}\sum_{k=0}^{10} \left( e^{\frac{\pi i(2k+1)(2j-1)}{22}} h_{k+\frac{1}{2},11}(2\tau)+e^{\frac{\pi i(2(21-k)+1)(2j-1)}{22}}h_{21-k+\frac{1}{2},11}(2\tau) \right)\nonumber \\
&=\sqrt{\frac{1}{22}}\frac{\mathfrak{f}(2\tau)}{\eta(2\tau)}\sum_{k=0}^{10} 2\cos \frac{(2k+1)(2j-1)\pi}{22} h_{k+\frac{1}{2},11}(2\tau) \quad \text{(by \eqref{g-h-period})} \nonumber \\
&=\sqrt{\frac{2}{11}}\frac{\mathfrak{f}(2\tau)}{\eta(2\tau)}\sum_{k=0}^{4} \cos \frac{(2k+1)(2j-1)\pi}{22} \left(h_{k+\frac{1}{2},11}(2\tau)-h_{\frac{21}{2}-k,11}(2\tau)  \right) \nonumber \\
&=\sqrt{\frac{2}{11}}\frac{\mathfrak{f}(2\tau)}{\eta(2\tau)}\sum_{k=0}^{4} \cos \frac{(2k+1)(2j-1)\pi}{22} \left(h_{4k+2,44}\left(\frac{\tau}{2}\right)-h_{42-4k,44}\left(\frac{\tau}{2}\right)  \right) ~~\text{(by \eqref{hg-double})}  \nonumber \\
&=\sqrt{\frac{2}{11}}\frac{\mathfrak{f}(2\tau)}{\eta(2\tau)}\sum_{k=0}^{4} \cos \frac{(2k+1)(2j-1)\pi}{22} g_{2k+1,11}\left(\frac{\tau}{2}\right). \quad \text{(by \eqref{g-h-change})}
 \label{V1-tran}
\end{align}
Setting $j=1,2,3,4,5$ we obtain
\begin{align*}
    &V_1(-\frac{1}{2\tau})=\alpha_5U_1(\tau)+\alpha_4U_2(\tau)+\alpha_3U_3(\tau)+\alpha_2U_4(\tau)+\alpha_1U_5(\tau),
    \\
    &V_2(-\frac{1}{2\tau})=\alpha_4U_1(\tau))+\alpha_1U_2(\tau)-\alpha_2U_3(\tau)-\alpha_5U_4(\tau)-\alpha_3U_5(\tau),
    \\
    &V_3(-\frac{1}{2\tau})=\alpha_3U_1(\tau)-\alpha_2U_2(\tau)-\alpha_4U_3(\tau)+\alpha_1U_4(\tau)+\alpha_5U_5(\tau),
    \\
    &V_4(-\frac{1}{2\tau})=\alpha_2U_1(\tau)-\alpha_5U_2(\tau)+\alpha_1U_3(\tau)+\alpha_3U_4(\tau)-\alpha_4U_5(\tau),
    \\
    &V_5(-\frac{1}{2\tau})=\alpha_1U_1(\tau)-\alpha_3U_2(\tau)+\alpha_5U_3(\tau)-\alpha_4U_4(\tau)+\alpha_2U_5(\tau).
\end{align*}
This proves the first transformation formula in \eqref{eq-conj-tran}.

(2) In view of \eqref{F-defn} and \eqref{Fc-defn} we have
\begin{align}
&U_1^*(2\tau)=q^{-\frac{5}{44}}\Big(F_0(1,1,1;q)-F_1(1,1,1;q)\Big)\nonumber \\
&=q^{-\frac{5}{44}}F(1,1,1;-q) =\zeta_{88}^5U_1(2\tau+1).
\end{align}
Here for the last equality we used the fact that replacing $\tau$ by $\tau+\frac{1}{2}$ is equivalent to replacing $q$ by $-q$.
Arguing in this way, we can prove that
\begin{align}
&U_\ell^*(\tau)=\lambda_\ell U_\ell(\tau+1), \quad 1\leq \ell \leq 5, \\
&\lambda_1=\zeta_{88}^5, ~~ \lambda_2=-\zeta_{88}, ~~ \lambda_3=-\zeta_{88}^{-7}, ~~ \lambda_4=\zeta_{88}^{-19}, ~~\lambda_5=\zeta_{88}^{-35}.
\end{align}
If we denote $\Lambda=\mathrm{diag}(\lambda_1,\lambda_2,\lambda_3,\lambda_4,\lambda_5)$, then
\begin{align}\label{U-star-U-relation}
U^*(\tau)=\Lambda U(\tau+1).
\end{align}
From \eqref{U-g-1}--\eqref{U-g-5} we have for $1\leq \ell \leq 5$ that
\begin{align}\label{add-Ui-tau}
U_\ell(\tau+1)=\frac{\mathfrak{f}(\tau+1)}{\eta(\tau+1)}g_{2\ell-1,11}\left(\frac{\tau+1}{4}\right)=e^{-\frac{\pi i}{8}}\frac{\mathfrak{f}_1(\tau)}{\eta(\tau)}g_{2\ell-1,11}\left(\frac{\tau+1}{4}\right).
\end{align}
From \eqref{wU-wV-defn} and Theorem \ref{thm-G-parity} we have
\begin{align}
{V}_1^*\big(\tau+\frac{1}{2}\big)&=-\zeta_{176}^{-7}q^{1-\frac{7}{88}}\frac{(-q^2;q^2)_\infty}{(q^2;q^2)_\infty}(q,q^{10},q^{11};q^{11})_\infty
=-\zeta_{176}^{-7}\frac{\mathfrak{f}_2(2\tau)}{\eta(2\tau)}g_{9,11}\left(\frac{\tau}{2}\right), \label{wV-1} \\
{V}_2^*\big(\tau+\frac{1}{2}\big)&=\zeta_{176}^{25}q^{\frac{25}{88}}\frac{(-q^2;q^2)_\infty}{(q^2;q^2)_\infty} (q^3,q^8,q^{11};q^{11})_\infty
=\zeta_{176}^{25}\frac{\mathfrak{f}_2(2\tau)}{\eta(2\tau)}g_{5,11}\left(\frac{\tau}{2}\right), \label{wV-2} \\
{V}_3^*\big(\tau+\frac{1}{2}\big)&=\zeta_{176}q^{\frac{1}{88}}\frac{(-q^2;q^2)_\infty}{(q^2;q^2)_\infty}(q^5,q^6,q^{11};q^{11})_\infty
=\zeta_{176}\frac{\mathfrak{f}_2(2\tau)}{\eta(2\tau)}g_{1,11}\left(\frac{\tau}{2}\right), \label{wV-3} \\
{V}_4^*\big(\tau+\frac{1}{2}\big)&=\zeta_{176}^{9}q^{\frac{9}{88}}\frac{(-q^2;q^2)_\infty}{(q^2;q^2)_\infty}(q^4,q^7,q^{11};q^{11})_\infty
=\zeta_{176}^{9}\frac{\mathfrak{f}_2(2\tau)}{\eta(2\tau)}g_{3,11}\left(\frac{\tau}{2}\right),  \label{wV-4}\\
{V}_5^*\big(\tau+\frac{1}{2}\big)&= \zeta_{176}^{49}q^{\frac{49}{88}}\frac{(-q^2;q^2)_\infty}{(q^2;q^2)_\infty}(q^2,q^9,q^{11};q^{11})_\infty
=\zeta_{176}^{49}\frac{\mathfrak{f}_2(2\tau)}{\eta(2\tau)}g_{7,11}\left(\frac{\tau}{2}\right). \label{wV-5}
\end{align}
We let $W(\tau)=(W_1(\tau),W_2(\tau),W_3(\tau),W_4(\tau),W_5(\tau))^\mathrm{T}$ where
\begin{align}
W_j(\tau):=\frac{\mathfrak{f}_2(2\tau)}{\eta(2\tau)}g_{2j-1,11}\left(\frac{2\tau-1}{4}\right), \quad 1\leq j \leq 5,
\end{align}
and we denote
\begin{align}
P=\begin{pmatrix}
0 & 0 & 0 & 0 & -\zeta_{176}^{-7} \\
0 & 0 & \zeta_{176}^{25} & 0 & 0 \\
\zeta_{176} & 0 & 0 & 0 & 0 \\
0 & \zeta_{176}^9 & 0 & 0 & 0 \\
0 & 0 & 0 & \zeta_{176}^{49} & 0
\end{pmatrix}.
\end{align}
Replacing $\tau$ by $\tau-\frac{1}{2}$ in \eqref{wV-1}--\eqref{wV-5}, we deduce that
\begin{align}\label{wV-W-relation}
{V}^*(\tau)=PW(\tau).
\end{align}
Using Lemma \ref{lem-add-g-tran} we have for $1\leq j\leq 5$ that
\begin{align}
&W_{j}\left(-\frac{1}{2\tau}\right)= \frac{\mathfrak{f}_2(-1/\tau)}{\eta(-1/\tau)}g_{2j-1,11}\left(-\frac{\tau+1}{4\tau}\right)
=\frac{1}{\sqrt{-2i\tau}}\frac{\mathfrak{f}_1(\tau)}{\eta(\tau)}g_{2j-1,11}\left(-\frac{\tau+1}{4\tau}\right) \nonumber \\
&=\sqrt{\frac{i}{22}} \frac{\mathfrak{f}_1(\tau)}{\eta(\tau)}\sum_{\ell=1}^5 \Big(e^{-\pi i\frac{2j+2\ell-3+3(2j+2\ell-4)^2}{22}}+ e^{-\pi i\frac{2j-2\ell-1+3(2j-2\ell-2)^2}{22}} \Big) g_{2\ell-1,11}\left(\frac{\tau+1}{4}\right) \nonumber \\
&=\sum_{\ell=1}^5 c_{j\ell}U_\ell(\tau+1), \quad \text{(by \eqref{add-Ui-tau})}\label{W-tran-mid}
\end{align}
where
\begin{align}
c_{j\ell}=\frac{1}{\sqrt{22}}e^{\frac{3\pi i}{8}}\Big(e^{-\pi i\frac{2j+2\ell-3+3(2j+2\ell-4)^2}{22}}+ e^{-\pi i\frac{2j-2\ell-1+3(2j-2\ell-2)^2}{22}} \Big). \label{cjl-exp}
\end{align}
Let $C=(c_{j\ell})$ be a $5\times 5$ matrix with the $(j,\ell)$-entry $c_{j\ell}$. The formula \eqref{W-tran-mid} can be rewritten as
\begin{align}\label{W-U-tran}
W\left(-\frac{1}{2\tau}\right)=CU(\tau+1).
\end{align}
Combining \eqref{U-star-U-relation}, \eqref{wV-W-relation} and \eqref{W-U-tran}, we deduce that
\begin{align}\label{wV-U-star}
{V}^*\left(-\frac{1}{2\tau}\right)=MU^*(\tau), \quad M=PC\Lambda^{-1}.
\end{align}
Simplifying the entries in $M$ we see that $M=S$ and this proves the first formula in \eqref{eq-conj-tran-add}.

(3) From \eqref{wU-wV-defn} we have
\begin{align}
g(\tau)=\frac{1}{2}\big(U(\tau)+{U}^*(\tau)\big), \quad g^\vee(\tau)=V(\tau)+{V}^*(\tau).
\end{align}
Mizuno's conjectural formula \eqref{eq-M54} is equivalent to
\begin{align}\label{M-54-equivalent}
 V\left(-\frac{1}{2\tau}\right)+{V}^*\left(-\frac{1}{2\tau}\right)=S\big(U(\tau)+{U}^*(\tau)\big).
\end{align}
This follows from the first formulas in \eqref{eq-conj-tran} and \eqref{eq-conj-tran-add}.
\end{proof}

Examples 7 and 8 are closely connected to each other. We will employ some results in Example 8 to complete the proof of Example 7.
\subsubsection{Example 7}
This example corresponds to
\begin{align*}
A=\begin{pmatrix}
1 & 1 & 1\\
1 & 2 & 2\\
1/2&1 &3/2\\
\end{pmatrix},\quad
AD=\begin{pmatrix}
2 & 2 & 1\\
2 & 4 & 2\\
1 & 2 &3/2\\
\end{pmatrix},\quad
b\in  \bigg\{
\begin{pmatrix}
0 \\ 0 \\ 0
\end{pmatrix},
\begin{pmatrix}
1 \\ 0 \\1/2
\end{pmatrix},
\begin{pmatrix}
1 \\ 2 \\3/2
\end{pmatrix}
\bigg\}.
\end{align*}
To prove its modularity, we first establish the following result.
\begin{theorem}\label{thm-7-parity}
    We have
    \begin{align}
        &\sum_{i,j,k\ge 0}\frac{q^{3i^2+j^2+2k^2+2ij+4ik+2jk}}{(q;q)_{2i}(q^2;q^2)_j(q^2;q^2)_k}
        =
        \frac{J_2J_{16,36}}{J_1J_4}, \label{conj-3-even-1}
        \\
         &\sum_{i,j,k\ge 0}\frac{q^{3i^2+j^2+2k^2+2ij+4ik+2jk+i+j}}{(q;q)_{2i}(q^2;q^2)_j(q^2;q^2)_k}
        =\frac{J_8J_{18}^4J_{36}}{J_{2,8}J_{2,18}J_{5,18}J_{6,18}J_{9,36}J_{8,36}}, \label{conj-3-even-2}
        \\
        &\sum_{i,j,k\ge 0}\frac{q^{3i^2+j^2+2k^2+2ij+4ik+2jk+3i+j+2k}}{(q;q)_{2i}(q^2;q^2)_j(q^2;q^2)_k}
        =\frac{J_8J_{18}^4J_{36}}{J_{2,8}J_{6,18}J_{7,18}J_{8,18}J_{9,36}J_{4,36}}, \label{conj-3-even-3}  \\
        &\sum_{i,j,k\ge 0}\frac{q^{3i^2+j^2+2k^2+2ij+4ik+2jk+3i+j+2k}}{(q;q)_{2i+1}(q^2;q^2)_j(q^2;q^2)_k}
        =\frac{J_8J_{18}^4J_{36}}{J_{2,8}J_{1,18}J_{4,18}J_{6,18}J_{9,36}J_{16,36}}, \label{conj-3-odd-1}
        \\
        &\sum_{i,j,k\ge 0}\frac{q^{3i^2+j^2+2k^2+2ij+4ik+2jk+4i+2j+2k}}{(q;q)_{2i+1}(q^2;q^2)_j(q^2;q^2)_k}
        =\frac{J_2J_{8,36}}{J_1J_4}, \label{conj-3-odd-2}
        \\
        &\sum_{i,j,k\ge 0}\frac{q^{3i^2+j^2+2k^2+2ij+4ik+2jk+6i+2j+4k}}{(q;q)_{2i+1}(q^2;q^2)_j(q^2;q^2)_k}
        =\frac{J_2J_{4,36}}{J_1J_4}. \label{conj-3-odd-3}
    \end{align}
\end{theorem}
This can be proved in a way similar to the proof of Theorem \ref{thm-8-parity} (see Example 8) using Bailey pairs. We omit details but present a different proof here. We will show that they are equivalent to some identities in Theorem \ref{thm-8-parity}.

To prove the equivalence, we establish two useful lemmas.
\begin{lemma}\label{lem26-27}
    We have
    \begin{align}
        \sum_{i\ge 0}\frac{u^iq^{2i^2-i}}{(q;q)_{2i}}
        =
        (u;q^2)_\infty\sum_{i\ge 0}\frac{u^i}{(q;q)_{2i}},\label{eq-u-1}\\
        \sum_{i\ge 0}\frac{u^iq^{2i^2+i}}{(q;q)_{2i+1}}
        =
        (u;q^2)_\infty\sum_{i\ge 0}\frac{u^i}{(q;q)_{2i+1}}. \label{eq-u-2}
    \end{align}
\end{lemma}
\begin{proof}
By Euler's identities \eqref{euler-1} and \eqref{euler-2} we have
\begin{align}
    \sum_{n=0}^\infty \frac{u^nq^{n(n-1)/2}}{(q;q)_n}=(-u;q)_\infty=\frac{(u^2;q^2)_\infty}{(u;q)_\infty}=(u^2;q^2)_\infty \sum_{n=0}^\infty \frac{u^n}{(q;q)_n}.
\end{align}
Extracting the terms with even and odd powers of $u$ and then replacing $u^2$ by $u$, we obtain \eqref{eq-u-1} and \eqref{eq-u-2}, respectively.
\end{proof}

\begin{lemma}\label{lem-equivalence-key}
We have
\begin{align}
    &\sum_{i,j,k\ge 0}\frac{u^{i+j}v^{i+k}q^{2jk}}{(q;q)_{2i}(q^2;q^2)_j(q^2;q^2)_k}
    =
     \sum_{i,j,k\ge 0}\frac{u^{i+j}v^{i+k}q^{2i^2-i}}{(q;q)_{2i}(q^2;q^2)_j(q^2;q^2)_k},\label{eq-2627-1}
     \\
     &\sum_{i,j,k\ge 0}\frac{u^{i+j}v^{i+k}q^{2jk}}{(q;q)_{2i+1}(q^2;q^2)_j(q^2;q^2)_k}
    =
     \sum_{i,j,k\ge 0}\frac{u^{i+j}v^{i+k}q^{2i^2+i}}{(q;q)_{2i+1}(q^2;q^2)_j(q^2;q^2)_k}.\label{eq-2627-2}
\end{align}
\end{lemma}
\begin{proof}
By  the $q$-binomial theorem \eqref{eq-qbinomial} we have
\begin{align}\label{id-uv}
    &\frac{(uv;q^2)_\infty}{(u,v;q^2)_\infty}
    =
    \sum_{j\ge 0}\frac{(v;q^2)_ju^j}{(q^2;q^2)_j(v;q^2)_\infty} \nonumber \\
     &=
    \sum_{j\ge 0}\frac{u^j}{(q^2;q^2)_j(vq^{2j};q^2)_\infty}
    =\sum_{j,k\ge 0}\frac{u^jv^kq^{2jk}}{(q^2;q^2)_j(q^2;q^2)_k}.
\end{align}
Replacing $u$ by $uv$ in \eqref{eq-u-1} and \eqref{eq-u-2} and then
multiplying both sides of the resulting identities by \eqref{id-uv}, we obtain \eqref{eq-2627-1} and \eqref{eq-2627-2}, respectively.
\end{proof}
\begin{proof}[Proof of Theorem \ref{thm-7-parity}]
Comparing the coefficients of $u^mv^n$ on both sides of the identities in Lemma \ref{lem-equivalence-key}, we deduce that for any $m,n\geq 0$
\begin{align}
    \sum_{\begin{smallmatrix} i+j=m,i+k=n \\ i,j,k\geq 0 \end{smallmatrix}} \frac{q^{2jk}}{(q;q)_{2i}(q^2;q^2)_j(q^2;q^2)_k}=\sum_{\begin{smallmatrix} i+j=m,i+k=n \\ i,j,k\geq 0 \end{smallmatrix}} \frac{q^{2i^2-i}}{(q;q)_{2i}(q^2;q^2)_j(q^2;q^2)_k}, \label{78-sum-id-1} \\
     \sum_{\begin{smallmatrix} i+j=m,i+k=n \\ i,j,k\geq 0 \end{smallmatrix}} \frac{q^{2jk}}{(q;q)_{2i+1}(q^2;q^2)_j(q^2;q^2)_k}=\sum_{\begin{smallmatrix} i+j=m,i+k=n \\ i,j,k\geq 0 \end{smallmatrix}} \frac{q^{2i^2+i}}{(q;q)_{2i+1}(q^2;q^2)_j(q^2;q^2)_k}. \label{78-sum-id-2}
\end{align}
Therefore, we have
\begin{align}
     &\sum_{i,j,k\ge 0}\frac{q^{3i^2+j^2+2k^2+2ij+4ik+2jk}u^{i+j}v^{i+k}}{(q;q)_{2i}(q^2;q^2)_j(q^2;q^2)_k} \nonumber \\
     &=\sum_{m,n\geq 0} q^{m^2+2n^2}u^mv^n \sum_{\begin{smallmatrix}
        i+j=m,i+k=n \\ i,j,k\geq 0
        \end{smallmatrix}} \frac{q^{2jk}}{(q;q)_{2i}(q^2;q^2)_j(q^2;q^2)_k} \\
        &=\sum_{m,n\geq 0} q^{m^2+2n^2}u^mv^n\sum_{i+j=m,i+k=n} \frac{q^{2i^2-i}}{(q;q)_{2i}(q^2;q^2)_j(q^2;q^2)_k}  \quad (\text{by \eqref{78-sum-id-1}}) \nonumber \\
        &=\sum_{i,j,k\ge 0}\frac{q^{5i^2+j^2+2k^2+2ij+4ik-i}u^{i+j}v^{i+k}}{(q;q)_{2i}(q^2;q^2)_j(q^2;q^2)_k}. \label{eq-78-transfer-1}
\end{align}
Setting $(u,v)=(1,1)$, $(q,1)$ and $(q,q^2)$ in \eqref{eq-78-transfer-1}, we obtain \eqref{conj-3-even-1}--\eqref{conj-3-even-3} from \eqref{ex27-2-1}, \eqref{ex27-3-1} and \eqref{ex27-4-1}, respectively.

Similarly, using \eqref{78-sum-id-2} we can prove that
\begin{align}
     &\sum_{i,j,k\ge 0}\frac{q^{3i^2+j^2+2k^2+2ij+4ik+2jk}u^{i+j}v^{i+k}}{(q;q)_{2i+1}(q^2;q^2)_j(q^2;q^2)_k}
        = \sum_{i,j,k\ge 0}\frac{q^{5i^2+j^2+2k^2+2ij+4ik+i}u^{i+j}v^{i+k}}{(q;q)_{2i+1}(q^2;q^2)_j(q^2;q^2)_k}. \label{eq-78-transfer-2}
\end{align}
Setting $(u,v)=(q,q^2)$, $(q^2,q^2)$ and $(q^2,q^4)$, we obtain \eqref{conj-3-odd-1}--\eqref{conj-3-odd-3} from \eqref{ex27-2-2}, \eqref{ex27-3-2} and \eqref{ex27-4-2}, respectively.
\end{proof}

As a consequence, we have the following theorem which justifies the modularity of this example.
\begin{theorem}\label{thm-7}
We have
\begin{align}
    &\sum_{i,j,k\ge 0}\frac{q^{3i^2+4j^2+8k^2+4ij+8ik+8jk}}{(q^4;q^4)_i(q^8;q^8)_j(q^8;q^8)_k}
    =
    \frac{J_{64,144}}{J_{4,16}}
    +q^3\frac{J_{32}J_{72}^4J_{144}}{J_{8,32}J_{4,72}J_{16,72}J_{24,72}J_{36,144}J_{64,144}}, \label{eq-thm7-1}
    \\
    &\sum_{i,j,k\ge 0}\frac{q^{3i^2+4j^2+8k^2+4ij+8ik+8jk+2i+4j}}{(q^4;q^4)_i(q^8;q^8)_j(q^8;q^8)_k}
    =
    q^5\frac{J_{32,144}}{J_{4,16}}
    +\frac{J_{32}J_{72}^4J_{144}}{J_{8,32}J_{8,72}J_{20,72}J_{24,72}J_{36,144}J_{32,144}}, \label{eq-thm7-2}
    \\
   &\sum_{i,j,k\ge 0}\frac{q^{3i^2+4j^2+8k^2+4ij+8ik+8jk+6i+4j+8k}}{(q^4;q^4)_i(q^8;q^8)_j(q^8;q^8)_k}
    =
   q^9\frac{J_{16,144}}{J_{4,16}}
    +\frac{J_{32}J_{72}^4J_{144}}{J_{8,32}J_{24,72}J_{28,72}J_{32,72}J_{36,144}J_{16,144}}. \label{eq-thm7-3}
\end{align}
\end{theorem}
\begin{proof}
Let
\begin{align}\label{exam7-F-defn}
    F(u,v,w)=F(u,v,w;q):= &\sum_{i,j,k\ge 0}\frac{q^{\frac{3}{4}i^2+j^2+2k^2+ij+2ik+2jk}u^iv^jw^k}{(q;q)_i(q^2;q^2)_j(q^2;q^2)_k}.
\end{align}
Then
\begin{align}\label{exam7-F-split}
    F(u,v,w)=F_0(u,v,w)+F_1(u,v,w)
\end{align}
where
\begin{align}\label{exam7-F-parity-defn}
  F_\sigma(u,v,w)=F_\sigma(u,v,w;q):=  &\sum_{\begin{smallmatrix}i,j,k\ge 0 \\ i\equiv \sigma \,\, \text{(mod 2)} \end{smallmatrix}}\frac{q^{\frac{3}{4}i^2+j^2+2k^2+ij+2ik+2jk}u^iv^jw^k}{(q;q)_i(q^2;q^2)_j(q^2;q^2)_k}.
\end{align}
Note that $F_0(1,1,1)$ and $F_1(1,1,1)$ are evaluated in \eqref{conj-3-even-1} and \eqref{conj-3-odd-1}, respectively. Substituting \eqref{conj-3-even-1} and \eqref{conj-3-odd-1} into \eqref{exam7-F-split}, we obtain \eqref{eq-thm7-1}.

Similarly, substituting \eqref{conj-3-even-2} and \eqref{conj-3-odd-2} into \eqref{exam7-F-split}, we obtain \eqref{eq-thm7-2}. Substituting \eqref{conj-3-even-3} and \eqref{conj-3-odd-3} into \eqref{exam7-F-split}, we obtain \eqref{eq-thm7-3}.
\end{proof}

To give a companion to the identities \eqref{eq-thm7-1}--\eqref{eq-thm7-3} (see Theorem \ref{thm-exam7-add} below), we first establish the following identities.
\begin{theorem}\label{ex8-new}
    We have
    \begin{align}
        &\sum_{i,j,k\ge 0}\frac{q^{3i^2+j^2+2k^2+2ij+4ik+2jk+i+j}}{(q;q)_{2i}(q^2;q^2)_j(q^2;q^2)_k}
        +\sum_{i,j,k\ge 0}\frac{q^{3i^2+j^2+2k^2+2ij+4ik+2jk+7i+3j+4k+3}}{(q;q)_{2i+1}(q^2;q^2)_j(q^2;q^2)_k}\nonumber \\
        &=\frac{(-q^2;q^2)_\infty(-q^3,q^6,-q^9;-q^9)_\infty}{(q^2;q^2)_\infty}, \label{ex8-new-1}
        \\
        &\sum_{i,j,k\ge 0}\frac{q^{3i^2+j^2+2k^2+2ij+4ik+2jk+4i+2j+2k}}{(q;q)_{2i}(q^2;q^2)_j(q^2;q^2)_k}
        +\sum_{i,j,k\ge 0}\frac{q^{3i^2+j^2+2k^2+2ij+4ik+2jk+4i+2j+2k+1}}{(q;q)_{2i+1}(q^2;q^2)_j(q^2;q^2)_k}\nonumber \\
        &=\frac{(-q;q^2)_\infty
        (q^{12},q^{24},q^{36};q^{36})_\infty}{(q^2;q^2)_\infty}.\label{ex8-new-2}
    \end{align}
\end{theorem}
\begin{proof}
From \eqref{eq-78-transfer-1} and \eqref{eq-78-transfer-2}  we have
    \begin{align}
        &\sum_{i,j,k\ge 0}\frac{q^{3i^2+j^2+2k^2+2ij+4ik+2jk+i+j}}{(q;q)_{2i}(q^2;q^2)_j(q^2;q^2)_k}
        +\sum_{i,j,k\ge 0}\frac{q^{3i^2+j^2+2k^2+2ij+4ik+2jk+7i+3j+4k+3}}{(q;q)_{2i+1}(q^2;q^2)_j(q^2;q^2)_k} \nonumber \\
        &=\sum_{i,j,k\ge 0}\frac{q^{5i^2+j^2+2k^2+2ij+4ik+j}}{(q;q)_{2i}(q^2;q^2)_j(q^2;q^2)_k}
        +\sum_{i,j,k\ge 0}\frac{q^{5i^2+j^2+2k^2+2ij+4ik+8i+3j+4k+3}}{(q;q)_{2i+1}(q^2;q^2)_j(q^2;q^2)_k} \nonumber \\
        &=\sum_{i,j,k\ge 0}\frac{q^{5i^2+j^2+2k^2+2ij+4ik-2i+j}}{(q;q)_{2i}(q^2;q^2)_j(q^2;q^2)_k}.\label{ex8-new-proof-1}
    \end{align}
See the remark below for explanation of the last equality. Now by \eqref{ex27-1-1} and \eqref{ex8-new-proof-1} we prove \eqref{ex8-new-1}.

Similarly, from \eqref{eq-78-transfer-1} and \eqref{eq-78-transfer-2} we have
    \begin{align}
        &\sum_{i,j,k\ge 0}\frac{q^{3i^2+j^2+2k^2+2ij+4ik+2jk+4i+2j+2k}}{(q;q)_{2i}(q^2;q^2)_j(q^2;q^2)_k}
        +\sum_{i,j,k\ge 0}\frac{q^{3i^2+j^2+2k^2+2ij+4ik+2jk+4i+2j+2k+1}}{(q;q)_{2i+1}(q^2;q^2)_j(q^2;q^2)_k} \nonumber \\
        &=\sum_{i,j,k\ge 0}\frac{q^{5i^2+j^2+2k^2+2ij+4ik+3i+2j+2k}}{(q;q)_{2i}(q^2;q^2)_j(q^2;q^2)_k}
        +\sum_{i,j,k\ge 0}\frac{q^{3i^2+j^2+2k^2+2ij+4ik+5i+2j+2k+1}}{(q;q)_{2i+1}(q^2;q^2)_j(q^2;q^2)_k} \nonumber \\
        &=\sum_{i,j,k\ge 0}\frac{q^{5i^2+j^2+2k^2+2ij+4ik+3i+2j+2k}}{(q;q)_{2i+1}(q^2;q^2)_j(q^2;q^2)_k}. \label{ex8-new-proof-2}
    \end{align}
 Now by \eqref{ex27-1-2} and \eqref{ex8-new-proof-2} we prove \eqref{ex8-new-2}.
\end{proof}

\begin{rem}
    We give more details on the last equalities in \eqref{ex8-new-proof-1} and \eqref{ex8-new-proof-2}. Let $A_1(i)$ and $A_2(i)$ be any sequences such that
    \begin{align}
        W_1(u)=\sum_{i\ge 0}\frac{u^{i}A_1(i)}{(q;q)_{2i}}
        \quad \text{and} \quad
        W_2(u)=\sum_{i\ge 0}\frac{u^{i}A_2(i)}{(q;q)_{2i+1}}
    \end{align}
    converge. Then it is easy to see that
    \begin{align}
        &W_1(u)-W_1(uq^2)=\sum_{i\ge 0}\frac{u^{i+1}A_1(i+1)}{(q;q)_{2i+1}}, \label{G-diff-1}\\
        &W_2(u)-qW_2(uq^2)=\sum_{i\ge 0}\frac{u^iA_2(i)}{(q;q)_{2i}}.\label{G-diff-2}
    \end{align}
    If we set
    \begin{align*}
        A_1(i)=\sum_{i,j,k\ge 0}\frac{q^{5i^2+j^2+2k^2+2ij+4ik-2i+j}}{(q;q)_{2i}(q^2;q^2)_j(q^2;q^2)_k}, \quad
        A_2(i)=\sum_{i,j,k\ge 0}\frac{q^{5i^2+j^2+2k^2+2ij+4ik+3i+2j+2k}}{(q;q)_{2i+1}(q^2;q^2)_j(q^2;q^2)_k},
    \end{align*}
    then we obtain the last equalities in \eqref{ex8-new-proof-1} and \eqref{ex8-new-proof-2} immediately by setting $u=1$ in \eqref{G-diff-1} and \eqref{G-diff-2}, respectively.
\end{rem}

\begin{theorem}\label{thm-exam7-add}
    We have
    \begin{align}\label{eq-exam7-add}
        &\sum_{i,j,k\ge 0}\frac{q^{3i^2+4j^2+8k^2+4ij+8ik+8jk+2i+4j}(1+q^{6i+4j+8k+1})}{(q^4;q^4)_i(q^8;q^8)_j(q^8;q^8)_k}\nonumber
        \\
        &=\frac{(-q^8;q^8)_\infty(-q^{12},q^{24},-q^{36};-q^{36})_\infty}{(q^8;q^8)_\infty}+q\frac{(-q^4;q^8)_\infty(q^{48},q^{96},q^{144};q^{144})}{(q^8;q^8)_\infty}.
    \end{align}
\end{theorem}
\begin{proof}
Note that the left side of \eqref{eq-exam7-add} is
\begin{align}
&F(q^2,q^4,1;q^4)+qF(q^8,q^8,q^8;q^4)  \\
&=\left(F_0(q^2,q^4,1;q^4)+qF_1(q^8,q^8,q^8;q^4)\right) +\left(qF_0(q^8,q^8,q^8;q^4)+F_1(q^2,q^4,1;q^4)\right). \nonumber
\end{align}
Substituting \eqref{ex8-new-1} and \eqref{ex8-new-2} into it, we obtain \eqref{eq-exam7-add}.
\end{proof}

\subsubsection{Example 8}
This example corresponds to
\begin{align*}
A=\begin{pmatrix}
1 & 0 & 1\\
0 & 2 & 2\\
1/2 & 1 & 5/2
\end{pmatrix},
AD=\begin{pmatrix}
2 & 0 & 1\\
0 & 4 & 2\\
1 & 2 & 5/2
\end{pmatrix},
b\in  \bigg\{
\begin{pmatrix}
0 \\ 0 \\ -1/2
\end{pmatrix},
\begin{pmatrix}
1 \\ 0 \\ -1
\end{pmatrix},
\begin{pmatrix}
1 \\ 0 \\ 0
\end{pmatrix},
\begin{pmatrix}
1 \\ 2 \\ 1
\end{pmatrix}
\bigg\}.
\end{align*}
To prove its modularity, we first prove the following result mentioned in the previous example.
\begin{theorem}\label{thm-8-parity}
We have
\begin{align}
    &\sum_{i,j,k\ge 0}\frac{q^{5i^2+j^2+2k^2+2ij+4ik-2i+j}}{(q;q)_{2i}(q^2;q^2)_j(q^2;q^2)_k}
    =\frac{(-q^2;q^2)_\infty(-q^3,q^6,-q^9;-q^9)_\infty}{(q^2;q^2)_\infty},\label{ex27-1-1}
    \\
    &\sum_{i,j,k\ge 0}\frac{q^{5i^2+j^2+2k^2+2ij+4ik+3i+2j+2k}}{(q;q)_{2i+1}(q^2;q^2)_j(q^2;q^2)_k}
    =
    \frac{(-q;q^2)_\infty(q^{12},q^{24},q^{36};q^{36})_\infty}{(q^2;q^2)_\infty},\label{ex27-1-2}
    \\
     &\sum_{i,j,k\ge 0}\frac{q^{5i^2+j^2+2k^2+2ij+4ik-i}}{(q;q)_{2i}(q^2;q^2)_j(q^2;q^2)_k}
    =
    \frac{(-q;q^2)_\infty(q^{16},q^{20},q^{36};q^{36})_\infty}{(q^2;q^2)_\infty},\label{ex27-2-1}
    \\
     &\sum_{i,j,k\ge 0}\frac{q^{5i^2+j^2+2k^2+2ij+4ik+4i+j+2k}}{(q;q)_{2i+1}(q^2;q^2)_j(q^2;q^2)_k}
    =\frac{(-q^2;q^2)_\infty(-q,q^8,-q^9;-q^9)_\infty}{(q^2;q^2)_\infty},\label{ex27-2-2}
    \\
    &\sum_{i,j,k\ge 0}\frac{q^{5i^2+j^2+2k^2+2ij+4ik+j}}{(q;q)_{2i}(q^2;q^2)_j(q^2;q^2)_k}
    =\frac{(-q^2;q^2)_\infty(q^4,-q^5,-q^9;-q^9)_\infty}{(q^2;q^2)_\infty},\label{ex27-3-1}
    \\
      &\sum_{i,j,k\ge 0}\frac{q^{5i^2+j^2+2k^2+2ij+4ik+5i+2j+2k}}{(q;q)_{2i+1}(q^2;q^2)_j(q^2;q^2)_k}
    =\frac{(-q,q^2)_\infty(q^{8},q^{28},q^{36};q^{36})_\infty}{(q^2;q^2)_\infty},\label{ex27-3-2}
    \\
    &\sum_{i,j,k\ge 0}\frac{q^{5i^2+j^2+2k^2+2ij+4ik+2i+j+2k}}{(q;q)_{2i}(q^2;q^2)_j(q^2;q^2)_k}
    =
    \frac{(-q^2;q^2)_\infty(q^2,-q^7,-q^9;-q^9)_\infty}{(q^2;q^2)_\infty},\label{ex27-4-1}    \\
    &\sum_{i,j,k\ge 0}\frac{q^{5i^2+j^2+2k^2+2ij+4ik+7i+2j+4k}}{(q;q)_{2i+1}(q^2;q^2)_j(q^2;q^2)_k}
    =\frac{(-q;q^2)_\infty(q^{4},q^{32},q^{36};q^{36})_\infty}{(q^2,q^2)_\infty}.\label{ex27-4-2}
\end{align}
\end{theorem}
\begin{proof}
For $\sigma=0,1$ we let
\begin{align}\label{proof-exam8-Fc-start}
    &F_\sigma(u,v,w)=F_\sigma(u,v,w;q):=\sum_{i,j,k\ge 0}\frac{q^{\frac{5}{4}(2i+\sigma)^2+j^2+2k^2+(2i+\sigma)j+2(2i+\sigma)k}u^{2i+\sigma}v^jw^k}{(q;q)_{2i+\sigma}(q^2;q^2)_j(q^2;q^2)_k}\nonumber\\
    &=q^{\frac{5}{4}\sigma^2}\sum_{i,k\ge 0}\frac{q^{5i^2+2k^2+4ik+5\sigma i+2\sigma k}u^{2i+\sigma}w^k}{(q;q)_{2i+\sigma}(q^2;q^2)_k}\sum_{j\ge 0}\frac{q^{j^2+(2i+\sigma)j}v^j}{(q^2;q^2)_j} \nonumber \\
    &=q^{\frac{5}{4}\sigma^2}\sum_{i,k\ge 0}\frac{q^{5i^2+2k^2+4ik+5\sigma i+2\sigma k}u^{2i+\sigma}w^k(-q^{2i+\sigma+1}v;q^2)_\infty}{(q;q)_{2i+\sigma}(q^2;q^2)_k} \nonumber\\
    &=q^{\frac{5}{4}\sigma^2}u^{\sigma}(-q^{\sigma+1}v;q^2)_\infty\sum_{i,k\ge 0}\frac{q^{5i^2+2k^2+4ik+5\sigma i+2\sigma k}u^{2i}w^k}{(q;q)_{2i+\sigma}(q^2;q^2)_k(-q^{\sigma+1}v;q^2)_i}.
\end{align}
(1)  From \eqref{proof-exam8-Fc-start} we have
\begin{align}
    &F_0(q^{-1},q,1)=(-q^2;q^2)_\infty\sum_{i,k\ge 0}\frac{q^{5i^2+2k^2+4ik-2i}}{(q;q)_{2i}(q^2;q^2)_k(-q^2;q^2)_i} \nonumber \\
    &=(-q^2;q^2)_\infty\sum_{m\ge 0}q^{2m^2}\sum_{i=0}^m\frac{q^{3i^2-2i}}{(q;q^2)_{i}(q^2;q^2)_{m-i}(q^4;q^4)_i}. \label{proof-exam8-1-F}
\end{align}
Substituting the Bailey pair \eqref{G(4.1)} into \eqref{id-BP-twice} and then replacing $q$ by $q^2$, we deduce that
\begin{align}
&\sum_{m\ge 0}q^{2m^2}\sum_{i=0}^m\frac{(-1)^iq^{3i^2-2i}}{(-q;q^2)_{i}(q^2;q^2)_{m-i}(q^4;q^4)_i} \nonumber \\
    &=\frac{1}{(q^2;q^2)_\infty}\Big(1+\sum_{n\ge 1}q^{\frac{9}{2}n^2-\frac{3}{2}n}(1+q^{3n})\Big) \nonumber \\
    &=\frac{1}{(q^2;q^2)_\infty}(q^3,q^6,q^9;q^9)_\infty.  \quad (\text{by \eqref{Jacobi}}) \label{proof-exam8-1-F-result}
\end{align}
Substituting \eqref{proof-exam8-1-F-result} with $q$ replaced by $-q$ into \eqref{proof-exam8-1-F}, we obtain \eqref{ex27-1-1}.

(2) Substituting the Bailey pair \eqref{C(6)} into \eqref{id-BP-twice} and then replacing $q$ by $q^2$, we deduce from \eqref{proof-exam8-Fc-start} that
\begin{align*}
    &q^{-\frac{1}{4}}F_1(q^{-1},q,1)=(-q^3;q^2)_\infty\sum_{i,k\ge 0}\frac{q^{5i^2+2k^2+4ik+3i+2k}}{(q;q)_{2i+1}(q^2;q^2)_k(-q^3;q^2)_i}\\
    &=\frac{(-q^3;q^2)_\infty}{1-q}\sum_{m\ge 0}q^{2m^2+2m}\sum_{i=0}^m\frac{q^{3i^2+i}}{(q^2;q^2)_{i}(q^2;q^2)_{m-i}(q^6;q^4)_i}\\
    &=\frac{(-q;q^2)_\infty}{(q^2;q^2)_\infty}\Big(\sum_{n\ge 0} (-1)^nq^{18n^2+6n}+\sum_{n\ge 0}(-1)^{n+1}q^{18n^2+30n+12}  \Big)\\
    &=\frac{(-q;q^2)_\infty}{(q^2;q^2)_\infty}(q^{12},q^{24},q^{36};q^{36})_\infty.  \quad (\text{by \eqref{Jacobi}})
\end{align*}
This proves \eqref{ex27-1-2}.

(3) Substituting the Bailey pair \eqref{C(5)} into \eqref{id-BP-twice} and then replacing $q$ by $q^2$, we deduce from \eqref{proof-exam8-Fc-start} that
\begin{align*}
    &F_0(q^{-\frac{1}{2}},1,1)=(-q;q^2)_\infty\sum_{i,k\ge 0}\frac{q^{5i^2+2k^2+4ik-i}}{(q;q)_{2i}(q^2;q^2)_k(-q;q^2)_i}
    \\
    &=(-q;q^2)_\infty\sum_{m\ge 0}q^{2m^2}\sum_{i=0}^m\frac{q^{3i^2-i}}{(q^2;q^2)_{i}(q^2;q^2)_{m-i}(q^2;q^4)_i}\\
    &=\frac{(-q;q^2)_\infty}{(q^2;q^2)_\infty}\Big(1+\sum_{n\ge 1}(-1)^nq^{18n^2}(q^{2n}+q^{-2n})\Big)\\
    &=\frac{(-q;q^2)_\infty}{(q^2;q^2)_\infty}(q^{16},q^{20},q^{36};q^{36})_\infty.  \quad (\text{by \eqref{Jacobi}})
\end{align*}
This proves \eqref{ex27-2-1}.

(4) From \eqref{proof-exam8-Fc-start} we have
\begin{align}
    &q^{-\frac{3}{4}}F_1(q^{-\frac{1}{2}},1,1)=(-q^2;q^2)_\infty\sum_{i,k\ge 0}\frac{q^{5i^2+2k^2+4ik+4i+2k}}{(q;q)_{2i+1}(q^2;q^2)_k(-q^2;q^2)_i}
   \nonumber \\
    &=\frac{(-q^2;q^2)_\infty}{1-q}\sum_{m\ge 0}q^{2m^2+2m}\sum_{i=0}^m\frac{q^{3i^2+2i}}{(q^3;q^2)_{i}(q^2;q^2)_{m-i}(q^4;q^4)_i}. \label{proof-exam8-4-F}
\end{align}
Substituting the Bailey pair \eqref{G(5)} into \eqref{id-BP-twice} and then replacing $q$ by $q^2$, we obtain
\begin{align}
&\sum_{m\ge 0}q^{2m^2+2m}\sum_{i=0}^m\frac{(-1)^iq^{3i^2+2i}}{(-q^3;q^2)_{i}(q^2;q^2)_{m-i}(q^4;q^4)_i}
 \nonumber   \\
    &=\frac{1}{(q^4;q^2)}\sum_{n\ge 0} (-1)^n\frac{q^{\frac{9}{2}n^2+\frac{7}{2}n}(1-q^{2n+1})}{1-q}  \nonumber \\
    &=\frac{1+q}{(q^2;q^2)_\infty}(q,q^8,q^9;q^9)_\infty.  \quad (\text{by \eqref{Jacobi}}) \label{proof-exam8-4-F-result}
\end{align}
Substituting \eqref{proof-exam8-4-F-result} with $q$ replaced by $-q$ into \eqref{proof-exam8-4-F}, we obtain  \eqref{ex27-2-2}.

(5) From \eqref{proof-exam8-Fc-start} we have
\begin{align}
    &F_0(1,q,1)=(-q^2;q^2)_\infty\sum_{i,k\ge 0}\frac{q^{5i^2+2k^2+4ik}}{(q;q)_{2i}(q^2;q^2)_k(-q^2;q^2)_i} \nonumber \\
    &=(-q^2;q^2)_\infty\sum_{m\ge 0}q^{2m^2}\sum_{i=0}^m\frac{q^{3i^2}}{(q;q^2)_{i}(q^2;q^2)_{m-i}(q^4;q^4)_i}. \label{proof-exam8-5-F}
\end{align}
Substituting the Bailey pair \eqref{G(4)} into \eqref{id-BP-twice} and then replacing $q$ by $q^2$, we deduce that
\begin{align}
    &\sum_{m\ge 0}q^{2m^2}\sum_{i=0}^m\frac{(-1)^iq^{3i^2}}{(-q;q^2)_{i}(q^2;q^2)_{m-i}(q^4;q^4)_i} \nonumber \\
    &=\frac{1}{(q^2;q^2)_\infty}\Big(1+\sum_{n\ge 1} (-1)^nq^{\frac{9}{2}n^2-\frac{1}{2}n}\frac{1-q^{2n+1}}{1-q} \Big) \nonumber \\
    &=\frac{(q^4,q^5,q^9;q^9)_\infty}{(q^2;q^2)_\infty}.  \quad (\text{by \eqref{Jacobi}}) \label{proof-exam8-5-F-result}
\end{align}
Substituting \eqref{proof-exam8-5-F-result} with $q$ replaced by $-q$ into \eqref{proof-exam8-5-F}, we obtain \eqref{ex27-3-1}.

(6) Substituting the Bailey pair \eqref{C(7)} into \eqref{id-BP-twice} and then replacing $q$ by $q^2$, we deduce from \eqref{proof-exam8-Fc-start} that
\begin{align*}
    &q^{-\frac{5}{4}}F_1(1,q,1)=(-q^3;q^2)_\infty\sum_{i,k\ge 0}\frac{q^{5i^2+2k^2+4ik+5i+2k}}{(q;q)_{2i+1}(q^2;q^2)_k(-q^3;q^2)_i}\\
    &=\frac{(-q^3;q^2)_\infty}{1-q}\sum_{m\ge 0}q^{2m^2+2m}\sum_{i=0}^m\frac{q^{3i^2+3i}}{(q^2;q^2)_{i}(q^2;q^2)_{m-i}(q^6;q^4)_i}\\
    &=\frac{(-q^3;q^2)_\infty}{(1-q)(q^4;q^2)_\infty}\sum_{n\ge 0} \Big((-1)^nq^{18n^2+10n}+(-1)^{n+1}q^{18n^2+26n+8} \Big)\\
    &=\frac{(-q;q^2)_\infty}{(q^2;q^2)_\infty}(q^8,q^{28},q^{36};q^{36})_\infty.  \quad (\text{by \eqref{Jacobi}})
\end{align*}
This proves \eqref{ex27-3-2}.

(7)  From \eqref{proof-exam8-Fc-start} we have
\begin{align}
    & F_0(q,q,q^2)=(-q^2;q^2)_\infty\sum_{i,k\ge 0}\frac{q^{5i^2+2k^2+4ik+2i+2k}}{(q;q)_{2i}(q^2;q^2)_k(-q^2;q^2)_i}
    \nonumber \\
    &=(-q^2;q^2)_\infty\sum_{m\ge 0}q^{2m^2+2m}\sum_{i=0}^m\frac{q^{3i^2}}{(q;q^2)_{i}(q^2;q^2)_{m-i}(q^4;q^4)_i}. \label{proof-exam8-7-F}
\end{align}
Substituting the Bailey pair \eqref{G4.1-new} into \eqref{id-BP-twice} and then replacing $q$ by $q^2$, we deduce that
\begin{align}
     &\sum_{m\ge 0}q^{2m^2+2m}\sum_{i=0}^m\frac{(-1)^iq^{3i^2}}{(-q;q^2)_{i}(q^2;q^2)_{m-i}(q^4;q^4)_i}
    \\
    &=\frac{1}{(q^4;q^2)_\infty}\sum_{n\ge 0}(-1)^nq^{\frac{9}{2}n^2+\frac{5}{2}n}\frac{1-q^{4n+2}}{1-q^2} \nonumber \\
    &=\frac{(q^2,q^7,q^9;q^9)_\infty}{(q^2;q^2)_\infty}.  \quad (\text{by \eqref{Jacobi}}) \label{proof-exam8-7-F-result}
\end{align}
Substituting \eqref{proof-exam8-7-F-result} with $q$ replaced by $-q$ into \eqref{proof-exam8-7-F}, we obtain \eqref{ex27-4-1}.

(8) Substituting the Bailey pair \eqref{C7*} into \eqref{id-BP-twice} and then replacing $q$ by $q^2$, we deduce from \eqref{proof-exam8-Fc-start} that
\begin{align*}
    &q^{-\frac{9}{4}} F_1(q,q,q^2)=(-q^3;q^2)_\infty\sum_{i,k\ge 0}\frac{q^{5i^2+2k^2+4ik+7i+4k}}{(q;q)_{2i+1}(q^2;q^2)_k(-q^3;q^2)_i}
    \\
    &=(-q^3;q^2)_\infty\sum_{m\ge 0}q^{2m^2+4m}\sum_{i=0}^m\frac{q^{3i^2+3i}}{(q;q)_{2i+1}(q^2;q^2)_{m-i}(-q^3;q^2)_i}\\
    &=\frac{(-q^3;q^2)_\infty}{(1-q)(q^6;q^2)_\infty}\sum_{n\ge 0} (-1)^nq^{18n^2+14n} \times \frac{1-q^{8n+4}}{1-q^4}\\
    &=\frac{(-q;q^2)_\infty}{(q^2;q^2)_\infty}(q^4,q^{32},q^{36};q^{36}). \quad (\text{by \eqref{Jacobi}})
\end{align*}
This proves \eqref{ex27-4-2}.
\end{proof}

\begin{theorem}\label{thm-8}
We have
\begin{align}
&\sum_{i,j,k\ge 0}\frac{q^{5i^2+4j^2+8k^2+4ij+8ik-4i+4j}}{(q^4;q^4)_i(q^8;q^8)_j(q^8;q^8)_k}
    =\frac{J_2^2J_3J_{24}}{J_1J_4J_6J_8}, \label{add-thm8-1} \\
    &\sum_{i,j,k\ge 0}\frac{q^{5i^2+4j^2+8k^2+4ij+8ik-2i}}{(q^4;q^4)_i(q^8;q^8)_j(q^8;q^8)_k}
    =
    \frac{J_{64,144}}{J_{4,16}}
    +q^3\frac{J_{32}J_{72}^4J_{144}}{J_{8,32}J_{4,72}J_{16,72}J_{24,72}J_{36,144}J_{64,144}}, \label{add-thm8-2}
    \\
    &\sum_{i,j,k\ge 0}\frac{q^{5i^2+4j^2+8k^2+4ij+8ik+4j}}{(q^4;q^4)_i(q^8;q^8)_j(q^8;q^8)_k}=
    q^5\frac{J_{32,144}}{J_{4,16}}
    +\frac{J_{32}J_{72}^4J_{144}}{J_{8,32}J_{8,72}J_{20,72}J_{24,72}J_{36,144}J_{32,144}}, \label{add-thm8-3}
    \\
    &\sum_{i,j,k\ge 0}\frac{q^{5i^2+4j^2+8k^2+4ij+8ik+4i+4j+8k}}{(q^4;q^4)_i(q^8;q^8)_j(q^8;q^8)_k}=
    q^9\frac{J_{16,144}}{J_{4,16}}
    +\frac{J_{32}J_{72}^4J_{144}}{J_{8,32}J_{24,72}J_{28,72}J_{32,72}J_{36,144}J_{16,144}}. \label{add-thm8-4}
\end{align}
\end{theorem}
Note that the product sides of \eqref{add-thm8-2}--\eqref{add-thm8-4} are the same with the product sides of \eqref{eq-thm7-1}--\eqref{eq-thm7-3}.
\begin{proof}
Let
\begin{align}\label{proof-exam8-start}
    &F(u,v,w)=F(u,v,w;q):=\sum_{i,j,k\ge 0}\frac{q^{\frac{5}{4}i^2+j^2+2k^2+ij+2ik}u^{i}v^jw^k}{(q;q)_i(q^2;q^2)_j(q^2;q^2)_k}.
\end{align}
Then
$$F(u,v,w)=F_0(u,v,w)+F_1(u,v,w)$$
where $F_\sigma(u,v,w)$ ($\sigma=0,1$) was defined in \eqref{proof-exam8-Fc-start}. Utilizing Theorem \ref{thm-8-parity} and the method in \cite{Frye-Garvan}, we obtain the desired identities.
\end{proof}

\subsection{Examples 2, 9 and 11}
For these three examples, we will use the following key identity to reduce triple sums to double sums: for $n\geq 0$ we have
\begin{align}
\sum_{i+2j=n} \frac{q^{i(i-1)/2}}{(q;q)_i(q^2;q^2)_j}&=\frac{1}{(q;q)_n}. \quad \text{(\cite[Eq.\ (2.4)]{WW-I})} \label{lem-12}
\end{align}
\subsubsection{Example 2}
This example corresponds to
\begin{align*}
A=\begin{pmatrix}
1 & 1 & 1\\
1 & 2 & 2\\
1/2 & 1 & 2
\end{pmatrix},\quad
AD=\begin{pmatrix}
2 & 2 & 1\\
2 & 4 & 2\\
1 & 2 & 2
\end{pmatrix},\quad
b\in  \bigg\{
\begin{pmatrix}
0 \\ 0 \\-1/2
\end{pmatrix},
\begin{pmatrix}
1 \\ 0 \\-1/2
\end{pmatrix},
\begin{pmatrix}
1 \\ 2 \\1/2
\end{pmatrix}
\bigg\}.
\end{align*}
The Nahm series involved here will be connected to the following identities in the work of Li and Wang \cite{LW}:
\begin{align}
&\sum_{i,j\ge 0}\frac{q^{i^2+2j^2+2ij}}{(q^2;q^2)_i(q^4;q^4)_j}
    =
    \frac{(-q;q^2)_\infty(q^3,q^4,q^7;q^7)_\infty}{(q^2;q^2)_\infty}, \quad \text{(\cite[Eq.\ (3.42)]{LW})} \label{ZL.3.28}
    \\
    &\sum_{i,j\ge 0}\frac{q^{i^2+2j^2+2ij+2j}}{(q^2;q^2)_i(q^4;q^4)_j}
    =
    \frac{(-q;q^2)_\infty(q^2,q^5,q^7;q^7)_\infty}{(q^2;q^2)_\infty}, \quad \text{(\cite[Eq.\ (3.43)]{LW})}\label{ZL.3.29}
    \\
    &\sum_{i,j\ge 0}\frac{q^{i^2+2j^2+2ij+2i+2j}}{(q^2;q^2)_i(q^4;q^4)_j}
    =
    \frac{(-q;q^2)_\infty(q,q^6,q^7;q^7)_\infty}{(q^2;q^2)_\infty}. \quad \text{(\cite[Eq.\ (3.44)]{LW})}\label{ZL.3.30}
\end{align}
\begin{theorem}
We have
\begin{align}
\sum_{i,j,k\ge 0}\frac{q^{2i^2+4j^2+2k^2+4ij+2ik+4jk-i}}{(q^2;q^2)_i(q^4;q^4)_j(q^4;q^4)_k}
&= \frac{(-q;q^2)_\infty(q^3,q^4,q^7;q^7)_\infty}{(q^2;q^2)_\infty},
\label{table3.2.1}
\\
\sum_{i,j,k\ge 0}\frac{q^{2i^2+4j^2+2k^2+4ij+2ik+4jk-i+2k}}{(q^2;q^2)_i(q^4;q^4)_j(q^4;q^4)_k}
&= \frac{(-q;q^2)_\infty(q^2,q^5,q^7;q^7)_\infty}{(q^2;q^2)_\infty},
\label{table3.2.2}
\\
\sum_{i,j,k\ge 0}\frac{q^{2i^2+4j^2+2k^2+4ij+2ik+4jk+i+4j+2k}}{(q^2;q^2)_i(q^4;q^4)_j(q^4;q^4)_k}
&= \frac{(-q;q^2)_\infty(q,q^6,q^7;q^7)_\infty}{(q^2;q^2)_\infty}.
\label{table3.2.3}
\end{align}
\end{theorem}

\begin{proof}
We have
\begin{align}
&\sum_{i,j,k\ge 0}\frac{q^{2i^2+4j^2+2k^2+4ij+2ik+4jk-i}u^{i+2j}w^k}{(q^2;q^2)_i(q^4;q^4)_j(q^4;q^4)_k} \notag
\\
&=\sum_{k\ge 0}\frac{q^{2k^2}w^k}{(q^4;q^4)_k}
\sum_{m\ge 0}q^{m^2+2km}u^m
\sum_{i+2j=m}\frac{q^{i^2-i}}{(q^2;q^2)_i(q^4;q^4)_j} \notag
\\
&=\sum_{m,k\ge 0}\frac{q^{m^2+2km+2k^2}u^mw^k}{(q^2;q^2)_m(q^4;q^4)_k}.   \quad \text{(by \eqref{lem-12})}
\label{F3.2}
\end{align}
Setting $(u,w)=(1,1)$, $(1,q^2)$ and $(q^2,q^2)$ in   \eqref{F3.2}, and then using \eqref{ZL.3.28}--\eqref{ZL.3.30} we obtain \eqref{table3.2.1}--\eqref{table3.2.3}, respectively.
\end{proof}

\subsubsection{Example 9}
This example corresponds to
\begin{align*}
A=\begin{pmatrix}
1 & 1 & 1\\
1 & 3 & 3\\
1/2&3/2&5/2
\end{pmatrix},
AD=\begin{pmatrix}
2 & 2 & 1\\
2 & 6 & 3\\
1 & 3&5/2
\end{pmatrix},
b\in  \bigg\{
\begin{pmatrix}
0 \\ -1 \\ -1
\end{pmatrix},
\begin{pmatrix}
1 \\ -2 \\ -3/2
\end{pmatrix},
\begin{pmatrix}
1 \\ 0 \\ -1/2
\end{pmatrix}
\bigg\}.
\end{align*}
To prove its modularity, we need the following identities:
\begin{align}
    &\sum_{i,j\ge 0}\frac{q^{3i^2+4j^2+4ij-2i}}{(q^4;q^4)_i(q^8;q^8)_j}
    =\frac{1}{(q,q^4;q^5)_\infty(-q^2;q^2)_\infty}, \quad \text{(\cite[Eq.\ (3.57)]{LW})}
    \label{ZL.3.57}
    \\
    &\sum_{i,j\ge 0}\frac{q^{3i^2+4j^2+4ij-4i+4j+1}}{(q^4;q^4)_i(q^8;q^8)_j}
    =\frac{1}{(q,q^4;q^5)_\infty(-q^2;q^2)_\infty}, \quad \text{(\cite[Eq.\ (3.58)]{LW})}
    \label{ZL.3.58}
    \\
    &\sum_{i,j\ge 0}\frac{q^{3i^2+4j^2+4ij+4j}}{(q^4;q^4)_i(q^8;q^8)_j}
    =\frac{1}{(q^2,q^3;q^5)_\infty(-q^2;q^2)_\infty}. \quad \text{(\cite[Eq.\ (3.59)]{LW})}
    \label{ZL.3.59}
\end{align}
\begin{theorem}
We have
\begin{align}
\sum_{i,j,k\ge 0}\frac{q^{5i^2+12j^2+4k^2+12ij+4ik+8jk-4i-4j}}{(q^4;q^4)_i(q^8;q^8)_j(q^8;q^8)_k}
&=\frac{1}{(q,q^4;q^5)_\infty(-q^2;q^2)_\infty},
\label{table3.9.1}
\\
\sum_{i,j,k\ge 0}\frac{q^{5i^2+12j^2+4k^2+12ij+4ik+8jk-6i-8j+4k+1}}{(q^4;q^4)_i(q^8;q^8)_j(q^8;q^8)_k}
&=\frac{1}{(q,q^4;q^5)_\infty(-q^2;q^2)_\infty},
\label{table3.9.2}
\\
\sum_{i,j,k\ge 0}\frac{q^{5i^2+12j^2+4k^2+12ij+4ik+8jk-2i+4k}}{(q^4;q^4)_i(q^8;q^8)_j(q^8;q^8)_k}
&=\frac{1}{(q^2,q^3;q^5)_\infty(-q^2;q^2)_\infty}.
\label{table3.9.3}
\end{align}
\end{theorem}

\begin{proof}
We have
\begin{align}
&\sum_{i,j,k\ge 0}\frac{q^{5i^2+12j^2+4k^2+12ij+4ik+8jk-2i}u^{i+2j}w^{k}}{(q^4;q^4)_i(q^8;q^8)_j(q^8;q^8)_k}\notag
\\
&=
\sum_{k\ge 0}\frac{q^{4k^2}w^k}{(q^8;q^8)_k}
\sum_{m\ge 0}q^{3m^2+4km}u^m
\sum_{i+2j=m}\frac{q^{2i^2-2i}}{(q^4;q^4)_i(q^8;q^8)_j}\notag
\\
&=\sum_{m,k\ge 0}\frac{q^{3m^2+4km+4k^2}u^mw^k}{(q^4;q^4)_m(q^8;q^8)_k}.  \quad \text{(by \eqref{lem-12})}
\label{F3.9}
\end{align}
Setting $(u,w)=(q^{-2},1)$, $(q^{-4},q^4)$ and $(1,q^4)$ in \eqref{F3.9} and  using \eqref{ZL.3.57}--\eqref{ZL.3.59}, we obtain \eqref{table3.9.1}--\eqref{table3.9.3}, respectively.
\end{proof}

\subsubsection{Example 11}
This example corresponds to
\begin{align*}
A=\begin{pmatrix}
4 & 4 & 4\\
4 & 6 & 6\\
2 & 3 & 4\\
\end{pmatrix},\quad
AD=\begin{pmatrix}
8 & 8 & 4\\
8 & 12 & 6\\
4 & 6 & 4\\
\end{pmatrix},\quad
b\in  \bigg\{
\begin{pmatrix}
0 \\ -1\\ -1
\end{pmatrix},
\begin{pmatrix}
2 \\ 1 \\ 0
\end{pmatrix}
\bigg\}.
\end{align*}
To prove its modularity, we need the following identities:
\begin{align}
    &\sum_{i,j\ge 0}\frac{q^{\frac{3}{2}i^2+4ij+4j^2-\frac{1}{2}i}}{(q;q)_i(q^2;q^2)_j}
    =
    \frac{1}{(q,q^4;q^5)_\infty},\quad \text{(\cite[Eq.\ (47)]{Mizuno})}
    \label{mizuno-47}
    \\
    &\sum_{i,j\ge 0}\frac{q^{\frac{3}{2}i^2+4ij+4j^2+\frac{1}{2}i+2j}}{(q;q)_i(q^2;q^2)_j}
    =
    \frac{1}{(q^2,q^3;q^5)_\infty}.\quad \text{(\cite[Eq.\ (48)]{Mizuno})}
    \label{mizuno-48}
\end{align}
\begin{theorem}
We have
\begin{align}
&\sum_{i,j,k\ge 0}\frac{q^{2i^2+6j^2+4k^2+6ij+4ik+8jk-i-j}}{(q;q)_i(q^2;q^2)_j(q^2;q^2)_k}
=\frac{1}{(q,q^4;q^5)_\infty},
\label{table3.11.1}
\\
&\sum_{i,j,k\ge 0}\frac{q^{2i^2+6j^2+4k^2+6ij+4ik+8jk+j+2k}}{(q;q)_i(q^2;q^2)_j(q^2;q^2)_k}
=\frac{1}{(q^2,q^3;q^5)_\infty}.
\label{table3.11.2}
\end{align}
\end{theorem}
\begin{proof}
We have
\begin{align}
&\sum_{i,j,k\ge 0}\frac{q^{2i^2+6j^2+4k^2+6ij+4ik+8jk-\frac{1}{2}i}u^{i+2j}w^{k}}{(q;q)_i(q^2;q^2)_j(q^2;q^2)_k}\notag
\\
&=\sum_{k\ge 0}\frac{q^{4k^2}w^k}{(q^2;q^2)_k}
\sum_{m\ge 0}q^{\frac{3}{2}m^2+4km}u^m
\sum_{i+2j=m}\frac{q^{\frac{1}{2}(i^2-i)}}{(q;q)_i(q^2;q^2)_j}\notag
\\
&=\sum_{m,k\ge 0}\frac{q^{\frac{3}{2}m^2+4mk+4k^2}u^mw^k}{(q;q)_m(q^2;q^2)_k}. \quad \text{(by  \eqref{lem-12})}
\label{F3.11}
\end{align}
Setting $(u,w)=(q^{-\frac{1}{2}},1)$ and $(q^{\frac{1}{2}},q^2)$ in \eqref{F3.11}, by \eqref{mizuno-47}  and  \eqref{mizuno-48} we obtain \eqref{table3.11.1} and \eqref{table3.11.2}, respectively.
\end{proof}

\subsection{Example 3}
This example corresponds to
\begin{align*}
A=\begin{pmatrix}
1 & 1 & 0\\
1 & 4 & 4\\
0 & 2 & 3
\end{pmatrix},\quad
AD=\begin{pmatrix}
2 & 2 & 0\\
2 & 8 & 4\\
0 & 4 & 3
\end{pmatrix},\quad
b\in  \bigg\{
\begin{pmatrix}
1 \\ 0 \\ -1/2
\end{pmatrix},
\begin{pmatrix}
1 \\ 4 \\3/2
\end{pmatrix}
\bigg\}.
\end{align*}
We will connect it with the following identities found in
the work of Kur\c{s}ung\"oz \cite[Corollary 18]{Kursungoz-JCTA}:
\begin{align}
    \sum_{i,j\ge 0}\frac{q^{(3i^2-i)/2+4ij+4j^2}}{(q;q)_i(q^4;q^4)_j}
    &=
    \frac{1}{(q,q^4,q^7;q^8)_\infty}, \label{Kur.cor18.1}
    \\
    \sum_{i,j\ge 0}\frac{q^{(3i^2+3i)/2+4ij+4j^2+4j}}{(q;q)_i(q^4;q^4)_j}
    &=
    \frac{1}{(q^3,q^4,q^5;q^8)_\infty}.\label{Kur.cor18.2}
\end{align}
\begin{theorem}
We have
\begin{align}
\sum_{i,j,k\ge 0}\frac{q^{\frac{3}{2}i^2+4j^2+k^2+4ij+2jk-\frac{1}{2}i+k}}{(q;q)_i(q^2;q^2)_j(q^2;q^2)_k}
&=\frac{(q^3,q^5,q^8;q^8)_\infty}{(q;q)_\infty},
\label{table3.3.1}
\\
\sum_{i,j,k\ge 0}\frac{q^{\frac{3}{2}i^2+4j^2+k^2+4ij+2jk+\frac{3}{2}i+4j+k}}{(q;q)_i(q^2;q^2)_j(q^2;q^2)_k}
&=\frac{(q,q^7,q^8;q^8)_\infty}{(q;q)_\infty}.
\label{table3.3.2}
\end{align}
\end{theorem}
\begin{proof}
We have
\begin{align}
&\sum_{i,j,k\ge 0}\frac{q^{\frac{3}{2}i^2+4j^2+k^2+4ij+2jk}u^iv^jw^k}{(q;q)_i(q^2;q^2)_j(q^2;q^2)_k}\notag
\\
&=\sum_{i,j\ge 0}\frac{q^{\frac{3}{2}i^2+4ij+4j^2}u^iv^j}{(q;q)_i(q^2;q^2)_j}
\sum_{k\ge 0}\frac{q^{k^2-k}(wq^{1+2j})^k}{(q^2;q^2)_k}\notag
\\
&=\sum_{i,j\ge 0}\frac{q^{\frac{3}{2}i^2+4ij+4j^2}u^iv^j(-wq^{1+2j};q^2)_\infty}{(q;q)_i(q^2;q^2)_j}\quad \text{(by \eqref{euler-2})}
\notag
\\
&=(-wq;q^2)_\infty\sum_{i,j\ge 0}\frac{q^{\frac{3}{2}i^2+4ij+4j^2}u^iv^j}{(q;q)_i(q^2;q^2)_j(-wq;q^2)_j}.
\label{F3.3}
\end{align}
Setting $(u,v,w)=(q^{-\frac{1}{2}},1,q)$ and $(q^{\frac{3}{2}},q^4,q)$ in \eqref{F3.3}, and then using \eqref{Kur.cor18.1} and \eqref{Kur.cor18.2}, we obtain \eqref{table3.3.1} and \eqref{table3.3.2}, respectively.
\end{proof}

\subsection{Example 4}
This example corresponds to
\begin{align*}
A=\begin{pmatrix}
2 & 1 & 0\\
1 & 2 & 2\\
0 & 1 & 2
\end{pmatrix},\quad
AD=\begin{pmatrix}
4 & 2 & 0\\
2 & 4 & 2\\
0 & 2 & 2
\end{pmatrix},\quad
b\in  \bigg\{
\begin{pmatrix}
0 \\ 1 \\ 0
\end{pmatrix},
\begin{pmatrix}
2 \\ 1 \\ 0
\end{pmatrix},
\begin{pmatrix}
2 \\ 3 \\ 1
\end{pmatrix}
\bigg\}.
\end{align*}

We will confirm its modularity in Theorem \ref{thm-exam-4}. Before that, we first establish the identities announced in Theorem \ref{thm-exam-4-relation}. We need some preparations before giving the proof of this theorem.

Following \cite[Eq.\ (2.5)]{Andrews1974PNAS} we define for $1\leq i\leq k$ and $k\geq 2$ that
\begin{align}
Q_{k,i}(x):=\sum_{n_1,n_2,\dots,n_{k-1}\geq 0} \frac{x^{N_1+N_2+\cdots+N_{k-1}}q^{N_1^2+N_2^2+\cdots+N_{k-1}^2+N_i+\cdots+N_{k-1}}}{(q;q)_{n_1}(q;q)_{n_2}\cdots (q;q)_{n_{k-1}}}
\end{align}
where as before, $N_j=n_j+n_{j+1}+\cdots+n_{k-1}$ for $j=1,2,\dots,k-1$ and $N_j=0$ for $j\geq k$. For convenience, we let $Q_{k,0}(x)=0$. It was proved that \cite[Eq.\ (2.2)]{Andrews1974PNAS}
\begin{align}
    Q_{k,i}(x)-Q_{k,i-1}(x)=(xq)^{i-1}Q_{k,k-i+1}(xq). \label{Q-general}
\end{align}
In particular, when $k=4$ and $i=1,2,3,4$ we have
\begin{align}
&Q_{4,1}(x)=Q_{4,4}(xq), \label{Q-relation-1} \\
&Q_{4,2}(x)-Q_{4,1}(x)=(xq)Q_{4,3}(xq), \label{Q-relation-2} \\
&Q_{4,3}(x)-Q_{4,2}(x)=(xq)^2Q_{4,2}(xq), \label{Q-relation-3} \\
&Q_{4,4}(x)-Q_{4,3}(x)=(xq)^3Q_{4,1}(xq). \label{Q-relation-4}
\end{align}
We consider the two series in Theorem \ref{thm-exam-4-relation}:
\begin{align}
&L_2(x):=\sum_{i,j,k\geq 0} \frac{q^{i^2+2j^2+2k^2+2ij+2jk+j+2k}}{(q;q)_i(q^2;q^2)_j(q^2;q^2)_k}x^{i+2j+2k}, \label{L2-defn} \\
&L_4(x):=\sum_{i,j,k\geq 0} \frac{q^{i^2+2j^2+2k^2+2ij+2jk+j}}{(q;q)_i(q^2;q^2)_j(q^2;q^2)_k}x^{i+2j+2k}. \label{L4-defn}
\end{align}

Motivated by \eqref{Q-relation-1} and \eqref{Q-relation-3}, we will further define two series $L_1(x)$ and $L_3(x)$ in the following lemma. We aim to show that $L_i(x)$ ($i=1,2,3,4$) also satisfy the system of equations \eqref{Q-relation-1}--\eqref{Q-relation-4}.
\begin{lemma}\label{lem-L-system}
Let
\begin{align}
    &L_1(x):=L_4(xq), \label{L-relation-1}\\
    &L_3(x):=L_2(x)+(xq)^2L_2(xq).  \label{L-relation-3}
\end{align}
We have
\begin{align}
&L_2(x)-L_1(x)=xqL_3(xq),  \label{L-relation-2}\\
&L_4(x)-L_3(x)=(xq)^3L_1(xq). \label{L-relation-4}
\end{align}
\end{lemma}
\begin{proof}
We denote
\begin{align}
    P(i,j,k)=(i+j)^2+(j+k)^2+k^2.
\end{align}
Note the following simple but useful fact:
\begin{align}
P(i+1,j,k)-P(i,j,k)&=2i+2j+1, \label{Q-diff-1} \\
P(i,j+1,k)-P(i,j,k)&=2i+4j+2k+2, \label{Q-diff-2} \\
P(i,j,k+1)-P(i,j,k)&=2j+4k+2. \label{Q-diff-3}
\end{align}
We have
\begin{align}
&L_4(x)-L_2(x)=\sum_{i,j,k\geq 0} \frac{q^{i^2+2j^2+2k^2+2ij+2jk+j}(1-q^{2k})}{(q;q)_i(q^2;q^2)_j(q^2;q^2)_k}x^{i+2j+2k} \nonumber \\
&=\sum_{i,j,k\geq 0} \frac{q^{P(i,j,k+1)+j}}{(q;q)_i(q^2;q^2)_j(q^2;q^2)_k}x^{i+2j+2(k+1)}  \nonumber \\
&=x^2q^2\sum_{i,j,k\geq 0} \frac{q^{(i+j)^2+(j+k)^2+k^2+3j+4k}}{(q;q)_i(q^2;q^2)_j(q^2;q^2)_k}x^{i+2j+2k}. \quad \text{(by \eqref{Q-diff-3})}
\end{align}
Hence we have
\begin{align}
&L_4(x)-L_2(x)-x^2q^2L_2(xq) \nonumber \\
&=x^2q^2\sum_{i,j,k\geq 0} \frac{q^{(i+j)^2+(j+k)^2+k^2+3j+4k}(1-q^i)}{(q;q)_i(q^2;q^2)_j(q^2;q^2)_k}x^{i+2j+2k} \nonumber \\
&=x^2q^2\sum_{i,j,k\geq 0} \frac{q^{P(i+1,j,k)+3j+4k}}{(q;q)_i(q^2;q^2)_j(q^2;q^2)_k}x^{i+1+2j+2k} \nonumber \\
&=x^3q^3\sum_{i,j,k\geq 0} \frac{q^{(i+j)^2+(j+k)^2+k^2+2i+5j+4k}}{(q;q)_i(q^2;q^2)_j(q^4;q^4)_k}x^{i+2j+2k} \quad \text{(by \eqref{Q-diff-1})} \nonumber \\
&=x^3q^3L_4(xq^2). \label{L1L2-1}
\end{align}
Using \eqref{L-relation-1} and \eqref{L-relation-3} we deduce from \eqref{L1L2-1} that
\begin{align}
    L_4(x)-L_3(x)=(xq)^3L_1(xq).
\end{align}
This proves \eqref{L-relation-4}.

It now remains to prove \eqref{L-relation-2}. By \eqref{L-relation-1} and \eqref{L-relation-3} it is equivalent to
\begin{align}\label{add-L-equiv}
L_2(x)-L_4(xq)=xqL_2(xq)+x^3q^5L_2(xq^2).
\end{align}
By definition we have
\begin{align}
&L_2(x)-L_4(xq)=\sum_{i,j,k\geq 0} \frac{q^{P(i,j,k)+j+2k}(1-q^{i+2j})}{(q;q)_i(q^2;q^2)_j(q^2;q^2)_k}x^{i+2j+2k} \\
&=\sum_{i,j,k\geq 0} \frac{q^{P(i,j,k)+j+2k}(1-q^i)}{(q;q)_i(q^2;q^2)_j(q^2;q^2)_k}x^{i+2j+2k} +\sum_{i,j,k\geq 0} \frac{q^{P(i,j,k)+j+2k}q^i(1-q^{2j})}{(q;q)_i(q^2;q^2)_j(q^2;q^2)_k}x^{i+2j+2k} \nonumber \\
&=\sum_{i,j,k\geq 0} \frac{q^{P(i+1,j,k)+j+2k}}{(q;q)_i(q^2;q^2)_j(q^2;q^2)_k}x^{i+1+2j+2k} +\sum_{i,j,k\geq 0} \frac{q^{P(i,j+1,k)+j+1+2k}q^i}{(q;q)_i(q^2;q^2)_j(q^2;q^2)_k}x^{i+2(j+1)+2k} \nonumber \\
&=\sum_{i,j,k\geq 0} \frac{q^{P(i,j,k)+2i+3j+2k+1}}{(q;q)_i(q^2;q^2)_j(q^2;q^2)_k}x^{i+1+2j+2k} +\sum_{i,j,k\geq 0} \frac{q^{P(i,j,k)+3i+5j+4k+3}}{(q;q)_i(q^2;q^2)_j(q^2;q^2)_k}x^{i+2j+2k+2} \nonumber \\
&=xqG(xq)+x^2q^3G(xq^2), \label{L2L4G}
\end{align}
where in the last second equality we used \eqref{Q-diff-1} and \eqref{Q-diff-2} and
\begin{align}
    G(x):=\sum_{i,j,k\geq 0}\frac{q^{P(i,j,k)+i+j}}{(q;q)_i(q^2;q^2)_j(q^2;q^2)_k}x^{i+2j+2k}. \label{exam4-G-defn}
\end{align}
From \eqref{L2-defn} and \eqref{exam4-G-defn} we have
\begin{align}
&L_2(x)-G(x)=\sum_{i,j,k\geq 0}\frac{q^{P(i,j,k)+j}\left((1-q^i)-(1-q^{2k})\right)}{(q;q)_i(q^2;q^2)_j(q^2;q^2)_k}x^{i+2j+2k} \nonumber \\
&=\sum_{i,j,k\geq 0}\frac{q^{P(i+1,j,k)+j}}{(q;q)_i(q^2;q^2)_j(q^2;q^2)_k}x^{i+1+2j+2k} -\sum_{i,j,k\geq 0}  \frac{q^{P(i,j,k+1)+j}}{(q;q)_i(q^2;q^2)_j(q^2;q^2)_k}x^{i+2j+2k+2}\nonumber \\
&=\sum_{i,j,k\geq 0}\frac{q^{P(i,j,k)+2i+3j+1}}{(q;q)_i(q^2;q^2)_j(q^2;q^2)_k}x^{i+2j+2k+1}-\sum_{i,j,k\geq 0}\frac{q^{P(i,j,k)+3j+4k+2}}{(q;q)_i(q^2;q^2)_j(q^2;q^2)_k}x^{i+2j+2k+2} \nonumber \\
&=xqT_1(x)-x^2q^2T_2(x), \label{L2GT1T2}
\end{align}
where in the last second equality we used \eqref{Q-diff-1} and \eqref{Q-diff-3} and
\begin{align}
&T_1(x):=\sum_{i,j,k\geq 0}\frac{q^{P(i,j,k)+2i+3j}}{(q;q)_i(q^2;q^2)_j(q^2;q^2)_k}x^{i+2j+2k}, \\
&T_2(x):=\sum_{i,j,k\geq 0}\frac{q^{P(i,j,k)+3j+4k}}{(q;q)_i(q^2;q^2)_j(q^2;q^2)_k}x^{i+2j+2k}.
\end{align}
We have
\begin{align}
&T_1(x)-G(xq)=\sum_{i,j,k\geq 0}\frac{q^{P(i,j,k)+2i+3j}(1-q^{2k})}{(q;q)_i(q^2;q^2)_j(q^2;q^2)_k}x^{i+2j+2k} \nonumber \\
&=\sum_{i,j,k\geq 0}\frac{q^{P(i,j,k+1)+2i+3j}}{(q;q)_i(q^2;q^2)_j(q^2;q^2)_k}x^{i+2j+2(k+1)} \nonumber \\
&=\sum_{i,j,k\geq 0}\frac{q^{P(i,j,k)+2i+5j+4k+2}}{(q;q)_i(q^2;q^2)_j(q^2;q^2)_k}x^{i+2j+2k+2} \quad \text{(by \eqref{Q-diff-3})} \nonumber \\
&=x^2q^2L_4(xq^2). \label{T1GL4}
\end{align}
Similarly,
\begin{align}
&T_2(x)-L_2(xq)=\sum_{i,j,k\geq 0}\frac{q^{P(i,j,k)+3j+4k}(1-q^{i})}{(q;q)_i(q^2;q^2)_j(q^2;q^2)_k}x^{i+2j+2k} \nonumber \\
&=\sum_{i,j,k\geq 0}\frac{q^{P(i+1,j,k)+3j+4k}}{(q;q)_i(q^2;q^2)_j(q^2;q^2)_k}x^{i+1+2j+2k} \nonumber \\
&=\sum_{i,j,k\geq 0}\frac{q^{P(i,j,k)+2i+5j+4k+1}}{(q;q)_i(q^2;q^2)_j(q^2;q^2)_k}x^{i+1+2j+2k} \quad \text{(by \eqref{Q-diff-1})}
 \nonumber \\
&=xqL_4(xq^2). \label{T2L2L4}
\end{align}
Comparing \eqref{T1GL4} with \eqref{T2L2L4}, we deduce that
\begin{align}
T_1(x)-G(xq)=xq(T_2(x)-L_2(xq)),
\end{align}
or equivalently,
\begin{align}
T_1(x)-xqT_2(x)=G(xq)-xqL_2(xq). \label{T1T2GL2}
\end{align}
Substituting \eqref{T1T2GL2} into \eqref{L2GT1T2}, we deduce that
\begin{align}
    G(x)+xqG(xq)=L_2(x)+x^2q^2L_2(xq). \label{GL2}
\end{align}
Substituting it with $x$ replaced by $xq$ into \eqref{L2L4G} we deduce that
\begin{align}
L_2(x)-L_4(xq)=xq\left(L_2(xq)+x^2q^4L_2(xq^2)\right).
\end{align}
This proves \eqref{add-L-equiv} and hence \eqref{L-relation-2} holds.
\end{proof}

\begin{lemma}\label{lem-LQ}
We have $Q_{4,i}(x)=L_i(x)$ for $i=1,2,3,4$.
\end{lemma}
\begin{proof}
From \eqref{L-relation-2} we have
\begin{align}
L_2(x)=L_1(x)+xqL_3(xq).
\end{align}
Substituting it into \eqref{L-relation-3} to eliminate $L_2(x)$ and $L_2(xq)$, we deduce that
\begin{align}
L_3(x)=L_1(x)+xqL_3(xq)+(xq)^2\left(L_1(xq)+xq^2L_3(xq^2)\right). \label{L3L1}
\end{align}
Substituting \eqref{L-relation-1} and \eqref{L-relation-4} into \eqref{L3L1} to eliminate $L_1$ and $L_3$, we deduce that
\begin{align}
&L_4(x)-(xq)^3L_4(xq^2)=L_4(xq)+xq\left(L_4(xq)-(xq^2)^3L_4(xq^3) \right) \nonumber \\
&\qquad \qquad +(xq)^2\left(L_4(xq^2)+xq^2\big(L_4(xq^2)-(xq^3)^3L_4(xq^4)  \big)  \right).
\end{align}
Upon simplification we find that $L_4(x)$ satisfies the $q$-difference equation
\begin{align}
F(x)=&(1+qx)F(qx)+q^2x^2(1+qx+q^2x)F(q^2x)  \nonumber \\
&-q^7x^4F(q^3x)-q^{13}x^6F(q^4x).
\end{align}
Similarly, from \eqref{Q-relation-1}--\eqref{Q-relation-4} we know that $Q_{4,4}(x)$ also satisfies this equation. Since the coefficients of $x^nq^m$ in $Q_{4,4}(x)$ and $L_4(x)$ match for $n,m\leq 0$, we conclude that $Q_{4,4}(x)=L_4(x)$. Now from the systems of equations \eqref{Q-relation-1}--\eqref{Q-relation-4} and \eqref{L-relation-1}--\eqref{L-relation-4}  we immediately deduce that $Q_{4,i}(x)=L_i(x)$ for $i=2,3,4$.
\end{proof}

\begin{proof}[Proof of Theorem \ref{thm-exam-4-relation}]
The desired identities are exactly the facts
$$L_4(x)=Q_{4,4}(x) \quad \text{and} \quad L_2(x)=Q_{4,2}(x),$$
which follow from Lemma \ref{lem-LQ}.
\end{proof}

\begin{theorem}\label{thm-exam-4}
We have
\begin{align}
\sum_{i,j,k\ge 0}\frac{q^{i^2+2j^2+2k^2+2ij+2jk+j}}{(q;q)_i(q^2;q^2)_j(q^2;q^2)_k}
=\frac{(q^4,q^5,q^9;q^9)_\infty}{(q;q)_\infty},
\\
\sum_{i,j,k\ge 0}\frac{q^{i^2+2j^2+2k^2+2ij+2jk+j+2k}}{(q;q)_i(q^2;q^2)_j(q^2;q^2)_k}
=\frac{(q^2,q^7,q^9;q^9)_\infty}{(q;q)_\infty},
\\
\sum_{i,j,k\ge 0}\frac{q^{i^2+2j^2+2k^2+2ij+2jk+i+3j+2k}}{(q;q)_i(q^2;q^2)_j(q^2;q^2)_k}
=\frac{(q,q^8,q^9;q^9)_\infty}{(q;q)_\infty}.
\end{align}
\end{theorem}
\begin{proof}
Setting $x=1,q$ in \eqref{conj-u-id-1} and setting $x=1$ in \eqref{conj-u-id-2}, then using the Andrews-Gordon identity \eqref{AG} with $k=4$, we obtain the desired identities.
\end{proof}

We also find the following identity to express the missing product $(q^3,q^6,q^9;q^9)_\infty$.
\begin{theorem}
We have
\begin{align}\label{exam4-new-id}
    \sum_{i,j,k\ge 0}\frac{q^{i^2+2j^2+2k^2+2ij+2jk+j+2k}(1+q^{i+2j+2k+2})}{(q;q)_i(q^2;q^2)_j(q^2;q^2)_k}
=\frac{(q^3,q^6,q^9;q^9)_\infty}{(q;q)_\infty}.
\end{align}
\end{theorem}
\begin{proof}
Note that the left side of \eqref{exam4-new-id} is exactly
$$L_2(1)+q^2L_2(q)=L_3(1)=Q_{4,3}(1).$$
Here we used \eqref{L-relation-3} and Lemma \ref{lem-LQ} for the equalities. Now using the Andrews--Gordon identity \eqref{AG} with $(k,i)=(4,3)$ we prove \eqref{exam4-new-id}.
\end{proof}

\subsection{Example 5}
This example corresponds to
\begin{align*}
&A=\begin{pmatrix}
2 & 2 & 2\\
2 & 4 & 4\\
1 & 2 & 3
\end{pmatrix}, \quad
AD=\begin{pmatrix}
4 & 4 & 2\\
4 & 8 & 4\\
2 & 4 & 3
\end{pmatrix},\\
&b\in  \bigg\{
\begin{pmatrix}
-1\\-2\\-3/2
\end{pmatrix},
\begin{pmatrix}
    -1\\0\\1/2
\end{pmatrix},
\begin{pmatrix}
    0\\-1\\-1
\end{pmatrix},
\begin{pmatrix}
    0\\0\\-1/2
\end{pmatrix},
\begin{pmatrix}
    0\\0\\0
\end{pmatrix},
\\ &\quad \quad \quad
\begin{pmatrix}
    0\\2\\3/2
\end{pmatrix},
\begin{pmatrix}
    1\\0\\-1/2
\end{pmatrix},
\begin{pmatrix}
    1\\0\\1/2
\end{pmatrix},
\begin{pmatrix}
    2\\2\\1/2
\end{pmatrix},
\begin{pmatrix}
    3\\4\\5/2
\end{pmatrix}
\bigg\}.
\end{align*}
As mentioned in Section \ref{sec-intro}, all of the Nahm sums $f_{A,b,c,d}(q)$ for quadruples $(A,b,c,d)$ in this example are modular functions of weight zero except for the cases with $b=(1,0,1/2)^\mathrm{T}$ and $(3,4,5/2)^\mathrm{T}$. We first establish the identities for the weight zero cases, for which Bressoud's identity \eqref{eq-Bressoud} and the following identities will be used:
\begin{align}
&\sum_{n=0}^\infty \frac{q^{2n^2}}{(q;q)_{2n}}=\frac{(-q^3,-q^5,q^8;q^8)_\infty}{(q^2;q^2)_\infty}, \quad \text{(\cite[Eq.\ (39)]{Slater1951})}   \label{S39}  \\
&\sum_{i,j\geq 0} \frac{u^{i+2j}q^{i^2+2ij+2j^2-i-j}}{(q;q)_i(q^2;q^2)_j}=(-u;q)_\infty. \quad \text{(\cite[Eq.\ (3.29)]{CW})}  \label{eq-Cao-Wang}
\end{align}

The constant term method will play an important role for this example. For any series $f(z)=\sum_{n\in \mathbb{Z}}a(n)z^n$, we define the constant term extractor $\mathrm{CT}$ as
\begin{align}
    \mathrm{CT} (f(z))=a(0).
\end{align}
\begin{theorem}
We have
\begin{align}
    &\sum_{i,j,k\ge 0}\frac{q^{\frac{3}{2}i^2+2j^2+4k^2+2ij+4ik+4jk-\frac{3}{2}i-j-2k}}{(q;q)_i(q^2;q^2)_j(q^2;q^2)_k}
    =(-1;q)_\infty,\label{122-1}
    \\
    &\sum_{i,j,k\ge 0}\frac{q^{\frac{3}{2}i^2+2j^2+4k^2+2ij+4ik+4jk+\frac{1}{2}i-j}}{(q;q)_i(q^2;q^2)_j(q^2;q^2)_k}
    =
    \frac{1}{(q;q^2)_\infty}, \label{122-2}
    \\
    &\sum_{i,j,k\ge 0}\frac{q^{3i^2+4j^2+8k^2+4ij+8ik+8jk-2i-2k}}{(q^2;q^2)_i(q^4;q^4)_j(q^4;q^4)_k}
    =(-q;q^2)_\infty,\label{122-3}
    \\
    &\sum_{i,j,k\ge 0}\frac{q^{\frac{3}{2}i^2+2j^2+4k^2+2ij+4ik+4jk-\frac{1}{2}i}}{(q;q)_i(q^2;q^2)_j(q^2;q^2)_k}
    =
    \frac{(q^3,q^3,q^6;q^6)_\infty}{(q;q)_\infty}, \label{122-4} \\
     &\sum_{i,j,k\ge 0}\frac{q^{3i^2+4j^2+8k^2+4ij+8ik+8jk}}{(q^2;q^2)_i(q^4;q^4)_j(q^4;q^4)_k}
   =
    \frac{(q^6,q^{10};q^{16})_\infty}{(q^3,q^4,q^5;q^8)_\infty},\label{122-5}
    \\
    &\sum_{i,j,k\ge 0}\frac{q^{\frac{3}{2}i^2+2j^2+4k^2+2ij+4ik+4jk+\frac{3}{2}i+2k}}{(q;q)_i(q^2;q^2)_j(q^2;q^2)_k}
    =
    \frac{1}{(q^2,q^3,q^4;q^6)_\infty}, \label{122-6} \\
&\sum_{i,j,k\ge 0}\frac{q^{\frac{3}{2}i^2+2j^2+4k^2+2ij+4ik+4jk-\frac{1}{2}i+j}}{(q;q)_i(q^2;q^2)_j(q^2;q^2)_k}=\frac{1}{(q;q^2)_\infty},\label{122-7}
\\
&\sum_{i,j,k\ge 0}\frac{q^{\frac{3}{2}i^2+2j^2+4k^2+2ij+4ik+4jk+\frac{1}{2}i+2j+2k}}{(q;q)_i(q^2;q^2)_j(q^2;q^2)_k}
    =
    \frac{1}{(q^2,q^3,q^4;q^6)_\infty}. \label{122-9}
\end{align}
\end{theorem}
\begin{proof}
Let
\begin{align}
    &F(u,v,w)=F(u,v,w;q):=\sum_{i,j,k\geq 0} \frac{q^{\frac{3}{2}i^2+2j^2+4k^2+2ij+4ik+4jk}u^iv^jw^k}{(q;q)_i(q^2;q^2)_j(q^2;q^2)_k}.
\end{align}
Using \eqref{Jacobi}, \eqref{euler-1} and \eqref{euler-2} we have
\begin{align}
    &F(u,v,w)=\CT \left[\sum_{i\geq 0}\frac{q^{\frac{1}{2}i^2}u^iz^i}{(q;q)_i} \sum_{j\geq 0} \frac{q^{j^2}v^jz^j}{(q^2;q^2)_j} \sum_{k\geq 0} \frac{z^{2k}w^k}{(q^2;q^2)_k} \sum_{n=-\infty}^\infty z^{-n}q^{n^2}  \right] \nonumber \\
    &=\CT \left[\frac{(-uzq^{1/2};q)_\infty (-qvz;q^2)_\infty (-qz,-q/z,q^2;q^2)_\infty}{(wz^2;q^2)_\infty}\right]. \label{122-F-start}
\end{align}

(1) From \eqref{122-F-start} we have
\begin{align}
    &F(q^{-3/2},q^{-1},q^{-2};q)=\CT \left[\frac{(-z;q^2)_\infty }{(q^{-1}z;q)_\infty} (-qz,-q/z,q^2;q^2)_\infty \right] \nonumber \\
&=\CT \left[\sum_{i=0}^\infty \frac{q^{-i}z^i}{(q;q)_i} \sum_{j=0}^\infty \frac{q^{j^2-j}z^j}{(q^2;q^2)_j} \sum_{k=-\infty}^\infty q^{k^2}z^{-k}  \right] \nonumber \\
&=\sum_{i,j\geq 0} \frac{q^{i^2+2ij+2j^2-i-j}}{(q;q)_i(q^2;q^2)_j}.
\end{align}
Now by \eqref{eq-Cao-Wang} with $u=1$ we obtain \eqref{122-1}.

(2) From \eqref{122-F-start} we have
\begin{align}
    &F(q^{1/2},q^{-1},1)=\CT \left[\frac{(-qz;q)_\infty (-z;q^2)_\infty (-qz,-q/z,q^2;q^2)_\infty}{(z^2;q^2)_\infty} \right] \nonumber \\
    &=\CT\left[\frac{(-q^2z;q^2)_\infty}{(z;q)_\infty} (-qz,-q/z,q^2;q^2)_\infty\right] \label{122-F-2-CT} \\
    &=\CT\left[\sum_{i\geq 0} \frac{z^i}{(q;q)_i} \sum_{j\geq 0} \frac{z^jq^{j^2+j}}{(q^2;q^2)_j} \sum_{k=-\infty}^\infty q^{k^2}z^{-k}   \right] \nonumber \\
    &=\sum_{i,j\geq 0} \frac{q^{i^2+2ij+2j^2+j}}{(q;q)_i(q^2;q^2)_j}. \nonumber
\end{align}
Now by Bressoud's identity \eqref{eq-Bressoud} with $(k,i)=(3,2)$ we obtain \eqref{122-2}.

(3) From \eqref{122-F-start} we have
\begin{align}
    &F(q^{-2},1,q^{-2};q^2)=\CT \left[\frac{(-q^2z;q^4)_\infty }{(q^{-1}z;q^2)_\infty} (-q^2z,-q^2/z,q^4;q^4)_\infty \right] \nonumber \\
&=\CT \left[\sum_{i=0}^\infty \frac{q^{-i}z^i}{(q^2;q^2)_i} \sum_{j=0}^\infty \frac{q^{2j^2}z^j}{(q^4;q^4)_j} \sum_{k=-\infty}^\infty q^{2k^2}z^{-k}  \right] \nonumber \\
&=\sum_{i,j\geq 0} \frac{q^{2i^2+4ij+4j^2-i}}{(q^2;q^2)_i(q^4;q^4)_j}.
\end{align}
Now by \eqref{eq-Cao-Wang} with $u=q^{1/2}$ we obtain \eqref{122-3}.

(4) From \eqref{122-F-start} we have
\begin{align}
&F(q^{-1/2},1,1;q^2)=\CT \left[\frac{(-qz;q^2)_\infty}{(z;q)_\infty} (-qz,-q/z,q^2;q^2)_\infty \right] \nonumber \\
&=\CT \left[\sum_{i\geq 0} \frac{z^i}{(q;q)_i} \sum_{j\geq 0} \frac{q^{j^2}z^j}{(q^2;q^2)_j} \sum_{k=-\infty}^\infty q^{k^2}z^{-k}\right] \nonumber \\
&=\sum_{i,j\geq 0} \frac{q^{i^2+2ij+2j^2}}{(q;q)_i(q^2;q^2)_j}.
\end{align}
Now by \eqref{eq-Bressoud} with $(k,i)=(3,3)$ we get \eqref{122-4}.

(5) From \eqref{122-F-start} we have
\begin{align}
&F(1,1,1;q^2)=\CT \left[\frac{(-qz;q^2)_\infty (-q^2z,-q^2/z,q^4;q^4)_\infty}{(z;q^2)_\infty (-z;q^4)_\infty} \right] \nonumber \\
&=\CT \left[\sum_{i\geq 0} \frac{(-q;q^2)_i}{(q^2;q^2)_i}z^i \sum_{j\geq 0} \frac{(-z)^j}{(q^4;q^4)_j} \sum_{k=-\infty}^\infty z^{-k}q^{2k^2}  \right] \nonumber \\
&=\sum_{i\geq 0} \frac{(-q;q^2)_i}{(q^2;q^2)_i} \sum_{j\geq 0} \frac{(-1)^jq^{2(i+j)^2}}{(q^4;q^4)_j} \nonumber \\
&=\sum_{i=0}^\infty \frac{(-q;q^2)_i}{(q^2;q^2)_i}q^{2i^2}(q^{4i+2};q^4)_\infty=(q^2;q^4)_\infty \sum_{i=0}^\infty \frac{q^{2i^2}}{(q;q)_{2i}}.
\end{align}
Now by \eqref{S39} we obtain \eqref{122-5}.

(6) From \eqref{122-F-start} we have
\begin{align}
&F(q^{3/2},1,q^2)=\CT \left[\frac{(-q^3z;q^2)_\infty }{(qz;q)_\infty} (-qz,-q/z,q^2;q^2)_\infty \right] \label{122-F-6-CT} \\
&=\CT \left[\sum_{i=0}^\infty \frac{q^iz^i}{(q;q)_i} \sum_{j=0}^\infty \frac{q^{j^2+2j}z^j}{(q^2;q^2)_j} \sum_{k=-\infty}^\infty q^{k^2}z^{-k}  \right] \nonumber \\
&=\sum_{i,j\geq 0} \frac{q^{i^2+2ij+2j^2+i+2j}}{(q;q)_i(q^2;q^2)_j}. \nonumber
\end{align}
Now by Bressoud's identity \eqref{eq-Bressoud} with $(k,i)=(3,1)$ we obtain \eqref{122-6}.

(7) From \eqref{122-F-start} and \eqref{122-F-2-CT} we have
\begin{align}
    &F(q^{-1/2},q,1;q)=\CT \left[\frac{(-q^2z;q^2)_\infty }{(z;q)_\infty} (-qz,-q/z,q^2;q^2)_\infty \right] =F(q^{1/2},q^{-1},1).
\end{align}
Hence by (2) we obtain \eqref{122-7}.

(8) From \eqref{122-F-start} and \eqref{122-F-6-CT} we have
\begin{align}
    &F(q^{1/2},q^2,q^2)=\CT \left[\frac{(-q^3z;q^2)_\infty}{(qz;q)_\infty} (-qz,-q/z,q^2;q^2)_\infty \right]=F(q^{3/2},1,q^2).
\end{align}
Hence by (6) we obtain \eqref{122-9}.
\end{proof}
Now we discuss the two special cases announced in Section 1.
\begin{proof}[Proof of Theorem \ref{thm-S}]
We denote the series in the left side of \eqref{S1-product} and \eqref{S2-product} by $S_1(q)$ and $S_2(q)$, respectively. We first prove that
\begin{align}\label{S1-add-S2}
    S_1(q)+qS_2(q)=\frac{J_2}{J_1}.
\end{align}
Subtracting $S_1(q)$ from the left hand-side of \eqref{122-7}, we have
\begin{align*}
    &\sum_{i,j,k\ge 0}\frac{q^{\frac{3}{2}i^2+2j^2+4k^2+2ij+4ik+4jk-\frac{1}{2}i+j}}{(q;q)_i(q^2;q^2)_j(q^2;q^2)_k}-\sum_{i,j,k\ge 0}\frac{q^{\frac{3}{2}i^2+2j^2+4k^2+2ij+4ik+4jk+\frac{1}{2}i+j}}{(q;q)_i(q^2;q^2)_j(q^2;q^2)_k}\\
    &=\sum_{i,j,k\ge 0}\frac{q^{\frac{3}{2}i^2+2j^2+4k^2+2ij+4ik+4jk-\frac{1}{2}i+j}(1-q^i)}{(q;q)_i(q^2;q^2)_j(q^2;q^2)_k} \\
    &=\sum_{i,j,k\ge 0}\frac{q^{\frac{3}{2}(i+1)^2+2j^2+4k^2+2(i+1)j+4(i+1)k+4jk-\frac{1}{2}(i+1)+j}}{(q;q)_i(q^2;q^2)_j(q^2;q^2)_k}=qS_2(q).
\end{align*}
Using \eqref{122-7} we obtain \eqref{S1-add-S2}.

From \eqref{122-F-start} we have
\begin{align}\label{S1-first}
&S_1(q)= F(q^{1/2},q,1)=\CT \left[\frac{(-qz;q)_\infty(-q^2z;q^2)_\infty (-qz,-q/z,q^2;q^2)_\infty}{(z^2;q^2)_\infty} \right] \nonumber \\
&=\CT \left[\frac{(-qz;q)_\infty (-q^2z,-qz,-q/z,q^2;q^2)_\infty}{(z;q)_\infty (-z,-qz;q^2)_\infty} \right] \nonumber \\
&=(q^2;q^2)_\infty\mathrm{CT} \left(\frac{(-qz;q)_\infty (-q/z;q^2)_\infty}{(1+z)(z;q)_\infty} \right) \nonumber \\
&=(q^2;q^2)_\infty \mathrm{CT} \left(\sum_{n=0}^\infty (-1)^nz^n  \sum_{k=0}^\infty \frac{(-q;q)_kz^k}{(q;q)_k} \sum_{j=0}^\infty \frac{q^{j^2}z^{-j}}{(q^2;q^2)_j}  \right) \quad \text{(by \eqref{eq-qbinomial} and \eqref{euler-2})}\nonumber \\
&=(q^2;q^2)_\infty \sum_{n=0}^\infty \sum_{k=0}^\infty (-1)^n \frac{(-q;q)_k}{(q;q)_k} \cdot \frac{q^{(n+k)^2}}{(q^2;q^2)_{n+k}} \nonumber \\
&=(q^2;q^2)_\infty \sum_{n=0}^\infty \frac{(-1)^nq^{n^2}}{(q^2;q^2)_n} \sum_{k=0}^n (-1)^k \frac{(-q;q)_k}{(q;q)_k}.
\end{align}
Applying Lemma \ref{lem-BP-sum} with the Bailey pair $(\alpha_n(1,q),\beta_n(1,q))$ in Lemma \ref{lem-new-BP}, we deduce that
\begin{align}\label{S1-mid}
&\sum_{n=0}^\infty \frac{(-1)^nq^{n^2}}{(q^2;q^2)_n} \sum_{k=0}^n (-1)^k \frac{(-q;q)_k}{(q;q)_k}=\sum_{n=0}^\infty q^{n^2}\beta_n(1,q)\nonumber \\
&=\frac{1}{(q;q)_\infty} \sum_{n=0}^\infty q^{n^2}\alpha_n(1,q) \nonumber \\
&=\frac{1}{(q;q)_\infty}\left(1+\sum_{n=1}^\infty \sum_{i=-n}^{n-1}(-1)^{i+n}q^{i^2+n^2}\right).
\end{align}
Recall Ramanujan's theta function (see \cite[Eq.\ (1.3.13)]{Berndt-book})
\begin{align}
    \varphi(q):=\sum_{n=-\infty}^\infty q^{n^2}=(-q;q^2)_\infty^2(q^2;q^2)_\infty
\end{align}
where the last equality follows from \eqref{Jacobi}. Replacing $q$ by $-q$, we have \cite[Eq.\ (2.2.12)]{Andrews}
\begin{align}\label{eq-phi-minus}
\varphi(-q)=\frac{J_1^2}{J_2}.
\end{align}
Note that
\begin{align}\label{eq-half-theta}
\sum_{n=1}^\infty (-1)^nq^{n^2}=\frac{1}{2}\Big( \sum_{n=-\infty}^\infty (-1)^nq^{n^2}-1\Big)=\frac{1}{2}\Big(\varphi(-q)-1 \Big).
\end{align}
We have
\begin{align}\label{S1-final}
&1+\sum_{n=1}^\infty \sum_{i=-n}^{n-1}(-1)^{i+n}q^{i^2+n^2}=1+\sum_{n=1}^\infty (-1)^nq^{n^2}-\sum_{n=1}^\infty q^{2n^2}+2\sum_{n=1}^\infty \sum_{i=1}^n (-1)^{i+n}q^{i^2+n^2} \nonumber \\
&=1+\sum_{n=1}^\infty (-1)^nq^{n^2}+\left(\sum_{n=1}^\infty (-1)^nq^{n^2}\right)^2 \nonumber \\
&=1+\frac{1}{2}\Big(\varphi(-q)-1\Big)+\frac{1}{4}\Big(\varphi(-q)-1\Big)^2 \nonumber \\
&=\frac{1}{4}\varphi^2(-q)+\frac{3}{4}=\frac{1}{4}\frac{J_1^4}{J_2^2}+\frac{3}{4}.
\end{align}
Substituting \eqref{S1-final} into \eqref{S1-mid} and then substituting the result into \eqref{S1-first}, we obtain \eqref{S1-product}.

Substituting \eqref{S1-product} into \eqref{S1-add-S2}, we obtain \eqref{S2-product}.
\end{proof}

\begin{corollary}
We have
\begin{align}
    &\sum_{i,j,k\ge 0}\frac{q^{\frac{3}{2}i^2+2j^2+4k^2+2ij+4ik+4jk+\frac{1}{2}i+j}}{(q;q)_i(q^2;q^2)_j(q^2;q^2)_k}=\frac{1}{(q^2;q^2)_\infty}\sum_{i=-\infty}^\infty (i+1)q^{2i^2+i}, \label{S1-theta}
    \\
    &\sum_{i,j,k\ge 0}\frac{q^{\frac{3}{2}i^2+2j^2+4k^2+2ij+4ik+4jk+\frac{5}{2}i+3j+4k}}{(q;q)_i(q^2;q^2)_j(q^2;q^2)_k}=\frac{1}{(q^2;q^2)_\infty}\sum_{i=-\infty}^\infty (i+1)q^{2i^2+3i}. \label{S2-theta}
\end{align}
\end{corollary}
\begin{proof}
Recall Ramanujan's theta function (see \cite[Eq.\ (2.2.13)]{Andrews} or \cite[Eq.\ (1.3.14)]{Berndt-book})
\begin{align}\label{theta-psi}
    \psi(q):=\sum_{n=0}^\infty q^{n(n+1)/2}=\frac{J_2^2}{J_1}
\end{align}
and Jacobi's identity \cite[Theorem 1.3.9]{Berndt-book}
\begin{align}\label{eq-Jacobi-id}
    J_1^3=\sum_{n=0}^\infty (-1)^n(2n+1)q^{n(n+1)/2}.
\end{align}
By \eqref{S1-product} and the above identities we have
\begin{align*}
    S_1(q)&=\frac{1}{4}\cdot \frac{1}{(q^2;q^2)_\infty} \Big(3\frac{J_2^2}{J_1}+J_1^3 \Big) \nonumber \\
    &=\frac{1}{4}\cdot \frac{1}{(q^2;q^2)_\infty} \Big(3\sum_{i=0}^\infty q^{i(i+1)/2}+\sum_{i=0}^\infty (-1)^i(2i+1)q^{i(i+1)/2}  \Big) \nonumber \\
     &=\frac{1}{4}\cdot \frac{1}{(q^2;q^2)_\infty} \Big(3\sum_{i=-\infty}^\infty q^{i(2i+1)}+\sum_{i=-\infty}^\infty (4i+1)q^{i(2i+1)}  \Big) \nonumber \\
    &=\frac{1}{(q^2;q^2)_\infty} \sum_{i=-\infty}^\infty (i+1)q^{2i^2+i}.
\end{align*}
Here for the third equality we used the fact that for any function $f$,
\begin{align}
    \sum_{i=0}^\infty f(i)=\sum_{i=0}^\infty f(2i)+\sum_{i=-\infty}^{-1} f(-2i-1)
\end{align}

Similarly, using \eqref{S2-product}, \eqref{theta-psi} and \eqref{eq-Jacobi-id} we obtain \eqref{S2-theta}.
\end{proof}

\subsection{Example 6}
This example corresponds to
\begin{align*}
A=\begin{pmatrix}
4 & 4 & 2\\
4 & 5 & 2\\
1 & 1 & 1
\end{pmatrix},\quad
AD=\begin{pmatrix}
8 & 8 & 2\\
8 & 10 & 2\\
2 & 2 & 1
\end{pmatrix},\quad
b\in  \bigg\{
\begin{pmatrix}
0 \\ 0 \\1/2
\end{pmatrix},
\begin{pmatrix}
4 \\ 4 \\1/2
\end{pmatrix}
\bigg\}.
\end{align*}
We will use the following result to reduce triple sums to some double sums.
\begin{lemma}\label{add-lem-sum}
For $n\geq 0$ we have
\begin{align}\label{eq-add-id-sum}
    \sum_{i+j=n} \frac{q^{i^2}}{(q^2;q^2)_i(q^2;q^2)_j}=\frac{(-q;q^2)_n}{(q^2;q^2)_n}.
\end{align}
\end{lemma}
\begin{proof}
By \eqref{eq-qbinomial}, \eqref{euler-1} and \eqref{euler-2} we have
\begin{align}
    \sum_{i=0}^\infty \frac{q^{i^2}z^i}{(q^2;q^2)_i} \sum_{j=0}^\infty \frac{z^j}{(q^2;q^2)_j}=\frac{(-zq;q^2)_\infty}{(z;q^2)_\infty}=\sum_{n=0}^\infty \frac{(-q;q^2)_n}{(q^2;q^2)_n}z^n.
\end{align}
Comparing the coefficients of $z^n$ on both sides, we obtain the desired identity.
\end{proof}

The modularity of this example follows from the following theorem.
\begin{theorem}
We have
\begin{align}
\sum_{i,j,k\ge 0}\frac{q^{\frac{1}{2}i^2+5j^2+4k^2+2ij+2ik+8jk+\frac{1}{2}i}}{(q;q)_i(q^2;q^2)_j(q^2;q^2)_k}
&=\frac{(-q;q)_\infty}{(q^4,q^{16};q^{20})_\infty},
\label{table3.6.1}
\\
\sum_{i,j,k\ge 0}\frac{q^{\frac{1}{2}i^2+5j^2+4k^2+2ij+2ik+8jk+\frac{1}{2}i+4j+4k}}{(q;q)_i(q^2;q^2)_j(q^2;q^2)_k}
&=\frac{(-q;q)_\infty}{(q^8,q^{12};q^{20})_\infty}.
\label{table3.6.2}
\end{align}
\end{theorem}
\begin{proof}
We have
\begin{align}
&F(u,v)=F(u,v;q):=\sum_{i,j,k\ge 0}\frac{q^{\frac{1}{2}i^2+5j^2+4k^2+2ij+2ik+8jk}u^iv^{j+k}}{(q;q)_i(q^2;q^2)_j(q^2;q^2)_k}\notag
\\
&=\sum_{i\ge 0}\frac{q^{\frac{1}{2}i^2}u^i}{(q;q)_i}
\sum_{m\ge 0}q^{4m^2+2im}v^m
\sum_{j+k=m}\frac{q^{j^2}}{(q^2;q^2)_j(q^2;q^2)_k}\notag
\\
&=\sum_{i,m\ge 0}\frac{q^{\frac{1}{2}i^2+2im+4m^2}u^iv^m}{(q;q)_i} \cdot
\frac{(-q;q^2)_m}{(q^2;q^2)_m} \quad \text{(by \eqref{eq-add-id-sum})}
\notag
\\
&=\sum_{m\ge 0}\frac{q^{4m^2}v^m(-q;q^2)_m}{(q^2;q^2)_m}
\sum_{i\ge 0}\frac{q^{\frac{1}{2}(i^2-i)}(uq^{2m+\frac{1}{2}})^i}{(q;q)_i}\notag
\\
&=\sum_{m\ge 0}\frac{q^{4m^2}v^m(-q;q^2)_m(-uq^{\frac{1}{2}+2m};q)_\infty}{(q^2;q^2)_m} \quad \text{(by \eqref{euler-2})}\notag
\\
&=(-uq^{\frac{1}{2}};q)_\infty
\sum_{m\ge 0}\frac{q^{4m^2}v^m(-q;q^2)_m}{(q^2;q^2)_m(-uq^{\frac{1}{2}};q)_{2m}}.
\label{F3.6}
\end{align}
In particular,  we have
\begin{align*}
&F(q^{\frac{1}{2}},v)=(-q;q)_\infty\sum_{m\ge 0}\frac{q^{4m^2}v^m(-q;q^2)_m}{(q^2;q^2)_m(-q;q)_{2m}}
=(-q;q)_\infty\sum_{m\ge 0}\frac{q^{4m^2}v^m}{(q^4;q^4)_m}.
\end{align*}
Setting $v=1$ and $q^4$, and then using \eqref{RR} with $q$ replaced by $q^4$ we obtain \eqref{table3.6.1} and \eqref{table3.6.2}, respectively.
\end{proof}

\subsection{Examples 10, 12, 13, 14 and 15}
For these examples, we will use the following identity to reduce triple sums to double sums: for $n\geq 0$ we have
\begin{align}
\sum_{i+j=n}\frac{q^{i^2+j^2-i}}{(q^2;q^2)_i(q^2;q^2)_j}
    &=\frac{q^{(n^2-n)/2}}{(q;q)_n}. \quad \text{(\cite[Eq.\ (2.6)]{WW-I})} \label{lem-13}
\end{align}

\subsubsection{Example 10}
This example corresponds to
\begin{align*}
&A=\begin{pmatrix}
3/2&1/2&1\\
1/2&3/2&1\\
1/2&1/2&1
\end{pmatrix},\quad
AD=\begin{pmatrix}
3&1&1\\
1&3&1\\
1&1&1
\end{pmatrix},\nonumber
\\
&b\in  \bigg\{
\begin{pmatrix}
-3/2\\-1/2\\1/2
\end{pmatrix},
\begin{pmatrix}
-1/2\\-3/2\\1/2
\end{pmatrix},
\begin{pmatrix}
-1/2\\1/2\\ 0
\end{pmatrix},
\begin{pmatrix}
-1/2\\1/2\\1/2
\end{pmatrix},
\begin{pmatrix}
1/2\\-1/2\\ 0
\end{pmatrix},\nonumber
\\&\quad \quad \quad
\begin{pmatrix}
1/2\\-1/2\\1/2
\end{pmatrix},
\begin{pmatrix}
1/2\\3/2\\1/2
\end{pmatrix},
\begin{pmatrix}
1/2\\3/2\\ 1
\end{pmatrix},
\begin{pmatrix}
3/2\\1/2\\1/2
\end{pmatrix},
\begin{pmatrix}
3/2\\1/2\\ 1
\end{pmatrix}
\bigg\}.
\end{align*}
We need the following identities to prove its modularity:
\begin{align}
    &\sum_{i,j\ge 0}\frac{q^{i^2+ij+\frac{1}{2}j^2-i+\frac{1}{2}j}}{(q;q)_i(q;q)_j}
    =
    2\frac{(q^4;q^4)_\infty}{(q;q)_\infty}, \quad \text{(\cite[Eq.\ (3.4)]{Wang-rank2})}\label{LW.rank2.3.5}
    \\
    &\sum_{i,j\ge 0}\frac{q^{2i^2+2ij+j^2}}{(q^2;q^2)_i(q^2;q^2)_j}
    =
    \frac{1}{(q,q^4,q^7;q^8)_\infty}, \quad \text{(\cite[Eq.\ (3.5)]{Wang-rank2})} \label{LW.rank2.3.6}
    \\
    &\sum_{i,j\ge 0}\frac{q^{i^2+ij+\frac{1}{2}j^2+\frac{1}{2}j}}{(q;q)_i(q;q)_j}
    =
    \frac{(q^2;q^2)^3_\infty}{(q;q)^2_\infty(q^4;q^4)_\infty}, \quad \text{(\cite[Eq.\ (3.6)]{Wang-rank2})} \label{LW.rank2.3.7}
    \\
    &\sum_{i,j\ge 0}\frac{q^{i^2+ij+\frac{1}{2}j^2+i+\frac{1}{2}j}}{(q;q)_i(q;q)_j}
    =
    \frac{(q^4;q^4)_\infty}{(q;q)_\infty},  \quad \text{(\cite[Eq.\ (3.7)]{Wang-rank2})} \label{LW.rank2.3.8}
    \\
    &\sum_{i,j\ge 0}\frac{q^{2i^2+2ij+j^2+2i+2j}}{(q^2;q^2)_i(q^2;q^2)_j}
    =
    \frac{1}{(q^3,q^4,q^5;q^8)_\infty}. \quad \text{(\cite[Eq.\ (3.8)]{Wang-rank2})} \label{LW.rank2.3.9}
\end{align}
\begin{theorem}
We have
\begin{align}
\sum_{i,j,k\ge 0}\frac{q^{i^2+3j^2+3k^2+2ij+2ik+2jk+i-j-3k}}{(q^2;q^2)_i(q^4;q^4)_j(q^4;q^4)_k}
&=2\frac{(q^8;q^8)_\infty}{(q^2;q^2)_\infty},
\label{table3.10.1}
\\
\sum_{i,j,k\ge 0}\frac{q^{i^2+3j^2+3k^2+2ij+2ik+2jk+j-k}}{(q^2;q^2)_i(q^4;q^4)_j(q^4;q^4)_k}
&=\frac{1}{(q,q^4,q^7;q^8)_\infty},
\label{table3.10.3}
\\
\sum_{i,j,k\ge 0}\frac{q^{i^2+3j^2+3k^2+2ij+2ik+2jk+i+j-k}}{(q^2;q^2)_i(q^4;q^4)_j(q^4;q^4)_k}
&=\frac{(q^4;q^4)^3_\infty}{(q^2;q^2)^2_\infty(q^8;q^8)_\infty},
\label{table3.10.4}
\\
\sum_{i,j,k\ge 0}\frac{q^{i^2+3j^2+3k^2+2ij+2ik+2jk+i+3j+k}}{(q^2;q^2)_i(q^4;q^4)_j(q^4;q^4)_k}
&=\frac{(q^8;q^8)_\infty}{(q^2;q^2)_\infty},
\label{table3.10.7}
\\
\sum_{i,j,k\ge 0}\frac{q^{i^2+3j^2+3k^2+2ij+2ik+2jk+2i+3j+k}}{(q^2;q^2)_i(q^4;q^4)_j(q^4;q^4)_k}
&=\frac{1}{(q^3,q^4,q^5;q^8)_\infty}.
\label{table3.10.8}
\end{align}
\end{theorem}
Interchanging $j$ with $k$, we obtain identities for the remaining choices of $b$.
\begin{proof}
We have
\begin{align}
&\sum_{i,j,k\ge 0}\frac{q^{i^2+3j^2+3k^2+2ij+2ik+2jk-2k}u^{i}v^{j+k}}{(q^2;q^2)_i(q^4;q^4)_j(q^4;q^4)_k}\notag
\\
&=\sum_{i\ge 0}\frac{q^{i^2}u^i}{(q^2;q^2)_i}
\sum_{m\ge 0}q^{m^2+2im}v^m
\sum_{j+k=m}\frac{q^{2j^2+2k^2-2k}}{(q^4;q^4)_j(q^4;q^4)_k}\notag
\\
&=\sum_{i,m\ge 0}\frac{q^{i^2+2im+2m^2-m}u^iv^m}{(q^2;q^2)_i(q^2;q^2)_m}.  \quad \text{(by \eqref{lem-13})}
\label{F3.10}
\end{align}
Setting $(u,v)= (q,q^{-1})$, $(1,q)$, $(q,q)$, $(q,q^3)$ and  $(q^{2},q^3)$ in \eqref{F3.10}, and then using \eqref{LW.rank2.3.5}--\eqref{LW.rank2.3.9}, we obtain \eqref{table3.10.1}--\eqref{table3.10.8}, respectively.
\end{proof}

\subsubsection{Example 12}
This example corresponds to
\begin{align*}
&A=\begin{pmatrix}
3/2&1/2& 1\\
1/2&3/2& 1\\
1/2&1/2&3/2
\end{pmatrix}, \quad
AD=\begin{pmatrix}
3&1& 1\\
1&3& 1\\
1&1&3/2
\end{pmatrix}, \nonumber\\
& b\in  \bigg\{
\begin{pmatrix}
-1/2\\1/2\\-1/2
\end{pmatrix},
\begin{pmatrix}
-1/2\\1/2\\ 0
\end{pmatrix},
\begin{pmatrix}
    1/2\\-1/2\\-1/2
\end{pmatrix},
\begin{pmatrix}
    1/2\\-1/2\\0
\end{pmatrix},
\begin{pmatrix}
1/2\\3/2\\1/2
\end{pmatrix},
\begin{pmatrix}
    3/2\\1/2\\1/2
\end{pmatrix}
\bigg\}.
\end{align*}
To prove its modularity, we need the following identities:
\begin{align}
&\sum_{i,j\ge 0}\frac{q^{3i^2+4ij+4j^2-2i}}{(q^4;q^4)_i(q^4;q^4)_j}
    =
    \frac{J_{14}J_{28}^2J_{2,28}}{J_{1,28}J_{4,28}J_{8,28}J_{13,28}}, \quad \text{(\cite[Eq.\ (3.74)]{Wang-rank2})}
    \label{LW.rank2.3.70}
    \\
    &\sum_{i,j\ge 0}\frac{q^{3i^2+4ij+4j^2}}{(q^4;q^4)_i(q^4;q^4)_j}
    =
    \frac{J_{14}J_{28}^2J_{6,28}}{J_{3,28}J_{4,28}J_{11,28}J_{12,28}}, \quad \text{(\cite[Eq.\ (3.75)]{Wang-rank2})}
    \label{LW.rank2.3.71}
    \\
    &\sum_{i,j\ge 0}\frac{q^{3i^2+4ij+4j^2+2i+4j}}{(q^4;q^4)_i(q^4;q^4)_j}
    =
    \frac{J_{14}J_{28}^2J_{10,28}}{J_{5,28}J_{8,28}J_{9,28}J_{12,28}}. \quad \text{(\cite[Eq.\ (3.76)]{Wang-rank2})}
    \label{LW.rank2.3.72}
\end{align}
\begin{theorem}
We have
\begin{align}
\sum_{i,j,k\ge 0}
\frac{q^{3i^2+6j^2+6k^2+4ij+4ik+4jk-2i-2j+2k}}{(q^4;q^4)_i(q^8;q^8)_j(q^8;q^8)_k}
&=\frac{J_{14}J_{28}^2J_{2,28}}{J_{1,28}J_{4,28}J_{8,28}J_{13,28}},
\label{table3.12.1}
\\
\sum_{i,j,k\ge 0}
\frac{q^{3i^2+6j^2+6k^2+4ij+4ik+4jk-2j+2k}}{(q^4;q^4)_i(q^8;q^8)_j(q^8;q^8)_k}
&=\frac{J_{14}J_{28}^2J_{6,28}}{J_{3,28}J_{4,28}J_{11,28}J_{12,28}},
\label{table3.12.2}
\\
\sum_{i,j,k\ge 0}
\frac{q^{3i^2+6j^2+6k^2+4ij+4ik+4jk+2i+2j+6k}}{(q^4;q^4)_i(q^8;q^8)_j(q^8;q^8)_k}
&=\frac{J_{14}J_{28}^2J_{10,28}}{J_{5,28}J_{8,28}J_{9,28}J_{12,28}}.
\label{table3.12.5}
\end{align}
\end{theorem}
Interchanging $j$ with $k$, we obtain identities for the remaining choices of $b$.
\begin{proof}
We have
\begin{align}
&\sum_{i,j,k\ge 0}\frac{q^{3i^2+6j^2+6k^2+4ij+4ik+4jk-4j}u^{i}v^{j+k}}{(q^4;q^4)_i(q^8;q^8)_j(q^8;q^8)_k}\notag
\\
&=\sum_{i\ge 0}\frac{q^{3i^2}u^i}{(q^4;q^4)_i}
\sum_{m\ge 0}q^{2m^2+4im}v^m
\sum_{j+k=m}\frac{q^{4j^2-4j+4k^2}}{(q^8;q^8)_j(q^8;q^8)_k}\notag
\\
&=\sum_{i,m\ge 0}\frac{q^{3i^2+4im+4m^2-2m}u^iv^m}{(q^4;q^4)_i(q^4;q^4)_m}. \quad \text{(by \eqref{lem-13})}
\label{F3.12}
\end{align}
Setting $(u,v)=(q^{-2},q^{2})$, $(1,q^2)$ and $(q^2,q^{6})$ in \eqref{F3.12}, and then using \eqref{LW.rank2.3.70}--\eqref{LW.rank2.3.72}, we obtain \eqref{table3.12.1}--\eqref{table3.12.5}, respectively.
\end{proof}

\subsubsection{Example 13}
This example corresponds to
\begin{align*}
&A=\begin{pmatrix}
3/2&1/2& 2\\
1/2&3/2& 2\\
1 & 1 & 4
\end{pmatrix}, \quad
AD=\begin{pmatrix}
3 & 1 & 2\\
1 & 3 & 2\\
2 & 2 & 4
\end{pmatrix},\nonumber
\\
&b\in  \bigg\{
\begin{pmatrix}
-1/2\\1/2\\ 0
\end{pmatrix},
\begin{pmatrix}
-1/2\\1/2\\ 1
\end{pmatrix},
\begin{pmatrix}
1/2\\-1/2\\0
\end{pmatrix},
\begin{pmatrix}
1/2\\-1/2\\1
\end{pmatrix},
\begin{pmatrix}
1/2\\3/2\\ 2
\end{pmatrix},
\begin{pmatrix}
3/2\\1/2\\2
\end{pmatrix}
\bigg\}.
\end{align*}
Its modularity follows from the following theorem.
\begin{theorem}
We have
\begin{align}
\sum_{i,j,k\ge 0}\frac{q^{2i^2+\frac{3}{2}j^2+\frac{3}{2}k^2+2ij+2ik+jk-\frac{1}{2}j+\frac{1}{2}k}}{(q;q)_i(q^2;q^2)_j(q^2;q^2)_k}
&=\frac{(q^3,q^4,q^7;q^7)_\infty}{(q;q)_\infty},
\label{table3.13.1}
\\
\sum_{i,j,k\ge 0}\frac{q^{2i^2+\frac{3}{2}j^2+\frac{3}{2}k^2+2ij+2ik+jk+i-\frac{1}{2}j+\frac{1}{2}k}}{(q;q)_i(q^2;q^2)_j(q^2;q^2)_k}
&=\frac{(q^2,q^5,q^7;q^7)_\infty}{(q;q)_\infty},
\label{table3.13.2}
\\
\sum_{i,j,k\ge 0}\frac{q^{2i^2+\frac{3}{2}j^2+\frac{3}{2}k^2+2ij+2ik+jk+2i+\frac{1}{2}j+\frac{3}{2}k}}{(q;q)_i(q^2;q^2)_j(q^2;q^2)_k}
&=\frac{(q,q^6,q^7;q^7)_\infty}{(q;q)_\infty}.
\label{table3.13.5}
\end{align}
\end{theorem}
Interchanging $j$ with $k$, we obtain  identities for the  remaining choices of $b$.
\begin{proof}
We have
\begin{align}
&\sum_{i,j,k\ge 0}\frac{q^{2i^2+\frac{3}{2}j^2+\frac{3}{2}k^2+2ij+2ik+jk-j}u^{i}v^{j+k}}{(q;q)_i(q^2;q^2)_j(q^2;q^2)_k}
\notag
\\
&=\sum_{i\ge 0}\frac{q^{2i^2}u^i}{(q;q)_i}
\sum_{m\ge 0}q^{\frac{1}{2}m^2+2im}v^m
\sum_{j+k=m}\frac{q^{j^2-j+k^2}}{(q^2;q^2)_j(q^2;q^2)_k}\notag
\\
&=\sum_{i,m\ge 0}\frac{q^{2i^2+2im+m^2-\frac{1}{2}m}u^iv^m}{(q;q)_i(q;q)_m}.  \quad \text{(by \eqref{lem-13})}
\label{F3.13}
\end{align}
Setting $(u,v)=(1,q^{\frac{1}{2}})$, $(q,q^{\frac{1}{2}})$ and $(q^2,q^{\frac{3}{2}})$ in \eqref{F3.13}, and then using \eqref{AG} with $(k,i)=(3,3),(3,2),(3,1)$, we obtain \eqref{table3.13.1}--\eqref{table3.13.5}, respectively.
\end{proof}

\subsubsection{Example 14}
This example corresponds to
\begin{align*}
&A=\begin{pmatrix}
5/2&3/2& 1\\
3/2&5/2& 1\\
1/2&1/2& 1
\end{pmatrix}, \quad
AD=\begin{pmatrix}
5&3& 1\\
3&5& 1\\
1&1& 1
\end{pmatrix},\nonumber
\\
&b\in  \bigg\{
\begin{pmatrix}
-1/2\\1/2\\1/2
\end{pmatrix},
\begin{pmatrix}
1/2\\-1/2\\1/2
\end{pmatrix},
\begin{pmatrix}
3/2\\5/2\\1/2
\end{pmatrix},
\begin{pmatrix}
5/2\\3/2\\1/2
\end{pmatrix}
\bigg\}.
\end{align*}
To prove its modularity, we establish the following theorem.
\begin{theorem}
We have
\begin{align}
\sum_{i,j,k\ge 0}\frac{q^{\frac{1}{2}i^2+\frac{5}{2}j^2+\frac{5}{2}k^2+ij+ik+3jk+\frac{1}{2}i-\frac{1}{2}j+\frac{1}{2}k}}{(q;q)_i(q^2;q^2)_j(q^2;q^2)_k}
&=\frac{(q^4,q^6,q^{10};q^{10})_\infty}{(q;q)_\infty},
\label{table3.14.1}
\\
\sum_{i,j,k\ge 0}\frac{q^{\frac{1}{2}i^2+\frac{5}{2}j^2+\frac{5}{2}k^2+ij+ik+3jk+\frac{1}{2}i+\frac{3}{2}j+\frac{5}{2}k}}{(q;q)_i(q^2;q^2)_j(q^2;q^2)_k}
&=\frac{(q^2,q^8,q^{10};q^{10})_\infty}{(q;q)_\infty}.
\label{table3.14.2}
\end{align}
\end{theorem}
Interchanging $j$ with $k$, we obtain identities for the  remaining choices of $b$.

We will connect them with the following identities:
\begin{align}
     &\sum_{i,j \ge 0}\frac{q^{2i^2+ij+\frac{1}{2}j^2+\frac{1}{2}j}}{(q;q)_i(q;q)_j}
    =
    \frac{(q^4,q^6,q^{10};q^{10})_\infty}{(q;q)_\infty}, \quad \text{(\cite[Eq.\ (3.29)]{Wang-rank2})} \label{LW.rank2.3.28}
    \\
    &\sum_{i,j \ge 0}\frac{q^{2i^2+ij+\frac{1}{2}j^2+2i+\frac{1}{2}j}}{(q;q)_i(q;q)_j}
    =
    \frac{(q^2,q^8,q^{10};q^{10})_\infty}{(q;q)_\infty}. \quad \text{(\cite[Eq.\ (3.30)]{Wang-rank2})} \label{LW.rank2.3.29}
\end{align}
\begin{proof}
We have
\begin{align}
&\sum_{i,j,k\ge 0}\frac{q^{\frac{1}{2}i^2+\frac{5}{2}j^2+\frac{5}{2}k^2+ij+ik+3jk-j}u^{i}v^{j+k}}{(q;q)_i(q^2;q^2)_j(q^2;q^2)_k}\notag
\\
&=\sum_{j\ge 0}\frac{q^{\frac{1}{2}i^2}u^i}{(q;q)_i}
\sum_{m\ge 0}q^{\frac{3}{2}m^2+im}v^m
\sum_{j+k=m}\frac{q^{j^2-j+k^2}}{(q^2;q^2)_j(q^2;q^2)_k}\notag
\\
&
=\sum_{i,m\ge 0}\frac{q^{\frac{1}{2}i^2+im+2m^2-\frac{1}{2}m}u^iv^m}{(q;q)_i(q;q)_m}. \quad \text{(by \eqref{lem-13})}
\label{F3.14}
\end{align}
Setting $(u,v)=(q^{\frac{1}{2}},q^{\frac{1}{2}})$ and $(q^{\frac{1}{2}},q^{\frac{5}{2}})$ in \eqref{F3.14}, and then using \eqref{LW.rank2.3.28} and \eqref{LW.rank2.3.29}, we obtain \eqref{table3.14.1} and \eqref{table3.14.2}, respectively.
\end{proof}

\subsubsection{Example 15}
This example corresponds to
\begin{align*}
&A=\begin{pmatrix}
5/2&3/2& 2\\
3/2&5/2& 2\\
1 & 1 & 2
\end{pmatrix},\quad
AD=\begin{pmatrix}
5&3& 2\\
3&5& 2\\
2&2& 2
\end{pmatrix},\nonumber
\\
&b\in  \bigg\{
\begin{pmatrix}
-1/2\\1/2\\ 0
\end{pmatrix},
\begin{pmatrix}
1/2\\-1/2\\0
\end{pmatrix},
\begin{pmatrix}
1/2\\3/2\\ 0
\end{pmatrix},
\begin{pmatrix}
3/2\\1/2\\0
\end{pmatrix},
\begin{pmatrix}
3/2\\5/2\\1
\end{pmatrix},
\begin{pmatrix}
5/2\\3/2\\1
\end{pmatrix}
\bigg\}.
\end{align*}
\begin{theorem}
We have
\begin{align}
\sum_{i,j,k\ge 0}\frac{q^{i^2+\frac{5}{2}j^2+\frac{5}{2}k^2+2ij+2ik+3jk-\frac{1}{2}j+\frac{1}{2}k}}{(q;q)_i(q^2;q^2)_j(q^2;q^2)_k}
&=\frac{(q^3,q^4,q^7;q^7)_\infty}{(q;q)_\infty},
\label{table3.15.1}
\\
\sum_{i,j,k\ge 0}\frac{q^{i^2+\frac{5}{2}j^2+\frac{5}{2}k^2+2ij+2ik+3jk+\frac{1}{2}j+\frac{3}{2}k}}{(q;q)_i(q^2;q^2)_j(q^2;q^2)_k}
&=\frac{(q^2,q^5,q^7;q^7)_\infty}{(q;q)_\infty},
\label{table3.15.3}
\\
\sum_{i,j,k\ge 0}\frac{q^{i^2+\frac{5}{2}j^2+\frac{5}{2}k^2+2ij+2ik+3jk+i+\frac{3}{2}j+\frac{5}{2}k}}{(q;q)_i(q^2;q^2)_j(q^2;q^2)_k}
&=\frac{(q,q^6,q^7;q^7)_\infty}{(q;q)_\infty}.
\label{table3.15.5}
\end{align}
\end{theorem}
Interchanging $j$ with $k$, we obtain identities for the  remaining choices of $b$.
\begin{proof}
We have
\begin{align}
&\sum_{i,j,k\ge 0}\frac{q^{i^2+\frac{5}{2}j^2+\frac{5}{2}k^2+2ij+2ik+3jk-j}u^{i}v^{j+k}}{(q;q)_i(q^2;q^2)_j(q^2;q^2)_k}
\notag
\\
&=\sum_{i\ge 0}\frac{q^{i^2}u^i}{(q;q)_i}
\sum_{m\ge 0}q^{\frac{3}{2}m^2+2im}v^m
\sum_{j+k=m}\frac{q^{j^2-j+k^2}}{(q^2;q^2)_j(q^2;q^2)_k}\notag
\\
&=\sum_{i,m\ge 0}\frac{q^{i^2+2im+2m^2-\frac{1}{2}m}u^iv^m}{(q;q)_i(q;q)_m}. \quad \text{(by \eqref{lem-13})}
\label{F3.15}
\end{align}
Setting $(u,v)=(1,q^{\frac{1}{2}})$, $(1,q^{\frac{3}{2}})$ and $(q,q^{\frac{5}{2}})$ in \eqref{F3.15}, and then using  \eqref{AG} with $(k,i)=(3,3),(3,2),(3,1)$, we obtain \eqref{table3.15.1}--\eqref{table3.15.5}, respectively.
\end{proof}

We end this paper with the following remark.
Suppose $(A,b,c,d)$ is a modular quadruple, Mizuno \cite[Conjecture 4.1]{Mizuno} conjectured that $(A^\star,b^\star,c^\star,d^\star)$ is also modular quadruple where
\begin{align}
    A^\star=A^{-1}, \quad b^\star=A^{-1}b, \quad c^\star=\frac{1}{2}b^\mathrm{T} (AD)^{-1}b-\frac{\mathrm{tr} D}{24}-c, \quad d^\star=d.
\end{align}
This is motivated and is consistent with the duality observation of Zagier \cite[p.\ 50]{Zagier} for modular triples. We shall call $(A^\star,b^\star,c^\star,d^\star)$ the dual quadruple of $(A,b,c,d)$.

Note that the dual examples of Mizuno's rank two examples are contained in his list \cite[Table 2]{Mizuno}. In contrast, the dual examples of Mizuno's rank three modular quadruples \cite[Tables 2 and 3]{Mizuno} are not included there. We have justified a number of these dual examples and will discuss them in forthcoming papers.

\subsection*{Acknowledgements}
This work was supported by the National Key R\&D Program of China (Grant No.\ 2024YFA1014500).

\end{document}